\documentclass[11pt,amsfonts]{amsart}

\usepackage{amsmath,amssymb,amsthm,enumerate}
\usepackage[all]{xy}
\usepackage{graphicx}
\usepackage{xcolor}
\usepackage{comment}
\usepackage[T1]{fontenc}

\SelectTips{eu}{12}

\newtheorem{thm}{Theorem}[section]
\newtheorem{prop}[thm]{Proposition}
\newtheorem{lem}[thm]{Lemma}
\newtheorem{cor}[thm]{Corollary}

\newtheorem{problem}[thm]{Problem}

\theoremstyle{definition}
\newtheorem{definition}[thm]{Definition}
\newtheorem{example}[thm]{Example}

\theoremstyle{remark}
\newtheorem{remark}[thm]{Remark}

\newtheorem*{notation}{Notation}

\numberwithin{equation}{section}

\newcommand{\RR}{\mathbb{R}}
\newcommand{\ZZ}{\mathbb{Z}}
\newcommand{\NN}{\mathbb{N}}

\DeclareMathOperator{\cl}{\mathrm{cl}}
\DeclareMathOperator{\scl}{\mathrm{scl}}

\newcommand{\hg}{g}
\newcommand{\bG}{N}
\newcommand{\hG}{G}
\newcommand{\cM}{V}

\newcommand{\genus}{l}

\newcommand{\II}{\mathcal{I}}

\newcommand{\HH}{\mathcal{H}}

\newcommand{\Aut}{\mathrm{Aut}}
\newcommand{\IA}{\mathrm{IA}}
\newcommand{\Ker}{\mathrm{Ker}}
\renewcommand{\Im}{\mathrm{Im}}
\newcommand{\Symp}{\mathrm{Symp}}
\newcommand{\Ham}{\mathrm{Ham}}
\newcommand{\Homeo}{\mathrm{Homeo}}

\newcommand{\flux}{\mathrm{Flux}}
\newcommand{\tflux}{\widetilde{\mathrm{Flux}}}
\newcommand{\cal}{\mathrm{Cal}}
\newcommand{\diff}{\mathrm{Diff}}
\newcommand{\tdiff}{\widetilde{\mathrm{Diff}}}
\newcommand{\IAA}{\mathrm{IA}}
\newcommand{\QQQ}{\mathrm{Q}}
\newcommand{\hQQQ}{\overline{\mathrm{Q}}}
\newcommand{\rQQQ}{\widehat{\mathrm{Q}}}
\newcommand{\HHH}{\mathrm{H}}
\newcommand{\EH}{\mathrm{EH}}
\newcommand{\GL}{\mathrm{GL}}
\newcommand{\SL}{\mathrm{SL}}
\newcommand{\Sp}{\mathrm{Sp}}
\newcommand{\Mod}{\mathrm{Mod}}
\newcommand{\ppi}{\Gamma}
\newcommand{\tae}{\xi}
\newcommand{\fis}{\Lambda}
\newcommand{\qm}{f}

\newcommand{\Hom}{\mathrm{Hom}}

\makeatletter
\newsavebox{\@brx}
\newcommand{\llangle}[1][]{\savebox{\@brx}{\(\m@th{#1\langle}\)}%
  \mathopen{\copy\@brx\kern-0.5\wd\@brx\usebox{\@brx}}}
\newcommand{\rrangle}[1][]{\savebox{\@brx}{\(\m@th{#1\rangle}\)}%
  \mathclose{\copy\@brx\kern-0.5\wd\@brx\usebox{\@brx}}}
\makeatother

\allowdisplaybreaks[1]

\makeatletter
\@namedef{subjclassname@2020}{%
  \textup{2020} Mathematics Subject Classification}
\makeatother


\begin{document}

\title{The space of non-extendable quasimorphisms}

\author[M. Kawasaki]{Morimichi Kawasaki}
\address[Morimichi Kawasaki]{Department of Mathematical Sciences, Aoyama Gakuin University, 5-10-1 Fuchinobe, Chuo-ku, Sagamihara-shi, Kanagawa, 252-5258, Japan}
\email{michikawa318@gmail.com}

\author[M. Kimura]{Mitsuaki Kimura}
\address[Mitsuaki Kimura]{Department of Mathematics, Kyoto University, Kitashirakawa Oiwake-cho, Sakyo-ku, Kyoto 606-8502, Japan}
\email{mkimura@math.kyoto-u.ac.jp}

\author[S. Maruyama]{Shuhei Maruyama}
\address[Shuhei Maruyama]{Department of Mathematics, Chuo University, 1-13-27 Kasuga, Bunkyo-ku, Tokyo, 112-8551, Japan}
\email{m17037h@math.nagoya-u.ac.jp}

\author[T. Matsushita]{Takahiro Matsushita}
\address[Takahiro Matsushita]{Department of Mathematical Sciences, University of the Ryukyus, Nishihara-cho, Okinawa 903-0213, Japan}
\email{mtst@sci.u-ryukyu.ac.jp}

\author[M. Mimura]{Masato Mimura}
\address[Masato Mimura]{Mathematical Institute, Tohoku University, 6-3, Aramaki Aza-Aoba, Aoba-ku, Sendai 9808578, Japan}
\email{m.masato.mimura.m@tohoku.ac.jp}

\begin{abstract}
For a pair $(G,N)$ of a group $G$ and its normal subgroup $N$, we consider the space of quasimorphisms and quasi-cocycles on $N$ non-extendable to $G$.
  To treat this space, we establish the five-term exact sequence of cohomology relative to the bounded subcomplex.
  As its application, we study the spaces associated with  the kernel of the (volume) flux homomorphism, the IA-automorphism group of a free group, and certain normal subgroups of  Gromov-hyperbolic  groups.

  Furthermore, we employ this space to prove that the stable commutator length is equivalent to the stable mixed commutator length for certain pairs of a group and its normal subgroup.
\end{abstract}

\maketitle

\tableofcontents








\section{Introduction}
\subsection{Motivations}\label{subsec:motive}
A {\it quasimorphism} on a group $G$ is a real-valued function $f \colon G \to \RR$ on $G$ satisfying
\[D(f) := \sup \{ |f(xy) - f(x) - f(y)| \; | \; x,y \in G\} < \infty.\]
We call $D(f)$ the {\it defect} of the quasimorphism $f$. A quasimorphism $f$ on $G$ is said to be \textit{homogeneous} if $f(x^n) = n \cdot f(x)$ for every $x \in G$ and for every integer $n$. Let $\QQQ(G)$ denote the real vector space consisting of homogeneous quasimorphisms on $G$. The (homogeneous) quasimorphisms are closely related to the second bounded cohomology group $\HHH^2_b(G)=\HHH^2_b(G;\RR)$, and have been extensively studied in geometric group theory and symplectic geometry (see \cite{Ca}, \cite{Fr}, and \cite{PR}).

In this paper, we treat a pair $(G,N)$ of a group $G$ and its normal subgroup $N$. Let $i\colon N\to G$ be the inclusion map. In this setting, we can construct the following two real vector spaces:
\begin{itemize}
  \item the space $\QQQ(N)^G$ of all \emph{$G$-invariant homogeneous quasimorphisms on $N$}, where $f\colon N\to \RR$ is said to be \emph{$G$-invariant} if $f(gxg^{-1})=f(x)$ for every $g\in G$ and every $x\in N$,
  \item the space $\HHH^1(N)^G+ i^{\ast}\QQQ(G)$, where $\HHH^1(N)^G$ is the space of all $G$-invariant homomorphisms from $N$ to $\RR$, and $i^{\ast}$ is the linear map from $\QQQ(G)$ to $\QQQ(N)$ induced by $i \colon N\hookrightarrow G$.
\end{itemize}
An element $f\in \QQQ(N)$ belongs to $i^{\ast}\QQQ(G)$ if and only if there exists $\hat{f}\in \QQQ(G)$ such that $\hat{f}|_{N}\equiv f$; in this case, we say that $f$ is \emph{extendable} to $G$. Since a homogeneous quasimorphism is conjugation invariant (see Lemma \ref{lem homogeneous}), the space $i^\ast \QQQ(G)$ is contained in $\QQQ(N)^G$.
The \emph{extendability problem} asks whether there exists $f\in \QQQ(N)^G$ that is not extendable to $G$, equivalently, whether the quotient vector space
\[
\QQQ(N)^G/i^{\ast}\QQQ(G)
\]
is non-zero.
A stronger version of this problem asks whether the quotient space
\[
\QQQ(N)^G/(\HHH^1(N)^G+i^{\ast}\QQQ(G))
\]
is non-zero.
 We have some reasons to take the quotient vector space over $\HHH^1(N)^G+i^{\ast}\QQQ(G)$, instead of one over $i^{\ast}\QQQ(G)$.
Elements in $\HHH^1(N)^G$ seem `trivial' as quasimorphisms in $\QQQ(N)^G$; also, when we apply the  Bavard duality theorem for stable mixed commutator lengths (see Theorem~\ref{thm Bavard}), exactly elements in $\HHH^1(N)^G$ behave trivially.
An example of a pair  $(G,N)$  such that $\QQQ(N)^G / i^* \QQQ(G)$ is non-zero is provided by Shtern \cite{Sh}, and an example of a pair such that $\QQQ(N)^G / (\HHH^1(N)^G + i^* \QQQ(G))$ is non-zero is provided by the first and second authors \cite{KK}.
 Some of the authors generalized the result of \cite{KK} and provide its extrinsic application in \cite{KKMM2} (see Theorem \ref{thm flux}).

In the present paper, we reveal that under a certain condition on $\Gamma=G/N$  and a mild condition on $G$, the quotient real vector space $\QQQ(N)^G/(\HHH^1(N)^G+i^{\ast}\QQQ(G))$ is \emph{finite dimensional}. For example, commutativity (or amenability) of $\Gamma$ and finite presentability of  $G$ suffice.
We exhibit here two such examples: one corresponds to a surface group  (Theorem~\ref{thm surface group}), and the other corresponds to  the fundamental group of a hyperbolic mapping torus  (Theorem~\ref{mapping torus thm}). The main novel point of these theorems is that we obtain \emph{non-zero finite dimensionality} of vector spaces associated with quasimorphisms: the (quotient) spaces of homogeneous quasimorphisms modulo genuine homomorphisms tend to be either zero or infinite dimensional for groups naturally appearing in geometric group theory. We discuss this point in more detail in the latter part of this subsection. 
 We remark that in Theorem~\ref{mapping torus thm} the group quotient $\Gamma=G/N$ is \emph{non-abelian solvable} in general.

\begin{thm}[Non-zero finite dimensionality in surface groups]\label{thm surface group}
Let $\genus$ be an integer greater than $1$, $\hG = \pi_1(\Sigma_\genus)$ the surface group with genus $\genus$, and $\bG$ the commutator subgroup $[\pi_1(\Sigma_\genus),\pi_1(\Sigma_\genus)]$ of $\pi_1(\Sigma_\genus)$. Then
\[ \dim \big( \QQQ(\bG)^\hG / i^* \QQQ(\hG)\big) = \genus (2\genus-1) \; \textrm{and} \; \dim \big( \QQQ(\bG)^{\hG} / (\HHH^1(\bG)^{\hG} + i^* \QQQ(\hG)) \big) = 1.\]
\end{thm}

For $\genus\in \NN$, let $\Mod(\Sigma_\genus)$ be the mapping class group of the surface $\Sigma_\genus$ and $s_\genus\colon \Mod(\Sigma_\genus)\to \Sp(2\genus,\ZZ)$ the symplectic representation. 
For a mapping class $\psi \in \Mod(\Sigma_\genus)$, we take a diffeomorphism $f$ representing $\psi$ and let $T_f$ denote the mapping torus of $f$.
The fundamental group of $T_f$ is isomorphic to the semidirect product $\pi_1(\Sigma_\genus) \rtimes_{f_*} \ZZ$ and surjects onto $\ZZ^{2\genus} \rtimes_{s_{\genus}(\psi)} \ZZ$ via the abelianization map $\pi_1(\Sigma_\genus) \to \HHH_1(\Sigma_\genus;\ZZ)$.
Note that the kernel of the surjection is equal to the commutator subgroup of $\pi_1(\Sigma_{\genus})$.

\begin{thm}[Non-zero finite dimensionality in  hyperbolic mapping tori]\label{mapping torus thm}
Let $\genus$ be an integer greater than $1$, $\psi \in \Mod(\Sigma_\genus)$ a pseudo-Anosov element and $f$ a diffeomorphism representing $\psi$.
Let $\hG$ be the fundamental group of the mapping torus $T_f$ and $\bG$ the kernel of the surjection $\hG \to \ZZ^{2\genus} \rtimes_{s_{\genus}(\psi)} \ZZ$. Then we have
\[
  \dim \big( \QQQ(\bG)^\hG / i^* \QQQ(\hG)\big) = \dim \Ker(I_{2\genus} - s_\genus(\psi)) + \dim \Ker \left(I_{2\genus \choose 2} - \bigwedge\nolimits^2 s_\genus(\psi)\right)
\]
and
\[
  \dim \big( \QQQ(\bG)^{\hG} / (\HHH^1(\bG)^{\hG} + i^* \QQQ(\hG)) \big) = \dim \Ker(I_{2\genus} - s_\genus(\psi)) + 1.
\]
Here, for $n \in \NN$, $I_n$ denotes the identity matrix of size $n$ and $\bigwedge\nolimits^2 s_\genus(\psi) \colon \bigwedge\nolimits^2 \RR^{2\genus} \to \bigwedge\nolimits^2 \RR^{2\genus}$ is the map induced by $s_\genus(\psi)$.

In particular, if $\psi$ is in the Torelli group $($that is, $s_\genus(\psi) = I_{2\genus}$$)$, then
\[
  \dim \big( \QQQ(\bG)^\hG / i^* \QQQ(\hG)\big) = 2\genus + {2\genus \choose 2} \; \textrm{and} \; \dim \big( \QQQ(\bG)^{\hG} / (\HHH^1(\bG)^{\hG} + i^* \QQQ(\hG)) \big) =2\genus + 1.
\]
\end{thm}

 In Theorem~\ref{mapping torus thm}, the pseudo-Anosov property for $\psi$ is assumed to ensure hyperbolicity of $G$; see Theorem~\ref{thm:hyperbolic_mapping_torus}. In Theorems~\ref{thm free group} and \ref{thm:F_n_torus}, we also obtain analogous results to Theorems~\ref{thm surface group} and \ref{mapping torus thm} in the free group setting.

In study of quasimorphisms, it is often quite hard to obtain \emph{non-zero finite dimensionality}. For instance, if a group $G$ can act non-elementarily in a certain good manner on a  Gromov-hyperbolic  geodesic space, then the dimension of $\QQQ(G)$ is of the cardinal of the continuum (\cite{BF}); contrastingly, a higher rank lattice $G$ has zero $\QQQ(G)$ (\cite{BurMon}). For a group $G$  such that the dimension of $\QQQ(G)$ is of the cardinal of the continuum, understanding \emph{all} quasimorphisms  on $G$ might have been considered as an impossible subject. Our study of \emph{the space of non-extendable quasimorphisms} might have some possibility of shedding light on this problem modulo `trivial or extendable' quasimorphisms.

Theorems~\ref{thm surface group} and \ref{mapping torus thm} treat the case where  $G$ is a non-elementary  Gromov-hyperbolic  group and $N$ a subgroup with solvable quotient. In this case, the result of Epstein and Fujiwara \cite{EF} implies that the dimension of $i^{\ast}\QQQ(G)$ is the cardinal of the \emph{continuum}; this implies that the dimension of $\QQQ(N)^G$ is also the cardinal of the \emph{continuum}. Nevertheless, our results (Theorems~\ref{main thm 2} and \ref{easy cor}) imply that the spaces $\QQQ(N)^G/i^{\ast}\QQQ(G)$ and $\QQQ(N)^G/(\HHH^1(N)^G+i^{\ast}\QQQ(G))$ are always both \emph{finite dimensional}. Theorems~\ref{thm surface group} and \ref{mapping torus thm} provide \emph{non-vanishing} examples, and it might be an interesting problem to understand \emph{all} quasimorphism \emph{classes} in these examples.

We outline how we deduce finite dimensionality of $\QQQ(N)^G/i^{\ast}\QQQ(G)$ and $\QQQ(N)^G/(\HHH^1(N)^G+i^{\ast}\QQQ(G))$ under certain conditions in our results. Our main theorem, Theorem~\ref{main thm} (stated in Subsection~\ref{subsec:mainthm}), establishes the \emph{five-term exact sequence} of the \emph{cohomology} $\HHH_{/b}^{\bullet}$ associated with a short exact sequence of groups
\[
1\longrightarrow N\longrightarrow G\longrightarrow \Gamma \longrightarrow 1.
\]
Here, $\HHH_{/b}^{\bullet}$ relates the \emph{bounded} cohomology $\HHH_{b}^{\bullet}$ with the \emph{ordinary} cohomology $\HHH^{\bullet}$; see Subsection~\ref{subsec:mainthm} fot the precise definition of $\HHH_{/b}^{\bullet}$. In our theorems (Theorems~\ref{main thm 2} and \ref{easy cor}), we assume that $\Gamma$ is \emph{boundedly $3$-acyclic}, meaning that $\HHH_b^2(\Gamma;\RR)=0$ and $\HHH_b^3(\Gamma;\RR)=0$ (Definition~\ref{def:bdd_acyc}). Then, the five-term exact sequence enables us to relate $\QQQ(N)^G/i^{\ast}\QQQ(G)$ and $\QQQ(N)^G/(\HHH^1(N)^G+i^{\ast}\QQQ(G))$, respectively to the \textrm{ordinary} second cohomology $\HHH^2(\Gamma)=\HHH^2(\Gamma;\RR)$ and $\HHH^2(G)=\HHH^2(G;\RR)$. Since second ordinary cohomology is finite dimensional under certain mild conditions, we obtain the desired finite dimensionality results. In this point of view, our main theorem (Theorem~\ref{main thm}) might be regarded as filling in a missing piece between the bounded cohomology theory and the ordinary cohomology theory.

We also note that the extendability and non-extendability of invariant quasimorphisms themselves have applications. See Subsection~\ref{equiv subsection} on application to the stable (mixed) commutator lengths, and Subsection~\ref{intro flux} on one to symplectic geometry. As a notable extrinsic application, we state the following theorem in \cite{KKMM2} by the authors.

\begin{thm}[{\cite[Theorem~1.1]{KKMM2}}] \label{thm flux}
Let $\Sigma_{\genus}$  be a closed orientable surface whose genus $l$ is at least two and  $\Omega$ an area form on $S$. Let $\diff_0(\Sigma_{\genus},\Omega)$ denote the identity component of the group of diffeomorphisms of $\Sigma_{\genus}$ that preserve $\Omega$.
Assume that a pair $f,g \in  \diff_0(\Sigma_{\genus},\Omega)$ satisfies $fg = gf$. Then
\[ \flux_\Omega(f) \smile \flux_\Omega(g) = 0\]
holds true. Here, $\flux_{\Omega}\colon \diff_0(\Sigma_{\genus},\Omega)\to \HHH^1(\Sigma_{\genus};\RR)$ is the volume flux homomorphism, and $\smile \colon \HHH^1(\Sigma_{\genus};\RR) \times \HHH^1(\Sigma_{\genus};\RR) \to \HHH^2(\Sigma_{\genus};\RR) \cong \RR$ denotes the cup product.
\end{thm}

 The statement of Theorem~\ref{thm flux} might not seem to have any relation to quasimorphisms. Nevertheless, the key to the proof is comparison between vanishing and non-vanishing of $\QQQ(N)^G/i^{\ast}\QQQ(G)$, where $G=\flux_{\Omega}^{-1}(\langle \flux_\Omega(f), \flux_\Omega(g)/k\rangle)$ for a sufficiently large integer $k$ and $N=\Ker(\flux_{\Omega})$; see Subsection~\ref{intro flux} for basic concepts around volume flux homomorphisms. (We discuss a related example in Example~\ref{ex:KKMM}.)

\subsection{Main theorem}\label{subsec:mainthm}

To treat the spaces of non-extendable quasimorphisms, we establish the five-term exact sequence of group cohomology relative to bounded cochain complexes. Throughout the paper, the coefficient module of the cohomology groups is the field $\RR$ of real numbers unless otherwise specified.

Let $V$ be a left normed $G$-module, and $C^n(G ; V)$ the space of functions from the $n$-fold direct product $G^{\times n}$ of $G$ to $V$. The group cohomology is defined by the cohomology group of $C^n(G ; V)$ with a certain differential (see Section \ref{sec:preliminaries} for the precise definition). Recall that the spaces $C^n_b(G; V)$ of the bounded functions form a subcomplex of $C^\bullet(G ; V)$, and its cohomology group is the bounded cohomology group of $G$. We write $C^\bullet_{/b}(G ; V)$ to indicate the quotient complex $C^\bullet (G; V) / C^\bullet_b(G ; V)$, and write $\HHH^\bullet_{/b}(G ; V)$ to mean its cohomology group.

Our main result is the five-term exact sequence of the cohomology $\HHH^\bullet_{/ b}$. Before stating our main theorem, we first recall the five-term exact sequence of ordinary group cohomology.

\begin{thm}[Five-term exact sequence of group cohomology] \label{thm five-term}
Let $1 \to \bG \xrightarrow{i} \hG \xrightarrow{p} \ppi \to 1$ be an exact sequence of groups and $\cM$ a left $\RR[\ppi]$-module.
Then there exists an exact sequence
\[0 \to \HHH^1(\ppi;V) \xrightarrow{p^*} \HHH^1(\hG;V) \xrightarrow{i^*} \HHH^1(\bG;V)^{\hG} \xrightarrow{\tau} \HHH^2(\ppi;V) \xrightarrow{p^*} \HHH^2(\hG;V).\]
\end{thm}

The following theorem is the main result in this paper:

\begin{thm}[Main Theorem] \label{main thm}
  Let $1 \to \bG \xrightarrow{i} \hG \xrightarrow{p} \ppi \to 1$ be an exact sequence of groups and $\cM$ a left Banach $\RR[\ppi]$-module equipped with a $\ppi$-invariant norm $\| \cdot \|$.
  Then there exists an exact sequence
  \begin{align}\label{seq:5-term}
    0 \to \HHH_{/b}^1(\ppi;\cM) \xrightarrow{p^*} \HHH_{/b}^1(\hG;\cM) \xrightarrow{i^*} \HHH_{/b}^1(\bG;\cM)^{\hG} \xrightarrow{\tau_{/b}} \HHH_{/b}^2(\ppi;\cM) \xrightarrow{p^*} \HHH_{/b}^2(\hG;\cM).
  \end{align}
  Moreover, the exact sequence above is compatible with the five-term exact sequence of group cohomology, that is, the following diagram commutes:
  \begin{align}\label{diagram:main}
  \xymatrix{
  0 \ar[r] & \HHH^1(\ppi;\cM) \ar[r]^-{p^*} \ar[d]^-{\tae_1} & \HHH^1(\hG;\cM) \ar[r]^-{i^*} \ar[d]^-{\tae_2} & \HHH^1(\bG;\cM)^{\hG} \ar[r]^-{\tau} \ar[d]^-{\tae_3} & \HHH^2(\ppi;\cM) \ar[r]^-{p^*} \ar[d]^-{\tae_4} & \HHH^2(\hG;\cM) \ar[d]^-{\tae_5} \\
  0 \ar[r] & \HHH_{/b}^1(\ppi;\cM) \ar[r]^-{p^*} & \HHH_{/b}^1(\hG;\cM) \ar[r]^-{i^*} & \HHH_{/b}^1(\bG;\cM)^{\hG} \ar[r]^-{\tau_{/b}} & \HHH_{/b}^2(\ppi;\cM) \ar[r]^-{p^*} & \HHH_{/b}^2(\hG;\cM).
  }
  \end{align}
   Here $\tae_i$'s are the maps induced from the quotient map $C^{\bullet} \to C_{/b}^{\bullet}$.
\end{thm}

\begin{remark}
  Since the first relative cohomology group $\HHH_{/b}^1(-) = H_{/b}^1(-;\RR)$ is isomorphic to the space $\QQQ(-)$ of homogeneous quasimorphisms, diagram (\ref{diagram:main}) gives rise to the following:
  \begin{align}\label{diagram_coh_qm_rel}
    \xymatrix{
    0 \ar[r] & \HHH^1(\ppi) \ar[r]^-{p^*} \ar[d]^-{\tae_1} & \HHH^1(\hG) \ar[r]^-{i^*} \ar[d]^-{\tae_2} & \HHH^1(\bG)^{\hG} \ar[r]^-{\tau} \ar[d]^-{\tae_3} & \HHH^2(\ppi) \ar[r]^-{p^*} \ar[d]^-{\tae_4} & \HHH^2(\hG) \ar[d]^-{\tae_5} \\
    0 \ar[r] & \QQQ(\ppi) \ar[r]^-{p^*} & \QQQ(\hG) \ar[r]^-{i^*} & \QQQ(\bG)^{\hG} \ar[r]^-{\tau_{/b}} & \HHH_{/b}^2(\ppi) \ar[r]^-{p^*} & \HHH_{/b}^2(\hG).
    }
  \end{align}
Note that the exactness of the sequence
\[0 \to \QQQ(\Gamma) \xrightarrow{p^*} \QQQ(G) \xrightarrow{i^*} \QQQ(N)^G\]
is well known (see Remark 2.90 of \cite{Ca}).
\end{remark}

\begin{remark}\label{remark:isomQZ}
  It is straightforward to show  that the quotient space $\HHH_{/b}^1(\bG;\cM)^{\hG}/i^*\HHH_{/b}^1(\hG;\cM)$ is isomorphic to $\rQQQ(\bG;\cM)^{\QQQ \hG}/i^* \rQQQ Z(\hG;\cM)$,
  where $\rQQQ Z(\hG;\cM)$ and $\rQQQ(\bG;\cM)^{\QQQ \hG}$ are the spaces of quasi-cocycles on $\hG$ and $\hG$-quasi-equivariant $V$-valued quasimorphisms on $\bG$, respectively
  (see Definition \ref{def:quasicocycle} and Section \ref{subsec:quasicocycle_extension}; see also Remark \ref{rem:V-valued_quasimorphisms}). 
  In Section \ref{subsec:quasicocycle_extension}, we will apply Theorem \ref{main thm} to the extension problem of $\hG$-quasi-equivariant quasimorphisms on $\bG$ to quasi-cocycles on $\hG$.
\end{remark}

This theorem provides several arguments to estimate the dimensions of the spaces $\QQQ(N)^G / i^* \QQQ(G)$ and $\QQQ(N)^G / (\HHH^1(N)^G + i^* \QQQ(G))$ as follows.
 Here we recall the definition of \emph{bounded $k$-acyclicity} of groups from \cite{Ivanov12} and \cite{MR21}.
\begin{definition}(bounded $k$-acyclicity)\label{def:bdd_acyc}
Let $k$ be a positive integer. A group $G$ is said to be \emph{boundedly $k$-acyclic} if $\HHH^i_b(G) = 0$ holds for every positive integer $i$ with $i \le k$.
\end{definition}
We note that $\HHH^1_b(G)=0$ for every  group $G$. We recall properties and examples of boundedly $k$-acyclic groups in Theorem~\ref{thm:bdd_acyc}. In particular, we recall that amenable groups, such as abelian groups, are boundedly $k$-acyclic for all $k$ (Theorem~\ref{amenable base}~(5)).

\begin{thm} \label{main thm 2}
If the quotient group $\Gamma=\hG/\bG$ is boundedly $3$-acyclic, then
\[\dim \big( \QQQ(\bG)^{\hG} / i^* \QQQ(\hG) \big) \leq \dim \HHH^2(\Gamma).\]
Moreover, if $G$ is  Gromov-hyperbolic,  then
\[\dim \big( \QQQ(\bG)^{\hG} / i^* \QQQ(\hG) \big) = \dim \HHH^2(\Gamma).\]
\end{thm}



On the space $\QQQ(\bG)^{\hG} / \left(\HHH^1(\bG)^{\hG} + i^* \QQQ(\hG)\right)$, we also obtain the following:

\begin{thm} \label{easy cor}
If $\Gamma=G/N$ is boundedly $3$-acyclic, then  the map $p^\ast \circ (\tae_4)^{-1} \circ \tau_{/b}$ induces an isomorphism
  \[
    \QQQ(N)^G/ \left(\HHH^1(N)^G + i^* \QQQ(G)\right) \cong \Im (p^\ast) \cap \Im (c_G),
  \]
  where $c_G\colon \HHH_b^2(G) \to \HHH^2(G)$ is the comparison map.
In particular, if $\Gamma$ is boundedly $3$-acyclic, then
\[\dim \big(\QQQ(\bG)^{\hG} / (\HHH^1(\bG)^{\hG} + i^* \QQQ(\hG)) \big) \le \dim \HHH^2(G).\]
\end{thm}

 When $N=[G,G]$, we have a more precise calculation of $\dim \big(\QQQ(\bG)^{\hG} / (\HHH^1(\bG)^{\hG} + i^* \QQQ(\hG)) \big)$  (see Corollary \ref{cor B}).
 As we mentioned in the previous subsection, there are many examples of finitely presented groups such that the space of its homogeneous quasimorphisms is infinite dimensional:  for instance, all non-elementary  Gromov-hyperbolic  groups (\cite{EF}).
Nevertheless, if we assume that $\ppi = \hG/\bG$ is  boundedly $3$-acyclic, then we have the following two statements. The space $\QQQ(\bG)^{\hG} / (\HHH^1(\bG)^{\hG} + i^* \QQQ(\hG))$ is finite dimensional if $\hG$ is finitely presented (following from Theorem \ref{easy cor}); the space $\QQQ(\bG)^{\hG} / i^* \QQQ(\hG)$ is finite dimensional if $\ppi$ is finitely presented (following from Theorem \ref{main thm 2}).


There are several known conditions that guarantee $\QQQ(\bG)^{\hG} = i^* \QQQ(\hG)$, {\it i.e.,} every $\hG$-invariant quasimorphism is extendable (see  \cite{Mal}, \cite{IshidaThesis},   \cite{Sh}, \cite{Ish}, and \cite{KKMM1}).
We say that a group homomorphism $p \colon \hG \to \ppi$ {\it virtually splits} if there exist a subgroup $\fis$ of finite index of $\ppi$ and a group homomorphism $s \colon \fis \to G$ such that $f \circ s(x) = x$ for every $x \in \fis$.
The first, second, fourth, and fifth authors showed that if the group homomorphism $p \colon \hG \to \ppi$ virtually splits, then $\QQQ(\bG)^{\hG} = i^* \QQQ(\hG)$ (see \cite{KKMM1}).
Thus the space $\QQQ(\bG)^{\hG}/i^*\QQQ(\hG)$, which we consider in Theorem \ref{main thm 2}, can be seen as a space of obstructions to the existence of virtual splittings.


\section{Other applications of the main theorem}
 In this section, we provide several other applications of our main theroem (Theorem~\ref{main thm}); we also use its corollaries, Theorems~\ref{main thm 2} and \ref{easy cor}. In the last part of this section, we briefly describe the organization of the present paper.
\subsection{On equivalences of $\scl_{\hG}$ and $\scl_{\hG,\bG}$}\label{equiv subsection}
As an application of the spaces of non-extendable quasimorphisms, we treat the equivalence problems of the stabilizations of usual and mixed commutator lengths.
For two non-negative-valued functions $\mu$ and $\nu$ on a group $G$,
we say that $\mu$ and $\nu$ are \textit{bi-Lipschitzly equivalent} (or \textit{equivalent} in short)
if there exist positive constants $C_1$ and $C_2$ such that $C_1 \nu \leq \mu \leq C_2 \nu$.
By Theorem \ref{easy cor}, $\HHH^2(\hG) = 0$ implies that $\QQQ(\bG)^{\hG} / (\HHH^1(\bG)^{\hG} + i^* \QQQ(\hG)) = 0$ if $\ppi = \hG/\bG$ is  boundedly $3$-acyclic.
We show that the condition $\QQQ(\bG)^{\hG} / (\HHH^1(\bG)^{\hG} + i^* \QQQ(\hG)) = 0$ implies that certain two stable word lengths related to commutators are bi-Lipschitzly equivalent.

Let $G$ be a group and $N$ a normal subgroup. A {\it $(G,N)$-commutator} is an element of $G$ of the form $[g,x] = gxg^{-1}x^{-1}$ for some $g \in G$ and $x \in N$.
Let $[G,N]$ be the group generated by the set of $(G,N)$-commutators.
Then $[G,N]$ is a normal subgroup of $G$. For an element $x$ in $[G,N]$, the {\it $(G,N)$-commutator length or the {\it mixed commutator length} of $x$ is defined to be the minimum number $n$ such that there exist $n$ $(G,N)$-commutators $c_1, \cdots, c_n$ such that $x = c_1 \cdots c_n$, and is denoted by $\cl_{G,N}(x)$.}
Then there exists a limit
\[\scl_{G,N} (x) := \lim_{n \to \infty} \frac{\cl_{G,N}(x^n)}{n}\]
and call $\scl_{G,N}(x)$ the {\it stable $(G,N)$-commutator length of} $x$.

When $\bG = \hG$, then $\cl_{G,G}(x)$ and $\scl_{G,G}(x)$ are called the commutator length and stable commutator length of $x$, respectively; and we write $\cl_{\hG}(x)$ and $\scl_{\hG}(x)$ instead of $\cl_{\hG,\hG}(x)$ and $\scl_{\hG, \hG}(x)$. The commutator lengths and stable commutator lengths have a long history of study, for instance, in the study of theory of mapping class groups (see \cite{EK}, \cite{CMS}, and \cite{MR3494163}) and diffeomorphism groups (see \cite{BIP}, \cite{TsuboiM}, \cite{Tsuboi12}, \cite{Tsuboi17} and \cite{BHW}).
The celebrated Bavard duality theorem \cite{Bav} describes the relationship between homogeneous quasimorphisms and the stable commutator length. In particular, for an element $x \in [G,G]$, $\scl_G(x)$ is non-zero if and only if there exists a homogeneous quasimorphism $f$ on $G$ with $f(x) \ne 0$.

In  \cite{KK} and \cite{KKMM1}, we  construct a pair $(G,N)$ such that $\scl_N$ and $\scl_{G,N}$ are not bi-Lipschitzly equivalent on $[N,N]$.  Contrastingly, in several cases it is known that $\scl_G$ and $\scl_{G,N}$ are bi-Lipschitzly equivalent on $[G,N]$. For example, if the map $p \colon G \to \Gamma = G/N$  virtually splits , then $\scl_G$ and $\scl_{G,N}$ are bi-Lipschitzly equivalent on $[G,N]$. In this paper, the vanishing of $\QQQ(N)^G / (\HHH^1(N)^G + i^* \QQQ(G))$ implies the equivalence of $\scl_G$ and $\scl_{G,N}$ as follows. We note that $\HHH^2(G) = 0$ implies $\QQQ(N)^G / (\HHH^1(N)^G + i^* \QQQ(G)) = 0$ by Theorem \ref{easy cor}.



\begin{thm} \label{main thm 3.1}
Assume that $\QQQ(\bG)^{\hG} = \HHH^1(\bG)^{\hG} + i^* \QQQ(\hG)$. Then
\begin{itemize}
  \item[$(1)$] $\scl_{\hG}$ and $\scl_{\hG,\bG}$ are bi-Lipschitzly equivalent on $[\hG,\bG]$.
  \item[$(2)$] If $\Gamma = G/N$ is  amenable, then $\scl_\hG(x) \le \scl_{\hG, \bG}(x) \le 2 \cdot \scl_{\hG}(x)$ for all $x \in [\hG, \bG]$.
  \item[$(3)$] If $\Gamma = G/N$ is solvable, then $\scl_{\hG}(x) = \scl_{\hG,\bG}(x)$ for all $x \in [\hG,\bG]$.
\end{itemize}
\end{thm}


\begin{remark}\label{remark=boundedlyacyclic}
Recently,  several examples of  non-amenable boundedly acyclic groups have been constructed (see  \cite{FLM1}, \cite{FLM2}, \cite{MN21} and \cite{Monod2021}: we also recall some ofhem in Theorem~\ref{thm:bdd_acyc}).
 However, our proof of (2) of Theorem~\ref{main thm 3.1} does \emph{not} remain working if the assumption of amenability of $\Gamma$ in (2) is replaced with bounded $3$-acyclicity. Indeed, in our proof, we use the fact that $\HHH^2_b(G) \to \HHH^2_b(N)^G$ is \emph{isometric}, which is deduced from amenability of $\Gamma$.
\end{remark}


By Theorem \ref{easy cor}, when $G/N$ is  boundedly $3$-acyclic, then $\HHH^2(G) = 0$ implies that $\QQQ(\bG)^{\hG} = \HHH^1(\bG)^{\hG} + i^* \QQQ(\hG)$, and hence $\scl_{G,N}$ and $\scl_{G}$ are equivalent on $[\hG, \bG]$. There are plenty of examples of groups whose second cohomology groups vanish as follows: 

\begin{itemize}
\item Free groups $F_n$.

\item Let $\genus$ be a positive integer.
Let $N_\genus$ be the non-orientable closed surface with genus $\genus$, and set $G = \pi_1(N_\genus)$. Then, $G = \langle a_1, \cdots, a_\genus \; | \; a_1^2 \cdots a_\genus^2 \rangle$ and $\HHH^2(G) = \HHH^2(N_\genus) = 0$.

\item Let $K$ be a knot in $S^3$. Then the knot group $G$ of $K$ is defined to be the fundamental group of the complement $S^3 \setminus K$.
Since $S^3 \setminus K$ is an Eilenberg-MacLane space, we have $\HHH^2(G) = \HHH^2(S^3 \setminus K) = \widetilde{\HHH}_0(K) = 0$.

\item The braid group $B_n$. Akita and Liu \cite{AL} gave sufficient conditions on a labelled graph $\Gamma$ such that the real second cohomology group of the Artin group $A(\Gamma)$ vanishes (see Corollary 3.21 of \cite{AL}).

\item Free products of the above groups.
\end{itemize}

 For other examples satisfying that $\QQQ(\bG)^{\hG} = \HHH^1(\bG)^{\hG} + i^* \QQQ(\hG)$,  see Example \ref{eg:free_prod} and Corollaries \ref{FujiwaraSoma}, \ref{cor:circle_bundle}, and \ref{cor FFF}.

 Finally, we discuss pairs $(G, N)$ with
non-equivalent $\scl_G$ and $\scl_{G,N}$. 
 In \cite{KK}, the first and second authors provided the first example of  such  $(G,N)$
 (Example~\ref{ex:KK}); we obtain another example with smaller $G$ in Example~\ref{ex:KKMM}, which follows from the work \cite{KKMM2}. These two examples may be seen as one example, coming from symplectic geometry. Unfortunately,
in the present paper, we are unable to provide any new example from a different background.
 We remark that some of the authors \cite{MMM} provided new examples after our work; see the discussion below Problem~\ref{scl not equiv}. 
By Theorem \ref{main thm 3.1}, the vanishing of $\QQQ(N)^G / (\HHH^1(N)^G + i^* \QQQ(G))$ implies the equivalence of $\scl_G$ and $\scl_{G,N}$.
 After this work, the authors \cite{coarse_group} proved that its converse holds if $N = [G, G]$. 
We discuss problems on the equivalence/non-equivalence in more detail in Subsection~\ref{subsec:equivalence}.


\subsection{The case of IA-automorphism groups of free groups}
Here we provide an example that  $\QQQ(\bG)^{\hG} = \HHH^1(\bG)^{\hG} + i^* \QQQ(\hG)$  but $\ppi$ is not  amenable.  Our example comes from the automorphism group of a free group and the IA-automorphisim group.
The group of automorphisms of a group $G$ is denoted by $\Aut(G)$. Let $\IAA_n$ be the IA-automorphism group of the free group $F_n$, {\it i.e.,} the kernel of the natural homomorphism $\Aut(F_n) \to \GL(n,\ZZ)$.
Let $\Aut(F_n)_+$ denote the preimage of $\SL(n,\ZZ)$ in $\Aut(F_n)$.
The following theorem will be proved in Section~\ref{section=proofAut}; see Theorem \ref{thm 7.1.2} for a more general statement.

\begin{thm} \label{thm 3.2}
\begin{enumerate}[$(1)$]
 \item For every $n\geq 2$, $\QQQ(\IAA_n)^{\Aut(F_n)} = i^* \QQQ(\Aut(F_n))$ and $\QQQ(\IA_n)^{\Aut_+(F_n)} = i^* \QQQ(\Aut_+(F_n))$ hold.
 \item For every $n\geq 6$ and for every subgroup $\hG$  of $\Aut(F_n)$ of finite index,  $\QQQ(\bG)^{\hG} = i^* \QQQ(\hG)$ holds. Here, $\bG = \IA_n\cap \hG$.
\end{enumerate}
\end{thm}

\begin{remark}\label{rem=Aut}
\begin{enumerate}[(1)]
  \item The bound `$n\geq 6$' in (2) of Theorem~\ref{thm 3.2} comes from (1) of Theorem \ref{thm 7.9 new}, which treats an effective bound of the \emph{Borel stable range} for second ordinary cohomology with the trivial real coefficient of $\SL_n$.
  \item Corollary 3.8 of \cite{Gersten} implies that $\HHH^2(\Aut(F_n)) = 0$ for $n \ge 5$. However, $\HHH^2(\fis)$ of a subgroup $\fis$ of finite index of $\Aut(F_n)$ is mysterious in general. Even on $\HHH^1$, quite recently it has been proved that $\HHH^1(\fis)=0$ if $n\geq 4$; the proof is based on Kazhdan's property (T) for $\Aut(F_n)$ for $n\geq 4$. See \cite{KNO}, \cite{KKN}, and \cite{Nitsche}. We refer to \cite{BHV} for a comprehensive treatise on property (T).
  Contrastingly, by \cite{McCool}, there exists a subgroup $\fis$ of finite index of $\Aut(F_3)$ such that $\HHH^1(\Lambda)\ne 0$.
  \item The same conclusions as ones in Theorem~\ref{thm 3.2} hold if we replace $\Aut(F_n)$ and $\IAA_n$ with $\mathrm{Out}(F_n)$ and $\overline{\IAA}_n$, respectively. Here, $\overline{\IAA}_n$ denotes the kernel of the natural map $\mathrm{Out}(F_n)\to \GL(n,\ZZ)$. Indeed, the proofs which will be presented in Section~\ref{section=proofAut} remain to work without any essential change.
  \item If $n\geq 3$ and if $\hG$ is a  subgroup of $\Aut(F_n)$ of finite index, then the real vector space $i^* \QQQ(\hG)$ is infinite dimensional. Indeed, we can employ \cite{BBF} to the acylindrically hyperbolic group $\mathrm{Out}(F_n)$, whose  amenable radical is trivial. Thus we may construct an infinitely collection of homogeneous quasimorphisms on $\mathrm{Out}(F_n)$ which is linearly independent even when these quasimorphisms are restricted on $[\overline{\IAA}_n\cap \overline{\hG},\overline{\IAA}_n\cap \overline{\hG}]$.
  Here $\overline{\hG}$ is the image of $\hG$ under the natural projection $\Aut(F_n)\to \mathrm{Out}(F_n)$. Then, consider the restriction of this collection on $\overline{\hG}$, and take the pull-back of it under the projection $\hG\to \overline{\hG}$.

  In fact, Corollary~1.2 of \cite{BBF} treats quasi-cocycles into unitary representations.
  Then the following may be deduced in a similar manner to one above: let $\hG$ be a subgroup of $\Aut (F_n)$ of finite index with $n\geq 3$, and $\ppi := \hG/(\IA_n \cap \hG)$. Let $(\pi,\HH)$ be a unitary $\ppi$-representation, and $(\overline{\pi},\HH)$ the pull-back of it under the projection $\hG \to \ppi$.
  Then the vector space $i^{\ast}\rQQQ Z (\hG,\overline{\pi},\HH)$ of the quasi-cocycles is infinite dimensional. Furthermore, Corollary~1.2 of \cite{BBF} and its proof can be employed to obtain the corresponding result to the setting where $\hG$ is a subgroup of $\Mod (\Sigma_l)$ of finite index with $l\geq 3$, and $(\pi,\HH)$ is a unitary representation of $\hG/(\II(\Sigma_{\genus}) \cap \hG)$.
  Here, $\II(\Sigma_{\genus})$ denotes the Torelli group.
%
\end{enumerate}
\end{remark}

If $(\hG,\bG)$ equals $(\Mod(\Sigma_\genus),\II(\Sigma_\genus))$ or its analog for the setting of subgroups of finite index, then the situation is subtle. See Theorem \ref{thm 7.4} for our result. We remark that the question on the extendability of quasimorphisms might be open; see Problem \ref{prob:mcg}.

\subsection{Applications to volume flux homomorphisms}\label{intro flux}

In Section \ref{flux section}, we will provide applications of Theorem \ref{easy cor} to diffeomorphism groups.

We study the problem to determine which cohomology class admits a bounded representative.
 Notably, the problem on (subgroups of) diffeomorphism groups is interesting and  has been studied in view of characteristic classes of fiber bundles.
However, the problem is often quite difficult, and in fact, there are only a few cohomology classes that are known to be bounded or not.
Here we restrict our attention to the case of degree two cohomology classes.
The best-known example is the Euler class of $\diff_+(S^1)$, which has a bounded representative.
The Godbillon--Vey class integrated along the fiber defines a cohomology class of $\diff_+(S^1)$, which has no bounded representatives \cite{MR298692}.
It was shown in \cite{Cal04} that the Euler class of $\diff_0(\RR^2)$ is unbounded.
In the case of three-dimensional manifolds,  the identity components of the diffeomorphism groups of several closed Seifert-fibered three-manifolds  admit cohomology classes of degree two which do not have bounded representatives \cite{Mann20}.


Let $M$ be an $m$-dimensional manifold and $\Omega$ a volume form.
Then, we can define the flux homomorphism (on the universal covering)
$\tflux_\Omega \colon \tdiff_0(M, \Omega) \to \HHH^{m-1}(M)$,
the flux group $\Gamma_\Omega$,
and the flux homomorphism
$\flux_\Omega \colon \diff_0(M, \Omega) \to \HHH^{m-1}(M)/ \Gamma_\Omega$; see Section \ref{flux section} for the precise definition.

As an application of Theorem \ref{easy cor}, we have a few results related to the comparison maps $\HHH^2_b(\diff_0(M, \Omega)) \to \HHH^2(\diff_0(M, \Omega))$ and $\HHH^2_b(\tdiff_0(M, \Omega)) \to \HHH^2(\tdiff_0(M, \Omega))$.

Kotschick and Morita \cite{KM} essentially pointed out that the spaces $\HHH^2(\diff_0(M, \Omega))$ and $\HHH^2(\tdiff_0 (M, \Omega))$ can be very large due to the following proposition (note that $\HHH^n(\RR^m;\RR)$ is isomorphic to $\mathrm{Hom}_\ZZ\left(\wedge_\ZZ^n(\RR^m);\RR\right)$).

\begin{prop}[\cite{KM}] \label{prop diff}
The homomorphisms
\[\flux^*_\Omega \colon \HHH^2\left(\HHH^{m-1}(M)/ \Gamma_\Omega\right) \to \HHH^2\left(\diff_0(M, \Omega)\right),\]
\[\tflux^*_\Omega \colon \HHH^2\left(\HHH^{m-1}(M)\right) \to \HHH^2\left(\tdiff_0(M, \Omega)\right)\]
induced by the flux homomorphisms are injective.
\end{prop}

As an application of Theorem \ref{easy cor}, we have the following theorem:
\begin{thm} \label{thm diff}
Let $(M,\Omega)$ be an $m$-dimensional closed manifold with a volume form $\Omega$. Then the following hold:
\begin{itemize}
\item[$(1)$] If $m = 2$ and the genus of $M$ is at least $2$, then there exists at least one non-trivial element of $\Im(\flux_\Omega^*)$ represented by a bounded $2$-cochain.
\item[$(2)$] Otherwise, every non-trivial element of $\Im(\flux_\Omega^*)$ and $\Im(\tflux_\Omega^*)$ cannot be represented by a bounded $2$-cochain.
\end{itemize}
\end{thm}
Note that in case (1), it is known that $\pi_1\left(\diff_0(M, \Omega)\right)=0$  ( for example, see Subsection 7.2.B of \cite{P01}), in particular, the flux group $\Gamma_\Omega$ is zero.

In the proof of (1) of Theorem \ref{thm diff}, we essentially prove the non-triviality of the cohomology class $c_P \in \Im(\flux_\Omega^*)$ called the \textit{Py class}.
In Subsection \ref{Py class subsec}, we provide some observations on the Py class.

\subsection{Organization of the paper}

Section \ref{sec:preliminaries} collects preliminary facts.
In Section \ref{sec:non-extendable}, we first prove Theorem \ref{main thm 2} and \ref{easy cor}, assuming Theorem \ref{main thm}.
Secondly, we show Theorems \ref{thm surface group} and \ref{mapping torus thm}.
In Section \ref{flux section}, we provide applications of Theorem \ref{main thm} to the volume flux homomorphisms.
Section \ref{sec:pf_of_five-term} is devoted to the proof of Theorem \ref{main thm}.
In Section \ref{sec:scl}, we prove Theorem \ref{main thm 3.1}. 
In Section \ref{section=proofAut}, we prove Theorem \ref{thm 3.2}.
In Section \ref{open problem section}, we provide several open problems. In Appendix, we show other exact sequences related to the space $\QQQ(\hG) / (\HHH^1(\bG)^\hG + i^* \QQQ(\hG))$ and the seven-term exact sequence of groups.

\section{Preliminaries}\label{sec:preliminaries}

Before proceeding to the main part of this section, we collect basic properties of quasimorphisms  and state them as Lemmas~\ref{lem homogeneous} and \ref{lem:homog}; see \cite{Ca} for more details.  Lemma~\ref{lem homogeneous} follows from the equality $(ghg^{-1})^n=gh^ng^{-1}$ for every $g,h\in G$ and every $n\in \ZZ$.

\begin{lem}[See \cite{Ca}] \label{lem homogeneous}
A homogeneous quasimorphism is conjugation invariant.
\end{lem}

In particular, the restriction of a homogeneous quasimorphism $f$ of $G$ to a normal subgroup $N$ is $G$-invariant.

For a quasimorphism $f \colon G \to \RR$, the Fekete lemma guarantees that the limit
\[ \bar{f}(x) = \lim_{n \to \infty} \frac{f(x^n)}{n}\]
exists. The function $\bar{f}$ defined by the above equation is called the \emph{homogenization of $f$}. Then the following hold;  see \cite[Lemma~2.58]{Ca} for (3).

\begin{lem}[See \cite{Ca}]\label{lem:homog}
The following hold:
\begin{enumerate}[$(1)$]
\item $\bar{f}$ is a homogeneous quasimorphism.

\item $|\bar{f}(x) - f(x)| \le D(f)$ for every $x \in G$.

\item $D(\bar{f}) \le 2 D(f)$.
\end{enumerate}
\end{lem}

In this section, we recall definitions and facts related to the cohomology of groups. For a more comprehensive introduction to this subject, we refer to \cite{Gr}, \cite{Ca}, and \cite{Fr}.

Let $V$ be a left $\RR[G]$-module and $C^n(G; V)$ the vector space consisting of functions from the $n$-fold direct product $G^{\times n}$ to $V$. Let $\delta \colon C^n(G ; V) \to C^{n+1}(G ; V)$ be the $\RR$-linear map defined by
\[(\delta f)(g_0, \cdots, g_n) = g_0 \cdot f(g_1, \cdots, g_n) + \sum_{i=1}^n (-1)^i f(g_0, \cdots, g_{i-1} g_i, \cdots, g_n) + (-1)^{n+1} f(g_0, \cdots, g_{n-1}).\]
Then $\delta^2 = 0$ and its $n$-th cohomology is the \textit{ordinary group cohomology} $\HHH^n(G ; V)$.

Next, suppose that $V$ is equipped with a $G$-invariant norm $\| \cdot \|$, {\it i.e.,} $\| g \cdot v \| = \| v \|$ for every $g \in G$ and for every $v \in V$. Then define $C^n_b(G;V)$ by the subspace
\[C^n_b(G ; V) = \Big\{ f \colon  G^{\times n}  \to V \; \Big| \; \sup_{(g_1, \cdots, g_n) \in G^{\times n}} \| f(g_1, \cdots, g_n)\| < \infty \Big\}\]
of $C^n(G; V)$. Then $C^\bullet_b(G ; V)$ is a subcomplex of $C^\bullet(G ; V)$, and we call the $n$-th cohomology of $C^\bullet_b(G; V)$ the \textit{$n$-th bounded cohomology of $G$}, and denote it by $\HHH^n_b(G; V)$.
 The inclusion $C^\bullet_b(\hG ; V) \to C^\bullet(\hG ; V)$ induces the map $c_{\hG} \colon \HHH^{\bullet}_b(\hG;V) \to \HHH^{\bullet}(\hG;V)$ called the \textit{comparison map}.

Let $\HHH_{/b}^{\bullet}(\hG;V)$ denote their relative cohomology, that is, the cohomology of the quotient complex $C_{/b}^{\bullet}(\hG;V) = C^{\bullet}(\hG;V)/C_b^{\bullet}(\hG;V)$.
Then, the short exact sequence of cochain complexes
\begin{align*}
  0 \to C_b^{\bullet}(\hG;V) \to C^{\bullet}(\hG;V) \to C_{/b}^{\bullet}(\hG;V) \to 0
\end{align*}
induces the cohomology long exact sequence
\begin{align}\label{seq:long_seq}
  \cdots \to \HHH_b^n(\hG;V) \xrightarrow{c_{\hG}} \HHH^n(\hG;V) \to \HHH_{/b}^n(\hG;V) \to \HHH_b^{n+1}(\hG;V) \to \cdots.
\end{align}

If we need to specify the $\hG$-representation $\rho$, we may use the symbols $\HHH^{\bullet}(\hG; \rho, V)$, $\HHH_b^{\bullet}(\hG; \rho, V)$, and $\HHH_{/b}^{\bullet}(\hG; \rho, V)$ instead of $\HHH^{\bullet}(\hG; V)$, $\HHH_b^{\bullet}(\hG; V)$, and $\HHH_{/b}^{\bullet}(\hG; V)$, respectively.
Let $\RR$ denote the field of real numbers equipped with the trivial $G$-action.
In this case, we write $\HHH^n(G)$, $\HHH^n_b(G)$, and $\HHH_{/b}^n(G)$ instead of $\HHH^n(G ; \RR)$, $\HHH^n_b(G ; \RR)$, and $\HHH_{/b}^n(G;\RR)$, respectively.

Let $N$ be a normal subgroup of $G$. Then $G$ acts on $N$ by conjugation, and hence $G$ acts on $C^n(N; V)$. This $\hG$-action is described by
\[({}^g f) (x_1, \cdots, x_n) = g\cdot f(g^{-1} x_1 g, \cdots, g^{-1} x_n g).\]
The action induces $G$-actions on $\HHH^n(N;V)$, $\HHH^n_b(N;V)$, and $\HHH_{/b}^n(N;V)$.
When $N = G$, these $G$-actions on $\HHH^n(G;V)$, $\HHH^n_b(G;V)$, and $\HHH_{/b}^n(G;V)$ are trivial.
By definition, a cocycle $f \colon N \to V$ in $C_{/b}^1(N;V)$ defines a class of $\HHH_{/b}^1(N;V)^G$ if and only if the function ${}^{g} f - f \colon N \to V$ is bounded for every $g \in \hG$.

Until the end of Section \ref{flux section}, we consider the case of the trivial real coefficients.
Let $f \colon G \to \RR$ be a homogeneous quasimorphism. Then $f$ is considered as an element of $C^1(G)$, and its coboundary $\delta f$ is
\[(\delta f)(x,y) = f(x) - f(xy) + f(y).\]
Since $f$ is a quasimorphism, the coboundary $\delta f$ is a bounded cocycle.
Hence we obtain a map $\delta \colon \QQQ(G) \to \HHH^2_b(G)$ by $\qm \mapsto [\delta \qm]$. Then the following lemma is well known:

\begin{lem} \label{lem 3.1}
The following sequence is exact:
\[0 \to \HHH^1(G) \to \QQQ(G) \xrightarrow{\delta} \HHH^2_b(G) \xrightarrow{c_G} \HHH^2(G).\]
\end{lem}

Let $\varphi \colon G \to H$ be a group homomorphism. A \textit{virtual section of $\varphi$} is a pair $(\fis, x)$ consisting of a subgroup $\fis$ of finite index of $H$ and a group homomorphism $s \colon \fis \to G$ satisfying $\varphi(s(x)) = x$ for every $x \in \fis$. The group homomorphism $\varphi$ is said to \textit{virtually split} if $\varphi$ admits a virtual section. 
 As mentioned at the end of the introduction, some of the authors showed the following proposition. 
For a further generalization of this result, see Theorem 1.4 of \cite{KKMM2}.

\begin{prop}[Proposition 6.4 of \cite{KKMM1}]\label{virtual split KKMM1}
If the projection $p \colon \hG \to \ppi$ virtually splits, then the map $i^* \colon \QQQ(G) \to \QQQ(N)^G$ is surjective.
\end{prop}

In the present paper, we often consider  amenable groups
 and boundedly acyclic groups.
 Here, we review basic properties related to them. First, we collect those for amenable groups  (for example, see \cite{Fr}  for more details).
\begin{thm}[Known results for amenable groups]\label{amenable base}
 The following hold.
\begin{enumerate}[$(1)$]
\item Every finite group is  amenable.
\item Every abelian group is  amenable.
\item Every subgroup of an  amenable group is  amenable.
\item Let $1 \to \bG \to \hG \to \ppi \to 1$ be an exact sequence of groups.
Then $\hG$ is  amenable if and only if $\bG$ and $\ppi$ are  amenable.
\item Every  amenable group is boundedly $k$-acyclic for all $k\geq 1$.
\end{enumerate}
\end{thm}

Secondly,  we collect 
 known results on bounded $k$-acyclicity for $k\geq 3$ 
by various researchers; these results are not used in the present paper, but it might be convenient to the reader to have some examples of boundedly $3$-acyclic groups that are non-amenable. See also Remark~\ref{rem:bdd_3} for one more example.

\begin{thm}[Known results for boundedly acyclic groups]\label{thm:bdd_acyc}
The following hold.
\begin{enumerate}[\textup{(}$1$\textup{)}]

\item \textup{(}\cite{MatsumotoMorita}\textup{)}  Let $n\in \NN$. Then, the group $\Homeo_{c}(\RR^n)$ of homeomorphisms on $\RR^n$ with compact support is boundedly acyclic.

\item \textup{(}combination of \cite{Mon04} and \cite{MonodShalom}\textup{)}
For $n \geq 3$, every lattice in ${\rm SL}(n,\RR)$ is $3$-boundedly acyclic.

\item \textup{(}\cite{BucherMonod}\textup{)} Burger--Mozes groups \cite{BurgerMozes} are $3$-boundedly acyclic.

\item \textup{(} see  \cite{MR21}\textup{)} Let $k\in \NN$. Let $1\to N \to G \to \Gamma \to 1$ be a short exact sequence of groups. Assume that $N$ is boundedly $k$-acyclic. Then $G$ is boundedly $k$-acyclic if and only if $\Gamma$ is.
\item \textup{(}\cite{FLM1}\textup{)} Every binate group \textup{(}see \cite[Definition~3.1]{FLM1}\textup{)} is boundedly acyclic.
\item \textup{(}\cite{FLM2}\textup{)} There exist continuum many non-isomorphic $5$-generated non-amenable groups that are boundedly acyclic. There exists a finitely presented non-amenable group that is boundedly acyclic.
\item \textup{(}\cite{Monod2021}\textup{)} Thompson's group $F$ is boundedly acyclic.
\item \textup{(}\cite{Monod2021}\textup{)} Let $L$ be an arbitrary group. Let $\Gamma$ be an infinite amenable group. Then the wreath product $L\wr \Gamma =\left( \bigoplus_{\Gamma}L\right)\rtimes \Gamma$ is boundedly acyclic.

\item \textup{(}\cite{MN21}\textup{)} For every integer $n$ at least two, the identity component $\Homeo_0(S^n)$ of the group of orientation-preserving homeomorphisms of $S^n$ is boundedly $3$-acyclic. The group $\Homeo_0(S^3)$ is boundedly $4$-acyclic.
\end{enumerate}
\end{thm}

On  (7),  we remark that it is a major open problem asking whether Thompson's group $F$ is amenable.

The seven-term exact sequence and the calculation of first cohomology mentioned below will be used in the proof of Theorem \ref{thm:F_n_torus}.

\begin{thm}[Seven-term exact sequence] \label{thm seven-term}
Let $1 \to \bG \xrightarrow{i} \hG \xrightarrow{p} \ppi \to 1$ be an exact sequence. Then there exists the following exact sequence:
\begin{align*}
  0 \to &\HHH^1(\ppi) \xrightarrow{p^*} \HHH^1(\hG) \xrightarrow{i^*} \HHH^1(\bG)^\hG \to \HHH^2(\ppi) \\
  &\to {\rm Ker}(i^* \colon \HHH^2(\hG) \to \HHH^2(\bG)) \xrightarrow{\rho} \HHH^1(\ppi ; \HHH^1(\bG)) \to \HHH^3(\ppi).
\end{align*}
Here $\HHH^1(\bG)$ is regarded as a left $\RR[\ppi]$-module by the $\ppi$-action induced from the conjugation $\hG$-action on $\bG$.
\end{thm}

\begin{lem}\label{lem:first_coh}
  For a left $\RR[\ZZ]$-module $V$, let $\rho \colon \ZZ \to \Aut(V)$ be the representation.
  Then, the first cohomology group $\HHH^1(\ZZ;V)$ is isomorphic to $V/\mathrm{Im}(\mathrm{id}_V - \rho(1))$.
\end{lem}

\begin{proof}
  By definition, the set $Z^1(\ZZ;V)$ of cocycles on $\ZZ$ with coefficients in $V$ is equal to the set of crossed homomorphisms, that is,
  \[
    \{ h \colon \ZZ \to V \mid h(n + m) = \rho(n)(h(m)) + h(n) \text{ for every } n, m \in \ZZ \}.
  \]
  Since every crossed homomorphism on $\ZZ$ is determined by its value on $1 \in \ZZ$, we have $Z^1(\ZZ;V) \cong V$.
  The set $B^1(\ZZ;V)$ of coboundaries on $\ZZ$ with coefficients in $V$ is equal to
  \[
    \{ h \colon \ZZ \to V \mid h(1) = v - \rho(1)(v) \text{ for some } v \in V \}.
  \]
  Hence we have $B^1(\ZZ;V) \cong \mathrm{Im}(\mathrm{id}_V - \rho(1))$ and the lemma follows.
\end{proof}

\section{The spaces of non-extendable quasimorphisms}\label{sec:non-extendable}

The purpose of this section is to provide several applications of our main theorem (Theorem \ref{main thm}) to the spaces $\QQQ(\bG)^\hG / i^* \QQQ(\hG)$ and $\QQQ(\bG)^\hG / (\HHH^1(\bG)^\hG + i^* \QQQ(\hG))$. In Section  \ref{subsec:dim_estimate},  we prove Theorems \ref{main thm 2} and \ref{easy cor} modulo Theorem \ref{main thm}, and in  Sections \ref{subsec:prf_surf}--\ref{subsec:other_ex},  we provide several examples of pairs $(G,N)$ such that the space $\QQQ(N)^G /(\HHH^1(N)^G + i^* \QQQ(G))$ does not vanish (Theorems \ref{thm surface group}, \ref{mapping torus thm}, and \ref{thm:one-relator}).

\subsection{Proofs of Theorems \ref{main thm 2} and \ref{easy cor}} \label{subsec:dim_estimate}

The goal of this section is to prove Theorems \ref{main thm 2} and \ref{easy cor} modulo Theorem \ref{main thm}.

First, we prove Theorem \ref{main thm 2}.
Recall that if $G$ is  Gromov-hyperbolic,  then the comparison map $\HHH^2_b(G) \to \HHH^2(G)$ is surjective \cite{MR919829}.
Hence, Theorem \ref{main thm 2} follows from the following:

\begin{thm} \label{thm A}
Let $1 \to \bG \to \hG \to \ppi \to 1$ be an exact sequence of groups.
Assume that  $\ppi$ is boundedly $3$-acyclic.
Then the following inequality holds:
\[\dim \left( \QQQ(N)^G / i^* \QQQ(G) \right) \leq \dim \HHH^2(\ppi).\]
Moreover,  if the comparison map $c_G\colon \HHH^2_b(G) \to \HHH^2(G)$ is surjective, then
\[\dim \left( \QQQ(N)^G / i^* \QQQ(G) \right) = \dim \HHH^2(\ppi).\]
\end{thm}

\begin{proof}
By Theorem \ref{main thm}, we have the exact sequence
  \[ \QQQ(\hG) \xrightarrow{i^*} \QQQ(\bG)^{\hG} \xrightarrow{\tau_{/b}} \HHH_{/b}^2(\ppi).\]
Hence, we have
  \[\dim \left( \QQQ(N)^G / i^* \QQQ(G)\right) \leq \dim \HHH_{/b}^2(\ppi).\]
Since  $\ppi$ is boundedly $3$-acyclic, the map $\tae_4\colon \HHH^2(\ppi)\to\HHH^2_{/b}(\ppi)$ is an isomorphism  by \eqref{seq:long_seq}, and therefore we have
\[\dim \left( \QQQ(N)^G / i^* \QQQ(G)\right) \leq \dim \HHH^2(\ppi).\]

Next, we show the latter assertion. Suppose that the comparison map $c_G\colon \HHH^2_b(G) \to \HHH^2(G)$ is surjective.
Then, the map $\tae_5 \colon \HHH^2(\hG)\to\HHH^2_{/b}(\hG)$ is the zero-map.
Since $\tae_4\colon \HHH^2(\ppi)\to\HHH^2_{/b}(\ppi)$ is an isomorphism, the map $p^\ast\colon \HHH^2_{/b}(\ppi) \to \HHH^2_{/b}(\hG)$ is also zero.
Hence
\[\dim \left( \QQQ(N)^G / i^* \QQQ(G) \right) = \dim \HHH_{/b}^2(\ppi) = \dim \HHH^2(\ppi). \qedhere \]
\end{proof}

To prove Theorem \ref{easy cor}, we use the following lemma in homological algebra.

\begin{lem}\label{lemma:general_nonsense}
  For a commutative diagram of $\RR$-vector spaces
  \[
  \xymatrix{
    & & & C \ar[d]^-{c}\\
    & B_2 \ar[r]^-{b_2} \ar[d]^-{c_2} & B_3 \ar[r]^-{b_3} \ar[d]^-{c_3}_-{\cong} & B_4 \ar[d]^-{c_4}\\
    A_1 \ar[r]^-{a_1} & A_2 \ar[r]^-{a_2} & A_3 \ar[r]^-{a_3} & A_4,
  }
  \]
  where the  rows  and the last column are exact and $c_3$ is an isomorphism, the map $b_3 \circ c_3^{-1} \circ a_2$ induces an isomorphism
  \[
    A_2/(\Im (a_1) + \Im (c_2)) \cong \Im (b_3) \cap \Im (c).
  \]
\end{lem}
Because the proof of Lemma \ref{lemma:general_nonsense} is done by a standard diagram chasing, we omit it.

\begin{proof}[Proof of Theorem $\ref{easy cor}$]
If $\Gamma=\hG/\bG$ is  boundedly $3$-acyclic,
$\tae_4\colon \HHH^2(\ppi)\to\HHH^2_{/b}(\ppi)$ is an isomorphism.
Therefore Theorem \ref{easy cor} follows by applying Lemma \ref{lemma:general_nonsense} to commutative diagram \eqref{diagram_coh_qm_rel}.
\end{proof}

The following corollary of Theorem \ref{easy cor} will be used in the proof of Theorem \ref{thm surface group}.

\begin{cor} \label{cor B}
Assume that $\bG$ is contained in the commutator subgroup  $[\hG,\hG]$  of $\hG$, and  $\ppi$ is boundedly $3$-acyclic. Then the following inequality holds:
\[\dim \big(\QQQ(\bG)^{\hG} / (\HHH^1(\bG)^{\hG} + i^* \QQQ(\hG))\big) \leq \dim \HHH^2(\ppi) - \dim \HHH^1(\bG)^{\hG}.\]
Moreover, if the comparison map $\HHH^2_b(\hG) \to \HHH^2(\hG)$ is surjective,
\[\dim \big(\QQQ(\bG)^{\hG} / (\HHH^1(\bG)^{\hG} + i^* \QQQ(\hG))\big) = \dim \HHH^2(\ppi) - \dim \HHH^1(\bG)^{\hG}.\]
\end{cor}
\begin{proof}
Since $\bG$ is contained in the commutator subgroup of $\hG$, we  conclude that the map $i^* \colon \HHH^1(\hG) \to \HHH^1(\bG)^{\hG}$ is zero, and hence $\dim \Im(p^*) = \dim \HHH^2(\ppi) - \dim \HHH^1(\bG)^{\hG}$.
Therefore Theorem \ref{easy cor} implies the corollary.
\end{proof}

\begin{remark}
  Let $1 \to \bG \to \hG \to \ppi \to 1$ be an exact sequence, and suppose that the group  $\bG$ is  amenable.
  Then it is known that the map $\tae_3 \colon \HHH^1(\bG)^{\hG} \to \QQQ(\bG)^{\hG}$ in (\ref{diagram_coh_qm_rel}) is an isomorphism.
  Hence, Lemma \ref{lemma:general_nonsense} implies that the composite $\tau \circ \tae_3^{-1} \circ i^*$ induces an isomorphism
  \begin{align*}
    \QQQ(\hG)/(\HHH^1(\hG) + p^*\QQQ(\ppi)) \cong \Im(\tau) \cap \Im(c_{\ppi}).
  \end{align*}
  This isomorphism was obtained in \cite{2012.10612} in a different way and applied to study boundedness of characteristic classes of foliated bundles.
\end{remark}

\subsection{Proof of Theorem \ref{thm surface group}} \label{subsec:prf_surf}

The goal of this subsection is to prove Theorem \ref{thm surface group} by using the results proved in the previous subsection.  This theorem treats surface groups. Before proceeding to this case, we first prove the following theorem for free groups.
%
\begin{thm}[Computations of dimensions for free groups] \label{thm free group}
For $n \ge 1$, set $\hG = F_n$ and $\bG = [F_n, F_n]$. Then
\[ \dim \big( \QQQ(\bG)^\hG /  i^* \QQQ(\hG) \big) = \frac{n(n-1)}{2} \; \textrm{and} \; \dim \big(\QQQ(\bG)^\hG / (\HHH^1(\bG)^\hG + i^* \QQQ(\hG))\big) = 0.\]
\end{thm}

\begin{proof}
  By Theorem \ref{thm A}, we have
  \[
    \dim \big( \QQQ(\bG)^\hG /   i^* \QQQ(\hG)  \big) = \dim \HHH^2(\hG/\bG) = \dim \HHH^2(\ZZ^n) = \frac{n(n-1)}{2}
  \]
  By Theorem \ref{easy cor}, we obtain
  \[
    \dim \big(\QQQ(\bG)^\hG / (\HHH^1(\bG)^\hG + i^* \QQQ(\hG))\big) \leq \dim \HHH^2(G) = 0. \qedhere
  \]
\end{proof}

Next we show Theorem \ref{thm surface group}.
In the proof, we need the precise description of the space $\HHH^1(
[F_n, F_n] )^{F_n}$ of $F_n$-invariant homomorphisms on the commutator subgroup $[F_n, F_n]$ of the free group $F_n$. Throughout this subsection, we write $a_1, \cdots, a_n$ to mean the canonical basis of $F_n$.


\begin{lem} \label{lem F_n}
Let $i$ and $j$ be integers such that $1 \le i < j \le n$. Then there exist $F_n$-invariant homomorphisms $\alpha_{i,j} \colon  [F_n,F_n]  \to \RR$ such that for $k,l \in \ZZ$ with $1 \le k < l \le n$,
\begin{eqnarray} \label{eqn F_n}
\alpha_{i,j}([a_k, a_l]) = \begin{cases}
1 & ((i,j) = (k,l)) \\
0 & {\rm otherwise.}
\end{cases}
\end{eqnarray}
Moreover, $\alpha_{i,j}$ are a basis of $\HHH^1(  [F_n,F_n] )^{F_n}$. In particular, \[\dim \HHH^1(  [F_n,F_n] )^{F_n} = \frac{n (n-1)}{2}.\]
\end{lem}

\begin{proof}
When $G = F_n$ and $N =   [F_n,F_n] $, the five-term exact sequence (Theorem \ref{thm five-term}) implies that the dimension of $\HHH^1(  [F_n,F_n] )^{F_n}$ is $n(n-1)/2$. Hence it suffices to construct $\alpha_{i,j}$ satisfying \eqref{eqn F_n}.

We first consider the case $n = 2$. Since $\dim(\HHH^1(  [F_2,F_2] )^{F_2}) = 1$, it suffices to show that there exists an $F_2$-invariant homomorphism $\alpha \colon  [F_2,F_2]  \to \RR$ with $\alpha ([a_1,a_2]) \ne 0$. Let $\varphi \colon  [F_2,F_2]  \to \RR$ be a non-trivial $F_2$-invariant homomorphism.
Then there exists a pair $x$ and $y$ of elements of $F_2$ such that $\varphi ([x,y]) \ne 0$. Let $f \colon F_2 \to F_2$ be the group homomorphism sending $a_1$ to $x$ and $a_2$ to $y$.
Then $\varphi \circ (f|_{ [F_2,F_2] }) \colon  [F_2,F_2]  \to \RR$ is an $F_2$-invariant homomorphism satisfying $\varphi \circ f([a_1,a_2]) \ne 0$.
This completes the proof of the case $n = 2$.

Suppose that $n \ge 2$. Then for $i, j \in \{ 1, \cdots, n\}$ with $i < j$, define a homomorphism $q_{i,j} \colon F_n \to F_2$ which sends $a_i$ to $a_1$, $a_j$ to $a_2$, and $a_k$ to the unit element of $F_2$ for $k \ne i, j$. Then $q_{i,j}$ induces a surjection $ [F_n,F_n] $ to $ [F_2,F_2] $, and induces a homomorphism $q_{i,j}^* \colon  \HHH^1( [F_2,F_2] )^{F_2} \to \HHH^1( [F_n,F_n] )^{F_n}$. Set $\alpha_{i,j} = \alpha_{1,2} \circ q_{i,j}$.
Then $\alpha_{i,j}$ clearly satisfies \eqref{eqn F_n}, and this completes the proof.
\end{proof}

Theorem \ref{thm surface group} follows from Corollary \ref{cor B} and the following proposition:

\begin{prop}\label{prop:inv_hom_surface_group}
For $\genus \ge 1$, the following equality holds:
\[\dim \HHH^1([\pi_1(\Sigma_\genus),\pi_1(\Sigma_\genus)])^{\pi_1(\Sigma_\genus)} = \genus(2\genus - 1) - 1.\]
\end{prop}
\begin{proof}
Recall that $\pi_1(\Sigma_\genus)$ has the following presentation:
\[\langle a_1, \cdots, a_{2\genus} \; | \; [a_1, a_2] \cdots [a_{2\genus - 1}, a_{2\genus}]\rangle.\]
Let $f \colon F_{2\genus} \to \pi_1(\Sigma_\genus)$ be the natural epimorphism sending $a_i$ to $a_i$, and $K$ the kernel of $f$, {\it i.e.,} $K$ is the normal subgroup generated by $[a_1, a_2] \cdots [a_{2\genus-1}, a_{2\genus}]$ in $F_{2\genus}$. Then $f$ induces an epimorphism $f|_{ [F_{2\genus},F_{2\genus}] } \colon  [F_{2\genus},F_{2\genus}]  \to [\pi_1(\Sigma_\genus),\pi_1(\Sigma_\genus)]$ between their commutator subgroups, and its kernel coincides with $K$ since $K$ is contained in $  [F_{2\genus},F_{2\genus}] $. This means that for a homomorphism $\varphi \colon  [F_{2\genus},F_{2\genus}]  \to \RR$, $\varphi$ induces a homomorphism $\overline{\varphi} \colon [\pi_1(\Sigma_\genus),\pi_1(\Sigma_\genus)] \to \RR$ if and only if
\[\varphi ([a_1, a_2] \cdots [a_{2\genus-1}, a_{2\genus}]) = 0.\]
It is straightforward to show  that $\varphi$ is $F_{2\genus}$-invariant if and only if $\overline{\varphi}$ is $\pi_1(\Sigma_\genus)$-invariant. Hence the image of the monomorphism $\HHH^1([\pi_1(\Sigma_\genus),\pi_1(\Sigma_\genus)])^{\pi_1(\Sigma_\genus)} \to \HHH^1( [F_{2\genus},F_{2\genus}] )^{F_{2\genus}}$ is the subspace consisting of elements
\[\sum_{i < j} k_{ij} \alpha_{ij}\]
such that
\[k_{1,2} + k_{3,4} + \cdots + k_{2\genus-1, 2\genus} = 0.\]
Since the dimension of $\HHH^1( [F_{2\genus},F_{2\genus}] )^{F_{2\genus}}$ is $\genus(2\genus-1)$ (see Lemma \ref{lem F_n}), this completes the proof.
\end{proof}

\begin{proof}[Proof of Theorem $\ref{thm surface group}$]
Since the abelianization $\ppi = \pi_1(\Sigma_\genus)/[\pi_1(\Sigma_\genus),\pi_1(\Sigma_\genus)]$ of the surface group is isomorphic to $\ZZ^{2\genus}$, we have $\dim \HHH^2(\ppi) = \genus(2\genus-1)$.  Thus the first assertion follows from
Theorem \ref{thm A} 
  Since the comparison map $\HHH_b^2(\pi_1(\Sigma_\genus)) \to \HHH^2(\pi_1(\Sigma_\genus))$ is surjective, we obtain
  \[
    \dim \left(\QQQ(\bG)^{\hG} / (\HHH^1(\bG)^{\hG} + i^* \QQQ(\hG))\right) = 1
  \]
  by Corollary \ref{cor B} and Proposition \ref{prop:inv_hom_surface_group}.
\end{proof}


\subsection{Proof of Theorem \ref{mapping torus thm} and a related example}

To prove Theorem \ref{mapping torus thm}, we now recall some terminology of mapping class groups.

Let $\genus$ be an integer at least $2$ and $\Sigma_{\genus}$ the oriented closed surface with genus $\genus$. The {\it mapping class group ${\rm Mod}(\Sigma_{\genus})$ of $\Sigma_{\genus}$} is the group of isotopy classes of orientation preserving diffeomorphisms on $\Sigma_{\genus}$.
By considering the action on the first homology group, the mapping class group $\Mod(\Sigma_{\genus})$ has a natural epimorphism $s_l \colon \Mod(\Sigma_{\genus}) \to \Sp(2\genus ; \ZZ)$ called the \emph{symplectic representation}.

For $\psi \in \Mod(\Sigma_\genus)$, we take a diffeomorphism $f$ that represents $\psi$.
The mapping torus $T_f$ is an orientable closed $3$-manifold equipped with a natural fibration structure $\Sigma_{\genus} \to T_f \to S^1$. 
The following is known.

\begin{thm}[\cite{math/9801045}]\label{thm:hyperbolic_mapping_torus}
A mapping class $\psi$ is a pseudo-Anosov element if and only if the mapping torus $T_f$ is a hyperbolic manifold.
\end{thm}

Set $\ppi = \ZZ^{2\genus} \rtimes_{s_{\genus}(\psi)} \ZZ$ and
\begin{align}\label{presentation_mapping_torus}
  \hG &= \pi_1(T_f) = \pi_1(\Sigma_\genus) \rtimes_{f_*} \ZZ \\
  &= \left\langle a_1, \cdots, a_{2\genus+1} \; \middle| \;
\begin{gathered}
    {[a_1, a_2] \cdots [a_{2\genus-1}, a_{2\genus}] = 1_{\hG}},  \\
    a_{2\genus+1} \cdot a_i = (f_* a_i) \cdot a_{2\genus+1}  \text{ for every } 1 \leq i \leq 2\genus
\end{gathered}
 \right\rangle, \nonumber
\end{align}
where $f_* \colon \pi_1(\Sigma_{\genus}) \to \pi_1(\Sigma_{\genus})$ is the pushforward of $f$.

\begin{lem}\label{lem:dim_G_GG}
  The following hold true.
  \begin{enumerate}[$(1)$]
    \item $\dim \HHH^2(\ppi) = \dim \Ker(I_{2\genus} - s_\genus(\psi)) + \dim \Ker \left(I_{2\genus \choose 2} - \bigwedge\nolimits^2 s_\genus(\psi)\right)$.
    \item $\dim \HHH^2(\hG) = \dim \Ker(I_{2\genus} - s_\genus(\psi)) + 1$.
  \end{enumerate}
\end{lem}

\begin{proof}
  Let $T^{2\genus}$ be the $2\genus$-dimensional torus.
  By the natural inclusion $\Sp(2\genus;\ZZ) \to \Homeo(T^{2\genus})$, we regard the element $s_{\genus}(\psi)$ as a  homeomorphism of $T^{2\genus}$.
  Let $M_{s_{\genus}(\psi)}$ be the mapping torus of $s_{\genus}(\psi) \in \Homeo(T^{2\genus})$.
  Since $M_{s_{\genus}(\psi)}$ is a $K(\ppi, 1)$-manifold, we have $\dim \HHH^2(\ppi) = \dim \HHH^2(M_{s_{\genus}(\psi)})$.
  Let us consider the cohomology long exact sequence
  \begin{align}\label{ex_seq:pair}
    \cdots \to \HHH^1(T^{2\genus}) \xrightarrow{\delta_1} \HHH^2(M_{s_{\genus}(\psi)}, T^{2\genus}) \to \HHH^2(M_{s_{\genus}(\psi)}) \to \HHH^2(T^{2\genus}) \xrightarrow{\delta_2} \HHH^3(M_{s_{\genus}(\psi)}, T^{2\genus}) \to \cdots.
  \end{align}
  Since $M_{s_{\genus}(\psi)}$ is a mapping torus, $\HHH^{n+1}(M_{s_{\genus}(\psi)}, T^{2\genus})$ is isomorphic to $\HHH^n(T^{2\genus})$ and the map $\delta_n$ is given by
  \[
    \mathrm{id}_{\HHH^n(T^{2\genus})} - s_l(\psi)^* \colon \HHH^n(T^{2\genus}) \to \HHH^n(T^{2\genus}) \cong \HHH^{n+1}(M_{s_{\genus}(\psi)}, T^{2\genus}).
  \]
  This, together with (\ref{ex_seq:pair}) and the fact that $H^2(T^{2\genus}) \cong \bigwedge\nolimits^2 H^1(T^{2\genus})$, implies that
  \[
    \dim \HHH^2(\ppi) = \dim \HHH^2(M_{s_l(\psi)}) = \dim \Ker(I_{2\genus} - s_\genus(\psi)) + \dim \Ker \left(I_{2\genus \choose 2} - \bigwedge\nolimits^2 s_\genus(\psi)\right).
  \]
  The computation of $\dim \HHH^2(\hG)$ is done in a similar manner.
\end{proof}

Let $\bG$ be the kernel of the natural epimorphism $\hG \to \ppi$.
Note that $\bG$ is isomorphic to the commutator subgroup of $\pi_1(\Sigma_{\genus})$.

\begin{lem}\label{lem:ineq_H1NG}
  The following inequality holds:
  \[
    \dim \HHH^1(\bG)^{\hG} \leq \dim \Ker \left(I_{2\genus \choose 2} - \bigwedge\nolimits^2 s_\genus(\psi)\right) -1.
  \]
\end{lem}

\begin{proof}
  We set $\HHH = \HHH_1(\Sigma_{\genus};\ZZ)$.
  Let $\iota \colon \HHH^1(\bG)^{\hG} \to \Hom\left(\bigwedge\nolimits^2 \HHH, \RR \right)$ be the map defined by
  \[
    \iota(h)(q(x)\wedge q(y)) = h([x, y]),
  \]
  where $x, y \in \pi_1(\Sigma_\genus)$ and $q \colon \pi_1(\Sigma_\genus) \to \HHH$ is the abelianization map.
 We claim that this map $\iota$ is well-defined. To verify this, let $h\in \HHH^1(\bG)^{\hG}$. By commutator calculus, $[x_1x_2,y]=x_1[x_2,y]x_1^{-1}\cdot [x_1,y]$ holds for every $x_1,x_2,y\in \pi_1(\Sigma_\genus)$. Since $h$ is $\hG$-invariant, this implies that
\[
h([x_1x_2,y])=h([x_1,y])+h([x_2,y]).
\]
In a similar manner to one above, we can see that
\[
h([xz,yw])=h([x,y])
\]
for every $x,y\in \pi_1(\Sigma_\genus)$ and every $z,w\in  \bG=[\pi_1(\Sigma_\genus),\pi_1(\Sigma_\genus)]$. Now, it is straightforward to confirm that $\iota$ is well-defined.
  Moreover, since $\bG$ is normally generated by $\{ [a_i, a_j] \}_{1 \leq i < j \leq 2\genus}$ in $\hG$, the map $\iota$ is injective.

  We set
  \[
    \Hom\left( \bigwedge\nolimits^2 \HHH, \RR \right)^{\bigwedge\nolimits^2 s_{\genus}(\psi)} = \left\{ h \in \Hom\left( \bigwedge\nolimits^2 \HHH, \RR \right) \; \middle| \; h \circ \bigwedge\nolimits^2 s_{\genus}(\psi) = h \right\}.
  \]
  Then the image of $\iota$ is contained in $\Hom\left( \bigwedge\nolimits^2 \HHH, \RR \right)^{\bigwedge\nolimits^2 s_{\genus}(\psi)}$.
  Indeed, for $1 \leq i < j \leq 2\genus$ and for $h \in \HHH^1(\bG)^{\hG}$, we have
  \begin{align*}
    \iota(h)\left( \bigwedge\nolimits^2 s_l(\psi) (q(a_i) \wedge q(a_j)) \right) &= h([f_* a_i, f_* a_j])\\
    & = h([a_{2\genus + 1} \cdot a_i \cdot a_{2\genus + 1}^{ -1 } , a_{2\genus + 1}\cdot a_j \cdot a_{2\genus + 1}^{ -1 } ]) \\
    & = h([a_i, a_j]) = \iota(h)(q(a_i) \wedge q(a_j)),
  \end{align*}
  where the second equality comes from the relation in (\ref{presentation_mapping_torus}) and the third equality comes from the $G$-invariance of $h$.

  Since $\dim \Hom\left( \bigwedge\nolimits^2 \HHH, \RR \right)^{\bigwedge\nolimits^2 s_{\genus}(\psi)}$ is equal to $\dim \Ker \left(I_{2\genus \choose 2} - \bigwedge\nolimits^2 s_\genus(\psi)\right)$, it suffices to show that the map
  \[
    \iota \colon \HHH^1(\bG)^{\hG} \to \Hom\left(\bigwedge\nolimits^2 \HHH, \RR \right)^{\bigwedge\nolimits^2 s_{\genus}(\psi)}
  \]
  is not surjective.
  We set $v_1 = q(a_1)\wedge q(a_2) + \cdots q(a_{2\genus - 1}) \wedge q(a_{2\genus}) \in \bigwedge\nolimits^2 \HHH$, then the map $\bigwedge\nolimits^2 s_\genus(\psi) \colon \bigwedge\nolimits^2 \HHH \to \bigwedge\nolimits^2 \HHH$ preserves $v_1$.
  Hence, for a suitable basis containing $v_1$, the dual $v_1^*$ is contained in $\Hom\left(\bigwedge\nolimits^2 \HHH, \RR \right)^{\bigwedge\nolimits^2 s_{\genus}(\psi)}$.
  However, $v_1^*$ is not contained in the image of $\iota$.
  Indeed, for every $h \in \HHH^1(\bG)^{\hG}$, we have
  \[
    \iota(h)(v_1) = h([a_1, a_2] \cdots [a_{2\genus-1}, a_{2\genus}]) = 0.
  \]
  Hence the map $\iota \colon \HHH^1(\bG)^{\hG} \to \Hom\left(\bigwedge\nolimits^2 \HHH, \RR \right)^{\bigwedge\nolimits^2 s_{\genus}(\psi)}$ is not surjective, and the lemma follows.
\end{proof}

\begin{proof}[Proof of Theorem $\ref{mapping torus thm}$]
  The group $\hG$ is  Gromov-hyperbolic  by Theorem \ref{thm:hyperbolic_mapping_torus} and $\ppi$ is amenable by Theorem \ref{amenable base} (4).
  Hence, Theorem \ref{main thm 2}, together with Lemma \ref{lem:dim_G_GG} (1), asserts that
  \[
    \dim ((\QQQ(\bG)^{\hG}/i^*\QQQ(\hG)) = \dim \HHH^2(\ppi) = \dim \Ker(I_{2\genus} - s_\genus(\psi)) + \dim \Ker \left(I_{2\genus \choose 2} - \bigwedge\nolimits^2 s_\genus(\psi)\right).
  \]

  By Theorem \ref{easy cor} and Lemma \ref{lem:dim_G_GG} (2), we obtain
  \[
    \dim (\QQQ(\bG)^{\hG} / (\HHH^1(\bG)^{\hG} + i^* \QQQ(\hG))) \leq \dim \HHH^2(\hG) = \dim \Ker(I_{2\genus} - s_\genus(\psi)) + 1.
  \]
  On the other hand, we have
  \begin{align*}
    \dim (\QQQ(\bG)^{\hG} / (\HHH^1(\bG)^{\hG} + i^* \QQQ(\hG))) &= \dim \HHH^2(\ppi) - \dim \HHH^1(\bG)^{\hG}\\
    &\geq \dim \Ker(I_{2\genus} - s_\genus(\psi)) + 1
  \end{align*}
  by Corollary \ref{cor B}, Lemma \ref{lem:dim_G_GG} (1), and Lemma \ref{lem:ineq_H1NG}.
\end{proof}

 As we mentioned in the introduction, we obtain an analog (Theorem~\ref{thm:F_n_torus}) of Theorem~\ref{mapping torus thm} in the free group setting.
For $n\in \NN$, let $\Aut(F_n)$ be the automorphism group of $F_n$.
Let $t_n\colon \Aut(F_n)\to \GL(n,\ZZ)$ be the representation induced by the action of $\Aut(F_n)$ on the abelianization of $F_n$.
 Then, the group $F_n \rtimes_{\psi} \ZZ$ naturally surjects onto $\ZZ^n \rtimes_{t_n(\psi)} \ZZ$ via abelianization of $F_n$. We say that an automorphism $\psi$ of $F_n$ is \emph{atoroidal} if it has no periodic conjugacy classes; namely, there does not exist a pair $(a,k)\in F_n\times \ZZ$ with $a\ne 1_{F_n}$ and $k\ne 0$ such that $\psi^k(a)$ is conjugate to $a$. Bestvina and Feighn \cite{BF1992} showed that $\psi\in \Aut(F_n)$ is atoroidal if and only if $F_n \rtimes_{\psi} \ZZ$ is  Gromov-hyperbolic .

\begin{thm}[Computations of dimensions for free-by-cyclic groups]\label{thm:F_n_torus}
  Let $n$ be an integer greater than $1$ and $\psi \in \Aut(F_n)$ an  atoroidal  automorphism.
  Set $\hG = F_n \rtimes_{\psi} \ZZ$ and let $\bG$ be the kernel of the surjection $G \to \ZZ^n \rtimes_{t_n(\psi)} \ZZ$ defined via the abelianization map $F_n \to \ZZ^n$.
  Then we have
  \[
    \dim \big( \QQQ(\bG)^\hG / i^* \QQQ(\hG)\big) = \dim \Ker(I_n - t_n(\psi)) + \dim \Ker \left(I_{n \choose 2} - \bigwedge\nolimits^2 t_n(\psi)\right)
  \]
  and
  \[
    \dim \big( \QQQ(\bG)^{\hG} / (\HHH^1(\bG)^{\hG} + i^* \QQQ(\hG)) \big) = \dim \Ker(I_n - t_n(\psi)),
  \]
  where 
  $\bigwedge^2 t_n(\psi)$ is the map induced by $t_n(\psi)$.
\end{thm}

\begin{proof}
  Since the quotient $\hG/\bG$ is isomorphic to $\ppi = \ZZ^n \rtimes_{t_n(\psi)} \ZZ$, we obtain
  \[
    \dim \big( \QQQ(\bG)^\hG / i^* \QQQ(\hG)\big) = \dim \Ker(I_n - t_n(\psi)) + \dim \Ker \left(I_{n \choose 2} - \bigwedge\nolimits^2 t_n(\psi)\right)
  \]
  in the same way as in the proof of Theorem \ref{mapping torus thm}.

  We set $\HHH = \HHH_1(F_n;\ZZ) = F_n / [F_n, F_n]$.
  As in Lemma \ref{lem:ineq_H1NG}, we can define a monomorphism $\iota \colon \HHH^1(\bG)^{\hG} \to \Hom\left(\bigwedge\nolimits\HHH, \RR \right)^{\bigwedge\nolimits^2  t_n(\psi) }$.
  Hence, together with Corollary \ref{cor B}, we obtain an inequality
  \[
    \dim \big( \QQQ(\bG)^{\hG} / (\HHH^1(\bG)^{\hG} + i^* \QQQ(\hG)) \big) \geq \dim \Ker(I_n - t_n(\psi)).
  \]
  On the other hand, we have
  \[
    \dim \big( \QQQ(\bG)^{\hG} / (\HHH^1(\bG)^{\hG} + i^* \QQQ(\hG)) \big) \leq \dim \HHH^2(\hG)
  \]
  by Theorem \ref{easy cor}.
  By the seven-term exact sequence (Theorem \ref{thm seven-term}) applied to the  short exact sequence $1 \to F_n \to \hG \to  \ZZ  \to 1$, we obtain that
  \[
    \HHH^2(\hG) \cong \HHH^1(\ZZ;\HHH^1(F_n)).
  \]
  By Lemma \ref{lem:first_coh}, $\HHH^1(\ZZ;\HHH^1(F_n))$ is isomorphic to
  \[
    \HHH^1(F_n)/ \mathrm{Im}(\mathrm{id}_{\HHH^1(F_n)} - \psi^*),
  \]
  where $\psi^* \colon \HHH^1(F_n) \to \HHH^1(F_n)$ is the pullback  of $\psi$.
  Hence we obtain that
  \[
    \dim \HHH^2(\hG) = \dim \HHH^1(\ZZ;\HHH^1(F_n)) = \dim \Ker (I_n - t_n(\psi))
  \]
  and the theorem follows.
\end{proof}

\subsection{Other examples} \label{subsec:other_ex}

It follows from Theorem \ref{easy cor} that $\HHH^2(\hG) = 0$ implies $\QQQ(\bG)^{\hG} = \HHH^1(\bG)^{\hG} + i^* \QQQ(\hG)$, and we provide several examples of groups $\hG$ with $\HHH^2(\hG) = 0$ in Subsection \ref{equiv subsection}.

As an application of \cite[Theorem 2.4]{MR1934011}, we provide another example of a group $G$ satisfying $\QQQ(\bG)^\hG = \HHH^1(\bG)^\hG + i^* \QQQ(\hG)$.


\begin{cor} \label{FujiwaraSoma}
Let $L$ be a hyperbolic link in $S^3$ such that the number of the connected components of $L$ is two.
Let $\hG$ be the link group of $L$ $($i.e., the fundamental group of the complement $S^3 \setminus L$ of $L)$ and $\bG$ the commutator subgroup of $\hG$.
Then we have $\QQQ(\bG)^{\hG} = \HHH^1(\bG)^{\hG} + i^* \QQQ(\hG)$.
\end{cor}

\begin{proof}
  By Theorem \ref{easy cor}, it suffices to show that the comparison map $c_{\hG} \colon \HHH_b^2(\hG) \to \HHH^2(\hG)$ is equal to zero.
  By using \cite[Theorem 2.4]{MR1934011}, we have $\Im (c_{\hG}) \neq \HHH^2(\hG)$.
 Since the number of the connected components of $L$ is two, the second cohomology group $\HHH^2(\hG)$ is isomorphic to $\RR$.
  Hence we obtain $\Im (c_{\hG}) = 0$.
\end{proof}

Here we provide other examples $(G,N)$ such that $\HHH^2(G) \ne 0$ and $\QQQ(N)^G = \HHH^1(N)^G + i^* \QQQ(G)$.

\begin{example} \label{eg:free_prod}
Let $n \in \NN$.
For $i=1,2,\ldots,n$, let $H_i$ be a boundedly $2$-acyclic group and assume that $\HHH^2(H_1) \ne 0$ (for example, we can take $H_1=\ZZ^2$).
Set  $G=H_1 \ast H_2 \ast \cdots \ast H_n$ and $N = [G, G]$.
Then we have $\HHH^2(G) = \HHH^2(H_1) \oplus \HHH^2(H_2) \oplus \cdots \oplus \HHH^2(H_n) \ne 0$ but the comparison map $c_G \colon \HHH^2_b(G) \to \HHH^2(G)$ is  the zero map.
It follows from Theorem \ref{easy cor} that $\QQQ(N)^G / (\HHH^1(N)^G + i^* \QQQ(G)) = 0$.
\end{example}

\begin{cor}\label{cor:circle_bundle}
  Let $E \to \Sigma_{\genus}$ be a non-trivial circle bundle over a closed oriented surface of genus $\genus > 1$.
  For the fundamental group $\hG = \pi_1(E)$ and its normal subgroup $\bG = [G, G]$, we have
  \[
    \dim \big( \QQQ(\bG)^\hG / i^* \QQQ(\hG)\big) = \genus(2\genus - 1) \; \textrm{and} \; \dim \big(\QQQ(\bG)^\hG / (\HHH^1(\bG)^\hG + i^* \QQQ(\hG))\big) = 0.
  \]
\end{cor}

\begin{remark}
  \begin{enumerate}
    \item The dimension of $\QQQ(G)$ (and hence, the dimension of $\QQQ(\bG)^{\hG}$) is the cardinal of the continuum since $G$ surjects onto the surface group $\pi_1(\Sigma_{\genus})$.
    \item The cohomology group $\HHH^2(\hG)$ is non-zero. In fact, the dimension of $\HHH^2(\hG) \cong \HHH^2(E)$ is equal to $2\genus$.
  \end{enumerate}
\end{remark}

\begin{proof}[Proof of Corollary \textup{\ref{cor:circle_bundle}}]
  Let $n$ be the Euler number of the bundle $E \to \Sigma_{\genus}$.
  Note that $n$ is non-zero since the bundle is non-trivial  (see Theorem 11.16 of \cite{Fr}).
  Since the group $G$ has a  presentation
  \begin{align}\label{presentation_circle_bundle}
    \hG &= \pi_1(E) = \left\langle a_1, \cdots, a_{2\genus + 1}\; \middle| \;
  \begin{gathered}
      {[a_1, a_2] \cdots [a_{2\genus-1}, a_{2\genus}] = a_{2\genus + 1}^{-n}},  \\
      [a_i, a_{2\genus + 1}] = 1_{\hG} \text{ for every } 1 \leq i \leq 2\genus
  \end{gathered}
   \right\rangle, \nonumber
  \end{align}
  the abelianization $\ppi = G/N$ is isomorphic to $\ZZ^{2\genus} \times (\ZZ/n\ZZ)$.
  Hence we have $\dim \HHH^2(\ppi) = \genus(2\genus - 1)$.
  By the relation $[a_i, a_{2\genus + 1}] = 1_{\hG}$ for each $i$ and the fact that $\bG$ is normally generated by $\{ [a_i, a_j] \}_{1 \leq i < j \leq 2\genus + 1}$ in $\hG$,
we obtain $\dim \HHH^1(\bG)^{\hG} = \genus(2\genus + 1) - 2\genus = \genus(2\genus - 1)$  by an argument similar to the proof of Proposition \ref{prop:inv_hom_surface_group}. 
  Hence Corollary \ref{cor B} asserts that
  \[
    \dim \big(\QQQ(\bG)^\hG / (\HHH^1(\bG)^\hG + i^* \QQQ(\hG))\big) \leq \dim \HHH^2(\ppi) - \dim \HHH^1(\bG)^{\hG} = 0.
  \]
  Since $\bG$ is the commutator subgroup of $G$, the space $\HHH^1(\bG)^{\hG}$ injects into $\QQQ(\bG)^\hG / i^* \QQQ(\hG)$.
  Hence we have
  \[
    \dim \big( \QQQ(\bG)^\hG / i^* \QQQ(\hG)\big) \geq \genus(2\genus - 1).
  \]
  On the other hand, Theorem \ref{thm A} asserts that
  \[
    \dim \big( \QQQ(\bG)^\hG / i^* \QQQ(\hG)\big) \leq \HHH^2(\ppi) =  \genus(2\genus - 1).
  \]
  This completes the proof.
\end{proof}

For elements $r_1, \cdots, r_m \in G$, we write $\llangle r_1, \cdots, r_m \rrangle$ to mean the normal subgroup of $G$ generated by $r_1, \cdots, r_m$.

\begin{cor} \label{cor FFF}
Let $r_1, \cdots, r_m \in [F_n, [F_n, F_n]]$ and set
\[G = F_n / \llangle r_1, \cdots, r_m \rrangle.\]
Then we have $\QQQ( [\hG,\hG] )^{\hG} = \HHH^1( [\hG,\hG] )^\hG + i^* \QQQ(\hG)$.
\end{cor}

\begin{proof}
Let $q$ be the natural projection $F_n \to G$. Then the image of the monomorphism $q^* \colon \HHH^1( [\hG,\hG] )^G \to \HHH^1( [F_n,F_n] )^{F_n}$ is the space of $F_n$-invariant homomorphisms $f \colon  [F_n,F_n]  \to \RR$ satisfying $f(r_1) = \cdots = f(r_m) = 0$. Since every $F_n$-invariant homomorphism of $ [F_n,F_n] $ vanishes on $[F_n, [F_n, F_n]]$, we  conclude that $q^*$ is an isomorphism, and hence we have $\dim \HHH^1( [\hG,\hG] )^G = n (n-1) / 2$.
Since $\Gamma = G /  [\hG,\hG]  = \ZZ^n$, we have $\dim \HHH^2(\Gamma) = n(n-1)/2$. Hence Corollary \ref{cor B} implies that $\QQQ( [\hG,\hG] )^G / (\HHH^1( [\hG,\hG] )^G + i^* \QQQ(G))$ is trivial.
\end{proof}

\begin{remark}
Suppose that $N$ is the commutator subgroup of $G$.
As will be seen in Corollaries \ref{cor extend homomorphism 1} and \ref{cor 5.3}, the sum $\HHH^1(N)^G + i^* \QQQ(G)$ is actually a direct sum in this case, and the map $\HHH^1(N)^G \to \QQQ(N)^G / i^* \QQQ(G)$ is an isomorphism.
Hence, if $G$ is a group provided in Corollary \ref{cor FFF} and $N$ is the commutator subgroup of $G$, then the basis of $\QQQ(N)^G / i^* \QQQ(G)$ is provided by the $G$-invariant homomorphism $\alpha'_{i,j} \colon N \to \RR$ for $1 \le i< j \le n$, which is the homomorphism induced by $\alpha_{i,j} \colon  [F_n,F_n]  \to \RR$ described in Lemma \ref{lem F_n}.
\end{remark}

As an example of a pair $(G,N)$ satisfying $\QQQ(N) \ne \HHH^1(N)^G + i^* \QQQ(G)$, we provide a certain family of one-relator groups.
Recall that a {\it one-relator group} is a group isomorphic to $F_n / \llangle r \rrangle$ for some positive integer $n$ and an element $r$ of $F_n$. 

\begin{thm}\label{thm:one-relator}
Let $n$ and $k$ be integers at least $2$, and $r$ an element of 
 $[F_n,F_n] \setminus [F_n, [F_n, F_n]]$. 
Set $\hG = F_n / \llangle r^k \rrangle$ and $\bG =  [\hG, \hG]. $
Then
\[\dim \left( \QQQ(\bG)^{\hG} / (\HHH^1(\bG)^{\hG} + i^* \QQQ(\hG)) \right) = 1.\]
\end{thm}

 We observe that $r \in [F_n,F_n] \setminus [F_n, [F_n, F_n]]$ is equivalent to the existance of $f_0 \in \HHH^1( [F_n, F_n] )^{F_n}$ with $f_0(r) \ne 0$.
 We use this observation in the proof of Theorem~\ref{thm:one-relator}. 

\begin{proof}[ Proof of Theorem~$\ref{thm:one-relator}$ ]
By Newman's Spelling theorem \cite{MR222152}, every one-relator group with torsion is hyperbolic, and hence $\hG$ is hyperbolic.
Indeed, $r$ does not belong to $\llangle r^k \rrangle$ since $f_0(x)$ belongs to $k f_0(r) \ZZ$ for every element $x$ of $\llangle r^k \rrangle$.
Since $\ppi = \hG/\bG$ is abelian, we  see that  
 $\ppi$ is boundedly $3$-acyclic.  By Corollary \ref{cor B}, it suffices to see
\[\dim \left( \QQQ(\bG)^{\hG} / (\HHH^1(\bG)^{\hG} + i^* \QQQ(\hG)) \right) = \dim \HHH^2(\ppi) - \dim \HHH^1(\bG)^{\hG} = 1.\]

Since $r^k \in  [F_n, F_n] $, we have  $\Gamma = \ZZ^n$, and  $\dim \HHH^2(\ppi) = n(n-1) / 2$.
Hence
 it only remains to show that
\begin{eqnarray} \label{one relator}
\dim \HHH^1(N)^G = \frac{n(n-1)}{2} - 1.
\end{eqnarray}
Let $q \colon F_n \to G = F_n / \llangle r^k \rrangle$ be the natural quotient. Then $q$ induces a monomorphism $q^* \colon \HHH^1(N)^G \to \HHH^1( [F_n, F_n] )^{F_n}$. As is the case of the proof of Proposition \ref{prop:inv_hom_surface_group}, it is straightforward to show that the image of $q^* \colon \HHH^1(N)^G \to \HHH^1( [F_n, F_n] )^{F_n}$ is the space of $F_n$-invariant homomorphisms $f \colon  [F_n, F_n]  \to \RR$ such that $f(r) = 0$.
Since there exists an element $f_0$ of $\HHH^1( [F_n, F_n] )^{F_n}$ with $f_0(r) \ne 0$, we  conclude that the codimension of the image of $q^* \colon \HHH^1(N)^G \to \HHH^1( [F_n, F_n] )^{F_n}$ is $1$. This implies \eqref{one relator}, and hence completes the proof.
\end{proof}

 After the authors closed up this work, we have obtained a generalization of Theorem \ref{thm:one-relator}; see Theorem 11.15 of \cite{coarse_group}.

\begin{remark}\label{rem:one-relator}
Let $k$ be a positive integer. Here we construct a finitely presented group $G$ satisfying
\begin{eqnarray} \label{dim k}
\dim \left( \QQQ( [\hG, \hG] )^G / (\HHH^1( [\hG, \hG] )^G + i^* \QQQ(G)) \right) = k.
\end{eqnarray}
Let $F_{2k} = \langle a_1, \cdots, a_{2k}\rangle$ be a free group and define the group $G$ by
\[G = \langle a_1, \cdots, a_{2k} \; | \; [a_1, a_2]^2, \cdots, [a_{2k-1}, a_{2k}]^2 \rangle.\]
Set $H = \langle a_1, a_2 \; | \; [a_1, a_2]^2 \rangle$. Then $G$ is the $k$-fold free product of $H$. Since $H$ is a one-relator group with torsion, $H$ is hyperbolic. Since a finite free product of hyperbolic groups is hyperbolic, $G$ is hyperbolic. Hence the comparison map $\HHH^2_b(G) \to \HHH^2(G)$ is surjective.

Let $q \colon F_{2k} \to G$ be the natural quotient. Then the image of the monomorphism $q^* \colon \HHH^1( [\hG, \hG] )^G \to \HHH^1( [F_{2k}, F_{2k}] )^{F_{2k}}$ consists of the $F_{2k}$-invariant homomorphisms $\varphi \colon  [F_{2k}, F_{2k}]  \to \RR$ such that $\varphi([a_{2i-1}, a_{2i}]) = 0$ for $i = 1, \cdots, k$.
Therefore Corollary \ref{cor B} implies  \eqref{dim k}. 
\end{remark}

\section{Cohomology classes induced by the flux homomorphism}\label{flux section}

First, we review the definition of the (volume) flux homomorphism (for instance, see \cite{Ban97}).

Let $\diff (M, \Omega)$ denote the group of diffeomorphisms on an $m$-dimensional smooth manifold $M$ which preserve a volume form $\Omega$ on $M$, $\diff_0(M, \Omega)$ the identity component of $\diff(M, \Omega)$, and $\tdiff_0(M, \Omega)$ the universal cover of $\diff_0(M, \Omega)$.
Then the {\it $($volume$)$ flux homomorphism} $\tflux_\Omega \colon \tdiff_0 (M, \Omega) \to \HHH^{m-1}(M)$ is defined by 
\[\tflux_\Omega ([\{ \psi^t \}_{t \in [0,1]}]) = \int_0^1 [\iota_{X_t} \Omega] dt,\]
where $X_t = \dot{\psi}_t$. The image of $\pi_1(\diff_0(M, \Omega))$ under $\tflux_\Omega$ is called the {\it flux group} of the pair $(M,\Omega)$, and denoted by $\Gamma_\Omega$.
The flux homomorphism $\tflux_\Omega$ descends a homomorphism
\[\flux_\Omega\colon \diff_0(M,\Omega) \to \HHH^{m-1}(M) / \Gamma_\Omega.\]
These homomorphisms are fundamental objects in theory of diffeomorphism groups, and have been extensively studied by several researchers (for example, see \cite{KKM06}, \cite{Ish14}).

As we wrote in Subsection \ref{intro flux},  Proposition \ref{prop diff} is essentially due to \cite{KM}; we state the proof for the reader's convenience.
\begin{proof}[Proof of Proposition $\ref{prop diff}$]
Suppose that the pair $(G,N)$ of groups is $(\diff_0(M, \Omega), \Ker (\flux_\Omega))$ or $(\widetilde{\diff}_0(M,\Omega), \Ker (\tflux_\Omega))$.
Since the kernels of the homomorphisms $\flux_\Omega$ and $\tflux_\Omega$ are perfect (see \cite{Th} and \cite{Ban}, see also Theorems 4.3.1 and 5.1.3 of \cite{Ban97}), we have  $\HHH^1(N) = 0$.
Hence this proposition follows from the five-term exact sequence (Theorem \ref{thm five-term}).
\end{proof}

To prove (1) of Theorem \ref{thm diff},
we use Py's Calabi quasimorphism $\qm_P\colon\Ker(\flux_\Omega)\to\RR$, which was introduced in \cite{Py06}.
For an oriented closed surface whose genus $\genus$ is at least $2$ and a volume form $\Omega$ on $M$,
Py constructed a $\diff_0(M, \Omega)$-invariant homogeneous quasimorphism $\qm_P\colon\Ker(\flux_\Omega)\to\RR$ on $\Ker(\flux_\Omega)$.

\begin{proof}[Proof of Theorem $\ref{thm diff}$]
First, we prove (1).
Suppose that $\Sigma_{\genus}$ is an oriented closed surface whose genus $\genus$ is at least $2$, and let $\Omega$ be its volume form.
Since in this case $\Gamma_\Omega$ is trivial  (as mentioned just after Theorem \ref{thm diff}), the two flux homomorphisms $\flux_\Omega$ and $\tflux_\Omega$ coincide.

Set $G = \diff_0(\Sigma_{\genus}, \Omega)$ and
 $N = \Ker(\flux_\Omega)$. 
Since $N$ is perfect (\cite[Th\'{e}or\`{e}m II.6.1]{Ban}), we have  $\HHH^1(N) = \HHH^1(N)^G = 0$.
Since $G /N$ is abelian, Theorem \ref{easy cor} implies that
\[\QQQ(N)^G / i^* \QQQ(G) = \QQQ(\bG)^{\hG} / (\HHH^1(\bG)^{\hG} + i^* \QQQ(\hG)) \cong  \Im(\flux_\Omega^*) \cap \Im(c_G).\]
Since Py's Calabi quasimorphism $\qm_P$ is not extendable to $G = \diff_0(\Sigma_{\genus}, \omega)$ (\cite[Theorem 1.11]{KK}), we  conclude that $\QQQ(N)^G / i^* \QQQ(G)$ is not trivial.
Hence, we  conclude that $\flux_\Omega^\ast \circ \tae_4^{-1} \circ \tau_{/b}([\qm_P])\in\Im(\flux_\omega^*) \cap \Im (c_G)$ is non-zero.

Now we show (2).
Suppose that $m = 2$.
The case that $M$ is a $2$-sphere is clear since $\HHH^1(M) = 0$, and hence the flux homomorphisms are trivial.
The case $M$ is a torus follows from the fact that both $\flux_\Omega$ and $\tflux_\Omega$ have section homomorphisms.
Hence, by Proposition \ref{virtual split KKMM1}, we have $\Im(\flux_\Omega^*) \cap \Im(c_G) \cong \QQQ(N)^G / i^* \QQQ(G) = 0$.

Suppose that $m \ge 3$.
Then Proposition \ref{prop Fathi} mentioned below implies that $\flux_\Omega$ has a section homomorphism.
Hence, by Proposition \ref{virtual split KKMM1}, we have $\Im(\flux_\Omega^*) \cap \Im(c_G) \cong \QQQ(N)^G / i^* \QQQ(G) = 0$.
This completes the proof.
\end{proof}

\begin{prop}[Proposition 6.1 of \cite{Fa}] \label{prop Fathi}
Let $m$ be an integer at least $3$, $M$ an $m$-dimensional differential manifold, and $\Omega$ a volume form on $M$. Then there exists a section homomorphism of the reduced flux homomorphism $\flux_\Omega \colon \diff_0(M, \Omega) \to \HHH^{m-1}(M,\Omega) / \Gamma_\Omega$.
In addition, there exists a section homomorphism of $\tflux_\Omega \colon \tdiff_0(M, \Omega) \to \HHH^{m-1}(M,\Omega)$.
\end{prop}


The idea of Theorem \ref{thm diff} is also useful in (higher-dimensional) symplectic geometry.
For notions in symplectic geometry, for example, see \cite{Ban97} and \cite{PR}.
For a symplectic manifold $(M,\omega)$, let $\Ham(M,\omega)$ denote the group of Hamiltonian diffeomorphisms with compact support.
For an exact symplectic manifold $(M,\omega)$, let $\cal_\omega\colon\Ham(M,\omega)\to\RR$ denote the Calabi homomorphism.
We note that the map $\cal_{\omega}^*$ is injective, where $\cal_\omega^*\colon \HHH^2(\RR;\RR) \to \HHH^2(\Ham(M,\omega);\RR)$ is the homomorphism induced by $\cal_\omega$.
Indeed, because $\Ker(\cal_\omega)$ is perfect (\cite{Ban}), we can prove the injectivity of $\cal_\omega^*$ similarly to the proof of Proposition \ref{prop diff}.
Then, we have the following theorem.

\begin{thm} \label{calabi hom}
For an exact symplectic manifold $(M,\omega)$, every non-trivial element of $\Im(\cal_\omega^*)$ cannot be represented by a bounded $2$-cochain.
\end{thm}

Note that $\cal_\omega\colon\Ham(M,\omega)\to\RR$ has a section homomorphism.
Indeed, for a (time-independent) Hamiltonian function whose integral over $M$ is $1$ and its Hamiltonian flow $\{\phi^t\}_{t\in\RR}$, the homomorphism $t \mapsto \phi^t$ is a section of the Calabi homomorphism $\cal_\omega$.
Hence the proof of Theorem \ref{calabi hom} is similar to Theorem \ref{thm diff}.

\section{Proof of the main theorem}\label{sec:pf_of_five-term}

The goal in this section is to prove Theorem \ref{main thm}, which is the five-term exact sequence of the cohomology of groups relative to the bounded subcomplex.

\begin{notation}
Throughout this section, $\cM$ denotes a Banach space equipped with the norm $\| \cdot \|$ and an isometric $G$-action whose restriction to $\bG$ is trivial.
  For a non-negative real number $D \geq 0$, the symbol $v \underset{D}{\approx} w$ means that the inequality $\|v - w\| \leq D$ holds.
  For functions $f, g \colon S \to \cM$ on a set $S$, the symbol $f \underset{D}{\approx} g$ means that the condition $f(s) \underset{D}{\approx} g(s)$ holds for every $s \in S$.
\end{notation}

\subsection{$\bG$-quasi-cocycle}
To define the map $\tau_{/b} \colon \HHH_{/b}^1(\bG;V)^{\hG} \to \HHH_{/b}^2(\ppi;V)$ in Theorem \ref{main thm}, it is convenient to introduce the notion called the $\bG$-quasi-cocycle.
First, we recall the definition of quasi-cocycles.

\begin{definition}\label{def:quasicocycle}
  Let $\hG$ be a group and $V$ a left $\RR[\hG]$-module with a $\hG$-invariant norm $\| \cdot \|$.
  A function $F \colon \hG \to V$ is called a \textit{quasi-cocycle} if there exists a non-negative number $D$ such that
  \begin{align*}
    F(g_1g_2) \underset{D}{\approx} F(g_1) + g_1 \cdot F(g_2)
  \end{align*}
  holds for every $g_1, g_2 \in \hG$.
  Such a smallest $D$ is called the \textit{defect} of $F$ and denoted by $D(F)$.
  Let $\rQQQ Z(\hG; V)$ denote the $\RR$-vector space of all quasi-cocycles on $\hG$.
\end{definition}

\begin{remark}
  If we need to specify the $\hG$-representation $\rho$, we use the symbol $\rQQQ Z(\hG;\rho, V)$ instead of $\rQQQ Z(\hG; V)$.
\end{remark}

We introduce the concept of $\bG$-quasi-cocycles, which is a generalization of the concept of partial quasimorphisms introduced in \cite{EP06} (see also \cite{MVZ}, \cite{Ka16}, \cite{Ki}, \cite{BK}, and \cite{KKMM1}).

\begin{definition}\label{def:N-quasimorphism}
  Let $\bG$ be a normal subgroup of $\hG$.
  A function $F \colon \hG \to V$ is called an \textit{$\bG$-quasi-cocycle} if there exists a non-negative number $D''$ such that
  \begin{align}\label{eq:Nquasi_condition}
    F(n g) \underset{D''}{\approx} F(n) + F(g) \ \text{ and } \ F(g n) \underset{D''}{\approx} F(g) + g \cdot F(n)
  \end{align}
  hold for every $g \in \hG$ and $n \in \bG$.
  Such a smallest $D''$ is called the \textit{defect} of the $\bG$-quasi-cocycle $F$ and denoted by $D''(F)$.
  Let $\rQQQ Z_{\bG}(\hG;V)$ denote the $\RR$-vector space of all $\bG$-quasi-cocycles on $\hG$.
\end{definition}



If the $\hG$-action on $V$ is trivial, then a quasi-cocycle is also called a {\it $V$-valued quasimorphism}.
    In this case, we use the symbol $\rQQQ(\hG;V)$ instead of $\rQQQ Z(\hG;V)$ to denote the space of  $V$-valued quasimorphisms.
    A $V$-valued quasimorphism $F$ is said to be \textit{homogeneous} if the condition $F(g^k) = k \cdot F(g)$ holds for every $g \in \hG$ and every $k \in \ZZ$.
    The homogenization of $V$-valued quasimorphisms is well-defined as in the case of ($\RR$-valued) quasimorphisms. We write $\QQQ(G; V)$ to denote the space of $V$-valued homogeneous quasimorphisms.

Recall that in our setting the restriction of the $\hG$-action on $V$ to $\bG$ is always trivial.
Then a left $\hG$-action on $\QQQ(\bG;V)$ is defined by
    \[
      ({}^{g}\qm)(n) = g \cdot \qm(g^{-1} n g)
    \]
for every $g \in \hG$ and every $n \in \bG$. We call an element of $\QQQ(N; V)^G$ a {\it $\hG$-equivariant $V$-valued homogeneous quasimorphism}.

\begin{remark}\label{rem:V-valued_quasimorphisms}
Note that an element $\qm \in \QQQ(\bG;V)$ belongs to $\QQQ(\bG;V)^{\hG}$ if and only if the condition
    \begin{align*}
      g \cdot \qm(n) = \qm(g n g^{-1})
    \end{align*}
    holds for every $g \in \hG$ and every $n \in \bG$.
This is the reason why we call an element of $\QQQ(\bG ; V)^{\hG}$ $\hG$-equivariant.
\end{remark}

\begin{remark}\label{remark:restriction_homogenization}
 The isomorphism $\HHH^1_{/b}(N ; V) \to \QQQ(N; V)$ given by the homogenization is compatible with the $\hG$-actions.
In particular, this isomorphism induces an isomorphism between $\HHH^1_{/b}(N ; V)^{\hG} \to \QQQ(N; V)^{\hG}$. 
\end{remark}

    The elements of $\QQQ(\bG;V)^{\hG} = \HHH_{/b}^1(\bG;V)^{\hG}$ are $\hG$-invariant (as cohomology classes).
    However, respecting the condition $g \cdot \qm(n) = \qm(g n g^{-1})$ for $\qm \in \QQQ(\bG;V)^{\hG}$, we call the elements of $\QQQ(\bG;V)^{\hG}$ \textit{$\hG$-equivariant $V$-valued homogeneous quasimorphisms}.

\begin{lem}\label{lemma:restriction_homogenization}
  Let $\bG$ be a normal subgroup of $\hG$ and $V$ a left $\RR[\hG]$-module.
  Assume that the induced $\bG$-action on $V$ is trivial.
  Then, for an $\bG$-quasi-cocycle $F \in \rQQQ Z_{\bG}(\hG;V)$, there exists a bounded cochain $b \in C_b^1(\hG;V)$ such that the restriction $(F + b)|_{\bG}$ is in $\QQQ(\bG;V)^{\hG}$.
\end{lem}

\begin{proof}
  By the definition of $\bG$-quasi-cocycles, the restriction $F|_{\bG} \colon \bG \to V$ is a quasimorphism.
  Let $\overline{F|_{\bG}}$ be the homogenization of $F|_N$. Then the map
  \begin{align*}
    b' = \overline{F|_{\bG}} - F|_{\bG} \colon \bG \to V
  \end{align*}
  is bounded.
  Define $b \colon \hG \to V$ by
  \[
    b(g) = \begin{cases}
      b'(g) & g \in \bG \\
      0 & \text{otherwise}.
    \end{cases}
  \]
  Then the map $b$ is also bounded.
  Set $\Phi = F + b$, then $\Phi|_{\bG} = (F + b)|_{\bG} = \overline{F|_{\bG}}$.
  Since $\Phi$ is an $\bG$-quasi-cocycle, we have
  \begin{align*}
    ({}^{g}\Phi)(n) = g \cdot \Phi(g^{-1} n g) \underset{D''(\Phi)}{\approx} \Phi(g \cdot g^{-1}n g) - \Phi(g) = \Phi(ng) - \Phi(g) \underset{D''(\Phi)}{\approx} \Phi(n)
  \end{align*}
  for $g \in \hG$ and $n \in \bG$.
  Hence the difference ${}^{g}\Phi - \Phi$ is in $C_b^1(\bG;V)$.
  Since $({}^{g}\Phi)|_{\bG}$ and $\Phi|_{\bG}$ are homogeneous quasimorphisms, we have ${}^{g}\Phi|_{\bG} - \Phi|_{\bG} = 0$, and this implies that the element $\Phi|_{\bG} = (F + b)|_{\bG}$ belongs to $\QQQ(\bG;V)^{\hG}$.
\end{proof}

If $V$ is the trivial $\hG$-module $\RR$, then $\bG$-quasi-cocycles are also called $\bG$-quasimorphisms (this word was first introduced in \cite{Ka17}).
In this case, Lemma \ref{lemma:restriction_homogenization} is as follows.
\begin{cor}\label{cor:restriction_homogenization}
  Let $\bG$ be a normal subgroup of $\hG$.
  For an $\bG$-quasimorphism $F \in \rQQQ_{\bG}(\hG)$, there exists a bounded cochain $b \in C_b^1(\hG)$ such that the restriction $(F + b)|_{\bG}$ is in $\QQQ(\bG)^{\hG}$.
\end{cor}

\subsection{The map $\tau_{/b}$}
Now we proceed to the proof of Theorem \ref{main thm}. The goal in this subsection is to construct the map $\tau_{/b} \colon \HHH^1_{/b}(\bG)^\hG \to \HHH^2_{/b}(\hG)$.
Here we only present the proofs in the case where the coefficient module $V$ is the trivial module $\RR$.
When $V \neq \RR$, the proofs remain to work without any essential change (see Remarks \ref{remark:restriction_homogenization}, \ref{remarks2}, and \ref{remarks3}).

First, we define the map $\tau_{/b} \colon \HHH_{/b}^1(\bG)^{\hG} \to \HHH_{/b}^2(\ppi)$.
Let $1 \to \bG \xrightarrow{i} \hG \xrightarrow{p} \ppi \to 1$ be a group extension.
 As a special case of Remark \ref{remark:restriction_homogenization}, we have isomorphisms $\HHH^1_{/b}(\bG) \to \QQQ(\bG)$ and $\HHH^1_{/b}(\bG)^{\hG} \to \QQQ(\bG)^{\hG}$.
By using these, we identify $\HHH_{/b}^1(\bG)$  and  $\HHH^1_{/b}(\bG)^{\hG}$ with $\QQQ(\bG)$  and   $\QQQ(\bG)^{\hG}$,  respectively. 


Let $\hQQQ_{\bG}(\hG) = \hQQQ_{\bG}(\hG;\RR)$ be the $\RR$-vector space of all $\bG$-quasimorphisms whose restrictions to $\bG$ are homogeneous quasimorphisms on $\bG$, that is,
\[
  \hQQQ_{\bG}(\hG) = \{ F \colon \hG \to \RR \; | \; \text{$F$ is an $\bG$-quasimorphism such that $F|_N \in \QQQ(\bG)^\hG$}\} \subset \rQQQ_{\bG}(\hG).
\]
By definition, the restriction of the domain defines a map
\[
  i^* \colon \hQQQ_{\bG}(\hG) \to \QQQ(\bG)^{\hG}.
\]

\begin{remark}\label{remarks2}
  In the case that the $\hG$-action on $V$ is non-trivial, we need to replace the space $\hQQQ_{\bG}(\hG)$ by
  \[
    \hQQQ Z_{\bG}^1(\hG;V) = \{ F \colon \hG \to V \; | \; \text{$F$ is an $\bG$-quasi-cocycle such that $F|_{\bG} \in \QQQ(\bG ; V)^{\hG}$} \}.
  \]

\end{remark}

\begin{lem}\label{lemma:surjectivity}
  The map $i^* \colon \hQQQ_{\bG}(\hG) \to \QQQ(\bG)^{\hG}$ is surjective.
\end{lem}

\begin{proof}
  Let $s \colon \ppi \to \hG$ be a  set-theoretic  section of $p$ satisfying $s(1_\ppi) = 1_{\hG}$.
  For $\qm \in \QQQ(\bG)^{\hG}$, define a map $F_{\qm, s} \colon \hG \to \RR$ by
  \begin{align*}
    F_{\qm, s}(g) = \qm(g \cdot sp(g)^{-1})
  \end{align*}
  for $g \in \hG$.
  Then the equality $F_{\qm, s}|_{\bG} = \qm$ holds since $sp(n) = 1_{\hG}$ for every $n \in \bG$.
  Moreover, the map $F_{\qm, s}$ is an $\bG$-quasimorphism.
  Indeed, we have
  \begin{align*}
    F_{\qm, s}(n g) &= \qm(n g \cdot sp(n g)^{-1}) = \qm(n g \cdot sp(g)^{-1}) \\
    &\underset{D(\qm)}{\approx} \qm(n) + \qm(g \cdot sp(g)^{-1}) = F_{\qm, s}(n) + F_{\qm, s}(g)
  \end{align*}
  and
  \begin{align*}
    F_{\qm, s}(g n) &= F_{\qm, s}(g n g^{-1} g) \underset{D(\qm)}{\approx} F_{\qm, s}(g n g^{-1}) + F_{\qm, s}(g)\\
    &= \qm(g n g^{-1}) + F_{\qm, s}(g) = \qm(n) + F_{\qm, s}(g) = F_{\qm, s}(n) + F_{\qm, s}(g)
  \end{align*}
by the definition of quasimorphisms and the $\hG$-invariance of $\qm$. This means $i^* (F_{\qm,s}) = f$, and hence the map $i^*$ is surjective.
\end{proof}

\begin{lem}\label{lemma:dF_lift_indep_bdd_error}
  For $F \in \hQQQ_{\bG}(\hG)$ and for $g_i, g_i' \in \hG$ satisfying $p(g_i) = p(g_i') \in \ppi$, the following condition holds:
  \begin{align*}
    \delta F(g_1, g_2) \underset{4D''(F)}{\approx} \delta F(g_1', g_2').
  \end{align*}
\end{lem}

\begin{proof}
  By the assumption, there exist  $n_1, n_2 \in \bG$ satisfying $g_1' = n_1 g_1$ and $g_2' = g_2 n_2$.
  Therefore we have
  \begin{align*}
    \delta F(g_1', g_2') &= F(g_2 n_2) - F(n_1 g_1 g_2 n_2) + F(n_1 g_1)\\
    &\underset{4D''(F)}{\approx} F(g_2) + F(n_2) - (F(n_1) + F(g_1 g_2) + F(n_2)) + F(n_1) + F(g_1)\\
    &= \delta F(g_1, g_2). \qedhere
  \end{align*}
\end{proof}

For $F \in \hQQQ_{\bG}(\hG)$ and a section $s \colon \ppi \to \hG$ of $p$, we set $\alpha_{F, s} = s^* \delta F \in C^2(\ppi)$.
By Lemma \ref{lemma:dF_lift_indep_bdd_error}, the element $[\alpha_{F, s}] \in C_{/b}^2(\ppi) = C^2(\ppi)/C_b^2(\ppi)$ is independent of the choice of the section $s$.
Therefore we set $\alpha_F = [\alpha_{F, s}] \in C_{/b}^2(\ppi)$.

\begin{lem}\label{lemma:cocycle_in_relative}
  The cochain $\alpha_F$ is a cocycle on $C_{/b}^{\bullet}(\ppi)$.
\end{lem}

\begin{proof}
  It suffices to show that the coboundary $\delta \alpha_{F, s}$ belongs to $C_b^3(\ppi)$.
  For $f, g, h \in \ppi$, we have
  \begin{align*}
    &\delta \alpha_{F, s}(f,g,h)\\
    &= \delta F(s(g), s(h)) - \delta F(s(fg), s(h)) + \delta F(s(f), s(gh)) - \delta F(s(f), s(g))\\
    &\underset{8D''(F)}{\approx} \delta F(s(g), s(h)) - \delta F(s(f)s(g), s(h)) \\
    & \hspace{5cm}+ \delta F(s(f), s(g)s(h)) - \delta F(s(f), s(g))\\
    &= \delta (\delta F)(s(f), s(g), s(h)) = 0
  \end{align*}
  by Lemma \ref{lemma:dF_lift_indep_bdd_error}.
\end{proof}

By Lemmas \ref{lemma:surjectivity} and \ref{lemma:cocycle_in_relative}, we obtain a map
\begin{align}\label{map:quasi-cocycle_to_relative_coh}
  \hQQQ_{\bG}(\hG) \to \HHH_{/b}^2(\ppi) ; F \mapsto [\alpha_{F}].
\end{align}

\begin{lem}\label{lemma:tau_well-def}
  The cohomology class $[\alpha_F] \in \HHH_{/b}^2(\ppi)$ depends only on the restriction $F|_{\bG}$.
\end{lem}

\begin{proof}
  Let $s \colon \ppi \to \hG$ be a section of $p$ and $\Phi$ an element of $\hQQQ_{\bG}(\hG)$ satisfying $\Phi|_{\bG} = F|_{\bG}$. Then, for every $g, h \in \ppi$, we have
  \begin{align*}
    (\alpha_{F, s} - \alpha_{\Phi, s})(g, h) &= \delta F(s(g), s(h)) - \delta \Phi(s(g), s(h)) \\
    &= F(s(h)) - F(s(g)s(h)) + F(s(g))\\
    &\hspace{3cm} - (\Phi(s(h)) - \Phi(s(g)s(h)) + \Phi(s(g)))\\
    &= \delta(F \circ s)(g,h) - \delta(\Phi \circ s)(g,h)\\
    &\hspace{1cm} + F(s(gh)) - F(s(g)s(h)) - (\Phi(s(gh)) - \Phi(s(g)s(h))).
  \end{align*}
  Since $F$ and $\Phi$ are $\bG$-quasimorphisms, we have
  \begin{align*}
    F(s(gh)) - F(s(g)s(h)) \underset{D''(F)}{\approx} F(s(gh)s(h)^{-1}s(g)^{-1}),\\
    \Phi(s(gh)) - \Phi(s(g)s(h)) \underset{D''(\Phi)}{\approx} \Phi(s(gh)s(h)^{-1}s(g)^{-1}).
  \end{align*}
  Together with the equality $F(s(gh)s(h)^{-1}s(g)^{-1}) = \Phi(s(gh)s(h)^{-1}s(g)^{-1})$, we have
  \begin{align*}
    \alpha_{F, s} - \alpha_{\Phi, s} \underset{D''(F) + D''(\Phi)}{\approx} \delta (F \circ s - \Phi \circ s),
  \end{align*}
  and this implies $[\alpha_F] = [\alpha_{\Phi}] \in \HHH_{/b}^2(\ppi)$.
\end{proof}

By Lemma \ref{lemma:tau_well-def}, the map defined in (\ref{map:quasi-cocycle_to_relative_coh}) descends to a map $\tau_{/b} \colon \QQQ(\bG)^{\hG} \to \HHH_{/b}^2(\ppi)$, that is,
the map $\tau_{/b}$ is defined by
\[
  \tau_{/b}(\qm) = [\alpha_F],
\]
where $F$ is an element of $\hQQQ_{\bG}(\hG)$ satisfying $F|_{\bG} = \qm$.
Under the isomorphism $\QQQ(\bG)^{\hG} \cong \HHH_{/b}^1(\bG)^{\hG}$, we obtain the map
\[
  \tau_{/b} \colon \HHH_{/b}^1(\bG)^{\hG} \to \HHH_{/b}^2(\ppi).
\]

\subsection{Proof of the exactness}
Now we proceed to the proof of the exactness of the sequence
\begin{align}\label{seq:five-term}
  0 \to \HHH_{/b}^1(\ppi) \xrightarrow{p^*} \HHH_{/b}^1(\hG) \xrightarrow{i^*} \QQQ(\bG)^{\hG} \xrightarrow{\tau_{/b}} \HHH_{/b}^2(\ppi) \xrightarrow{p^*} \HHH_{/b}^2(\hG),
\end{align}
where we identify $\QQQ(\bG)^{\hG}$ with $\HHH_{/b}^1(\bG)^{\hG}$.

\begin{prop}\label{proposition:ex_0}
  Sequence $(\ref{seq:five-term})$ is exact at $\HHH_{/b}^1(\ppi)$ and $\HHH_{/b}^1(\hG)$.
\end{prop}

\begin{remark}\label{remarks3}
  In the case of the trivial real coefficients, this proposition is well known (see \cite{Ca}).
  Indeed, the spaces $\HHH_{/b}^1(\ppi)$ and $\HHH_{/b}^1(\hG)$ are isomorphic to $\QQQ(\ppi)$ and $\QQQ(\hG)$, respectively, and the exactness of the sequence
\[ 0 \to \QQQ(\Gamma) \to \QQQ(G) \to \QQQ(N)^G\]
follows from the homogeneity of the elements of $\QQQ(\ppi)$.
  However, in general, the spaces $\HHH_{/b}^1(\ppi;V)$ and $\HHH_{/b}^1(\hG;V)$ are not isomorphic to the space of $V$-valued homogeneous quasimorphisms $\QQQ(\ppi;V)$ and $\QQQ(\hG;V)$, respectively.
  Therefore, we present proof of Proposition \ref{proposition:ex_0}
   which can be modified to the case of non-trivial coefficients without any essential change.

\end{remark}

\begin{proof}[Proof of Proposition $\ref{proposition:ex_0}$]
We first show the exactness at $\HHH_{/b}^1(\ppi)$. Let $a \in \HHH_{/b}^1(\ppi)$ and suppose $p^* a = 0$. Let $f \in C^1(\ppi)$ be a representative of $a$. Since $p^* a = 0$ in $\HHH^1_{/b}(G)$, there exists $c \in \RR \cong C^0(\Gamma)$ such that $p^* f - \delta c = p^* f$ is bounded.
Since $p$ is surjective, we  see that $f$ is bounded, and hence $a = 0$. This means the exactness at $\HHH_{/b}^1 (\ppi)$.

  Next we prove the exactness of $\HHH_{/b}^1(\hG)$. Since the map $p \circ i$ is zero, the composite $i^* \circ p^*$ is also zero.
  For $a \in \HHH_{/b}^1(\hG)$ satisfying $i^*a = 0$, it follows from Lemma \ref{lemma:restriction_homogenization} that there exists a representative $f \in C^1(\hG)$ of $a$ satisfying $f|_{\bG} = 0$.
  For a section $s \colon \ppi \to \hG$ of $p$, set $f_s = s^* f \colon \ppi \to \RR$.
  Then $f_s$ is a quasimorphism on $\ppi$.
  Indeed, since $f$ is a quasimorphism on $\hG$, we have
  \begin{align*}
    f_s(g_1 g_2) &= f(s(g_1 g_2)) = f(s(g_1 g_2)s(g_2)^{-1}s(g_1)^{-1}s(g_1)s(g_2))\\
    &\underset{D(f)}{\approx} f(s(g_1 g_2)s(g_2)^{-1}s(g_1)^{-1}) + f(s(g_1)s(g_2)) = f(s(g_1)s(g_2))\\
    &\underset{D(f)}{\approx} f(s(g_2)) + f(s(g_1)) = f_s(g_2) + f_s(g_1)
  \end{align*}
  by the triviality $f|_{\bG} = 0$.
  Hence the cochain $f_s$ is a cocycle on $C_{/b}^{ 1 }(\ppi)$, and let $a_s \in \HHH_{/b}^1(\ppi)$ denote the relative cohomology class represented by $f_s$.
  For $g \in \hG$, we have
  \begin{align*}
    p^*f_s(g) = f(sp(g)) = f(sp(g) g^{-1} g) \underset{D(f)}{\approx} f(sp(g) g^{-1}) + f(g) = f(g).
  \end{align*}
  Therefore, the cochain $p^* f_s$ is equal to $f$ as relative cochains on $\hG$, and this implies that the equality $p^*a_s = a$ holds.
\end{proof}

\begin{prop}\label{proposition:ex_1}
  Sequence $(\ref{seq:five-term})$ is exact at $\QQQ(\bG)^{\hG}$.
\end{prop}

\begin{proof}
  Note that representatives of first relative cohomology classes of $\hG$ are quasimorphisms, and that quasimorphisms on $\hG$ are $\bG$-quasimorphisms.
  For every $a \in \HHH_{/b}^1(\hG)$, there exists a representative $F \in C^1(\hG)$ of $a$ such that the restriction $F|_{\bG}$ is a homogeneous quasimorphism on $\bG$ by Lemma \ref{lemma:restriction_homogenization}.
  By the definition of the map $\tau_{/b} \colon \QQQ(\bG)^{\hG} \to \HHH_{/b}^2(\ppi)$, we have
  \begin{align*}
    \tau_{/b}(i^*(a)) = \tau_{/b}(F|_{\bG}) = [\alpha_F].
  \end{align*}
  Since the cochain $F$ is a quasimorphism, the cocycle $\alpha_F \in C_{/b}^2(\ppi)$ is equal to zero.
  Therefore we have $\tau_{/b}(i^*(a)) = [\alpha_F] = 0$.

  Suppose that $\qm \in \QQQ(\bG)^{\hG}$ satisfies $\tau_{/b}(\qm) = 0$.
  By Lemma \ref{lemma:surjectivity}, we obtain $F \in \hQQQ_{\bG}(\hG)$ satisfying $F|_{\bG} = \qm$.
  Let $s \colon \ppi \to \hG$ be a section of $p$.
  The triviality of $[\alpha_F] = \tau_{/b}(\qm) = 0$ implies that there exist  $\beta \in C^1(\ppi)$ and $b \in C_b^2(\ppi)$ satisfying
  \begin{align*}
    \alpha_{F, s} - \delta \beta = b.
  \end{align*}
  For $g_i \in \hG$, we have
  \begin{align*}
    \delta F(g_1, g_2) \underset{4D''(F)}{\approx} \delta F(sp(g_1), sp(g_2)) = \alpha_{F, s}(p(g_1), p(g_2))
  \end{align*}
  by Lemma \ref{lemma:dF_lift_indep_bdd_error}.
  Hence we have
  \begin{align*}
    \delta(F - p^*\beta)(g_1, g_2) \underset{4D''(F)}{\approx} (\alpha_{F, s} - \delta \beta)(p(g_1), p(g_2)) = p^*b(g_1, g_2).
  \end{align*}
  Since the cochain $b$ is bounded, the cochain
   $F - p^*\beta$ is a cocycle in  $C_{/b}^1(\hG)$.
  Moreover, since $F|_{\bG}=\qm$, the restriction $(F - p^*\beta+\beta(1_\Gamma))|_{\bG}$ is equal to $\qm$.
  Therefore we have $i^*([F-p^* \beta+\beta(1_\Gamma)]) = \qm$, and this implies the exactness.
\end{proof}

\begin{prop}\label{proposition:ex_2}
  Sequence $(\ref{seq:five-term})$ is exact at $\HHH_{/b}^2(\ppi)$.
\end{prop}

\begin{proof}
  For $\qm \in \QQQ(\bG)^{\hG}$, we have $F \in \hQQQ_{\bG}(\hG)$ satisfying $F|_{\bG} = \qm$ by Lemma \ref{lemma:surjectivity}.
  Then a representative of $p^*(\tau_{/b}(\qm)) \in \HHH_{/b}^2(\hG)$ is given by $p^*\alpha_{F, s} \in C^2(\hG)$ for some section $s \colon \ppi \to \hG$ of $p$.
  For $g_i \in \hG$, we have
  \begin{align*}
    p^*\alpha_{F, s}(g_1, g_2) = s^* \delta F(p(g_1), p(g_2)) = \delta F(sp(g_1), sp(g_2)) \underset{4D''(F)}{\approx} \delta F(g_1, g_2)
  \end{align*}
  by Lemma \ref{lemma:dF_lift_indep_bdd_error}.
  This implies $p^*(\tau_{/b}(\qm)) = 0$.

  For $a \in \HHH_{/b}^2(\ppi)$ satisfying $p^* a = 0$, let $\alpha \in C^2(\ppi)$ be a representative of $a$.
  We can assume that the cochain satisfies
  \begin{align}\label{triviality_condition}
    \alpha(1_{\ppi}, 1_{\ppi}) = 0.
  \end{align}
  Indeed, if $\alpha(1_{\ppi}, 1_{\ppi}) = c \in \RR$, then the cochain $\alpha - c$ satisfies (\ref{triviality_condition}) and is also a representative of $a$ since the constant function $c$ is bounded.
  Note that the cocycle condition of $C_{/b}^{\bullet}(\ppi)$ implies that there exists a non-negative constant $D$ such that the condition
  \[
    \delta \alpha \underset{D}{\approx} 0
  \]
  holds.
  Hence, for $\gamma_1, \gamma_2 \in \ppi$, we have
  \begin{align*}
    0 \underset{D}{\approx} \delta \alpha (\gamma_1, 1_{\ppi}, \gamma_2) = \alpha(1_{\ppi}, \gamma_2) - \alpha(\gamma_1, 1_{\ppi}).
  \end{align*}
  In particular, we have
  \begin{align}\label{quasi-normalized_condition}
    \alpha(1_{\ppi}, \gamma) \underset{D}{\approx} \alpha(1_{\ppi}, 1_{\ppi}) = 0 \  \text{ and } \ \alpha(\gamma, 1_{\ppi}) \underset{D}{\approx} \alpha(1_{\ppi}, 1_{\ppi}) = 0
  \end{align}
  for every $\gamma \in \ppi$.
  The equality $p^*a = 0$ implies that there exists $\beta \in C^1(\hG)$ and  a  non-negative constant $D'$ satisfying
  \begin{align}\label{quasi-normalized_condition2}
    p^* \alpha - \delta \beta \underset{D'}{\approx} 0.
  \end{align}
  Define a cochain $\zeta \colon \hG \to \RR$ by
  \begin{align}\label{eq:zeta}
    \zeta(g) = \beta(g) - \alpha(p(g), 1_{\ppi}),
  \end{align}
  then it is an $\bG$-quasimorphism.
  Indeed, by using $p(n) = 1_{\ppi}$, we have
  \begin{align*}
    \delta \zeta(n, g) &= \delta \beta(n, g) - (\alpha(p(g), 1_{\ppi}) - \alpha(p(g), 1_{\ppi}) + \alpha(1_{\ppi}, 1_{\ppi}))\\
    &\underset{D}{\approx} (\delta \beta - p^* \alpha)(g, n) \underset{D'}{\approx} 0,
  \end{align*}
  and
  \begin{align*}
    \delta \zeta(g, n) &= \delta \beta(g, n) - (\alpha(1_{\ppi}, 1_{\ppi}) - \alpha(p(g), 1_{\ppi}) + \alpha(p(g), 1_{\ppi}))\\
    &\underset{D}{\approx} (\delta \beta - p^* \alpha)(g, n) \underset{D'}{\approx} 0
  \end{align*}
  by (\ref{quasi-normalized_condition}) and (\ref{quasi-normalized_condition2}).
  By Lemma \ref{lemma:restriction_homogenization}, there exists a bounded cochain $b \in C_b^1(\hG)$ such that the restriction $(\zeta + b)|_{\bG}$ is in $\QQQ(\bG)^{\ppi}$.
  Set $\Phi = \zeta + b \in \hQQQ_{\bG}(\hG)$, then a representative of $\tau_{/b}(\Phi|_{\bG})$ is given by $\alpha_{\Phi, s}$ for some section $s \colon \ppi \to \hG$ of $p$.
  For $g_1, g_2 \in \ppi$, we have
  \begin{align*}
    (\alpha_{\Phi, s} - \alpha)(g_1, g_2) &= (\delta \Phi - p^*\alpha)(s(g_1), s(g_2))\\
    &\underset{D'}{\approx} (\delta \Phi - \delta \beta)(s(g_1), s(g_2))
  \end{align*}
  by (\ref{quasi-normalized_condition2}).
  By (\ref{eq:zeta}), we have
  \begin{align*}
    (\Phi - \beta)(g) = (\zeta + b - \beta)(g) = b(g) - \alpha(p(g), 1_{\ppi}).
  \end{align*}
  Together with (\ref{quasi-normalized_condition}) and the boundedness of $b$, the cochain $\Phi - \beta \colon \hG \to \RR$ is bounded.
  Hence the cochain $\alpha_{\Phi, s} - \alpha$ is also bounded, and this implies the equality $a = [\alpha_{\Phi}] = \tau_{/b}(\Phi|_{\bG})$.
  Therefore the proposition follows.
\end{proof}

\begin{proof}[Proof of Theorem $\ref{main thm}$]
  The exactness is obtained from Propositions \ref{proposition:ex_0}, \ref{proposition:ex_1}, and \ref{proposition:ex_2}.
  The commutativities of the first, second, and fourth squares are obtained from the cochain level calculations.
  The commutativity of the third square follows from the definition of the map $\tau_{/b}$ and Proposition \ref{prop:transgression_group_coh} below.
\end{proof}

\begin{prop}[{\cite[Proposition 1.6.6]{MR2392026}}]\label{prop:transgression_group_coh}
  Let $1 \to \bG \to \hG \to \ppi \to 1$ be an exact sequence and $V$ an $\ppi$-module.
  For a $\hG$-invariant homomorphism $f \in \HHH^1(\bG;V)^{\hG}$, there exists a map $F \colon \hG \to V$ such that the restriction $F|_{\bG}$ is equal to $f$ and the coboundary $\delta F$ descends to a group two cocycle $\alpha_F \in C^2(\ppi;V)$, that is, the equality $p^* \alpha_F = \delta F$ holds.
  Then the map $\tau \colon \HHH^1(\bG;V)^{\hG} \to \HHH^2(\ppi;V)$ in the five-term exact sequence of group cohomology is obtained by $\tau(f) = [\alpha_F]$.
\end{prop}

We conclude this section by the following applications of Theorem \ref{main thm} to the extendability of $\hG$-invariant homomorphisms.

\begin{prop} \label{prop extend homomorphism}
Let $\ppi = \hG/\bG$.
Assume that $\HHH_b^2(\ppi) = 0$ and $f \colon \bG \to \RR$ a $\hG$-invariant homomorphism on $\bG$. If $f$ is extended to a quasimorphism on $G$, then $f$ is extended to a homomorphism on $G$.
\end{prop}
\begin{proof}
  Note that the assumption $\HHH_b^2(\ppi) = 0$ implies that the map $\HHH^2(\ppi) \to \HHH_{/b}^2(\ppi)$ is injective.
  By the diagram chasing on (\ref{diagram_coh_qm_rel}), the proposition holds.
\end{proof}

This proposition immediately implies the following corollary:

\begin{cor} \label{cor extend homomorphism 1}
Let $\ppi = \hG/\bG$.
Assume that $\HHH^2_b(\ppi) = 0$ and $\bG$ is a subgroup of $[\hG,\hG]$. Then every non-zero $\hG$-invariant homomorphism $f \colon \bG \to \RR$ cannot be extended to $\hG$ as a quasimorphism. Namely, $\HHH^1(N)^G \cap i^* \QQQ(G) = 0$.
\end{cor}
\begin{proof}
Assume that a homomorphism $f \colon N \to \RR$ can be extended to $G$ as a quasimorphism. Then Proposition \ref{prop extend homomorphism} implies that there exists a homomorphism $f' \colon G \to \RR$ with $f'|_N = f$. Since $f'$ vanishes on $[G,G]$, we have $f = f'|_N = 0$.
\end{proof}

 Without the assumption $\HHH^2_b(\ppi) = 0$, 
there exists a $\hG$-invariant homomorphism which is extendable to $\hG$ as a quasimorphism such that it is \emph{not} extendable to $G$ as a (genuine) homomorphism.
To see this, let $\hG = \widetilde{\Homeo}_+(S^1)$ and $\bG = \pi_1(\Homeo_+(S^1))$.
Then, Poincar\'e's rotation number $\rho \colon \widetilde{\Homeo}_+(S^1) \to \RR$ is an extension of the homomorphism $\pi_1(\Homeo_+(S^1)) \cong \ZZ \hookrightarrow \RR$.
However, this homomorphism $\pi_1(\Homeo(S^1)) \to \RR$ cannot be extendable to $\widetilde{\Homeo}_+(S^1)$ as a homomorphism since $\widetilde{\Homeo_+}(S^1)$ is perfect.

\section{Proof of equivalences of $\scl_G$ and $\scl_{G,N}$} \label{sec:scl}

The goal of this section is to prove Theorem \ref{main thm 3.1}.
 In this section, in order to specify the domain of a quasimorphism, we use the symbols $D_G$ and $D_N$ to denote the defect of a quasimorphism on $G$ and $N$, respectively. 
The main tool in this section is  the Bavard duality theorem  for $\scl_{G,N}$, which are proved by the first, second, fourth, and fifth authors:

\begin{thm}[Bavard duality theorem for stable mixed commutator lengths, \cite{KKMM1}] \label{thm Bavard}
Let $N$ be a normal subgroup of a group $G$. Then, for every $x \in [G,N]$, the following equality holds:
\[\scl_{G,N}(x) = \frac{1}{2} \sup_{f \in \QQQ(N)^G - \HHH^1(N)^G} \frac{|f(x)|}{ D_N(f) }.\]
 Here we set that the supremum in the right-hand side of the above equality to be  zero if $\QQQ(N)^G = \HHH^1(N)^G$.
\end{thm}

This theorem yields the following criterion to show the equivalence of $\scl_{G,N}$ and $\scl_G$:

\begin{prop} \label{prop equivalence criterion}
 Let $C$ be a real number such that for every $f \in \QQQ(N)^G$ there exists $f' \in \QQQ(G)$ satisfying $f'|_N - f \in \HHH^1(N)^G$ and $ D_G(f')  \le C \cdot  D_N(f) $.
Then for every $x \in [G,N]$,
\[ \scl_G(x) \le \scl_{G,N}(x) \le C \cdot \scl_G(x).\]
\end{prop}
 The existence of such $C$ as in the assumption of Proposition~\ref{prop equivalence criterion} is equivalent to saying that $\QQQ(N)^G = \HHH^1(N)^G + i^* \QQQ(G)$. See Subsection~\ref{subsec:proof_of_scl} for details.

\begin{proof}
Let $x \in [G,N]$. It is clear that $\scl_G(x) \le \scl_{G,N}(x)$. Let $\varepsilon > 0$.
Then Theorem \ref{thm Bavard} implies that there exists $f \in \QQQ(N)^G$ such that
\[ \scl_{G,N}(x) - \varepsilon \le \frac{f(x)}{2 D_N(f) }.\]
 By assumption,
there exists $f' \in \QQQ(G)$ such that $f'' = f'|_N - f \in \HHH^1(N)^G$ and $ D_G(f')  \le C \cdot  D_N(f) $.
Since $f''$ is a $G$-invariant homomorphism and $x \in [G,N]$, we have $f''(x) = 0$, and hence $f'(x) = f(x)$.
Hence we have
\[ \scl_{G,N}(x) - \varepsilon \le \frac{f(x)}{2  D_N(f) } \le C \cdot \frac{f'(x)}{2 D_G(f') } \le C \cdot \scl_G(x).\]
Here we use Theorem \ref{thm Bavard} to prove the last inequality.
 Since $\varepsilon$ is an arbitrary number, we complete the proof.
\end{proof}

In the proofs of (2) and (3) of Theorem \ref{main thm 3.1}, we use  the following corollary of Proposition \ref{prop equivalence criterion}:

\begin{cor}\label{cor equivalence criterion}
Assume that $\QQQ(N)^G = \HHH^1(N)^G + i^* \QQQ(G)$, and that there exists $C \ge 1$ such that $ f'  \in  \QQQ(G) $ implies that $ D_N(f')  \le C \cdot  D_N(f'|_N) $.
Then for every $x \in [G,N]$,
\[ \scl_G(x) \le \scl_{G,N}(x) \le C \cdot \scl_G(x).\]
\end{cor}
\begin{proof}
Let $f \in \QQQ(N)^G$.
 Then, by the assumption that $\QQQ(N)^G = \HHH^1(N)^G + i^* \QQQ(G)$, there exists $f' \in \QQQ(G)$ such that $f'|_N - f$ is a $G$-invariant  homomorphism.
Note that $  D_N(f'|_N)  =   D_N(f) $.
Indeed, for every $a,b \in N$, we have
\[ f(ab) - f(a) - f(b) = f'(ab) - f'(a) - f'(b)\]
since $f'|_N - f$ is a homomorphism. Hence we have $C \cdot   D_N(f)  = C \cdot   D_N(f'|_N)  \ge   D_G(f') $. Hence Proposition \ref{prop equivalence criterion} implies that
\[ \scl_G(x) \le \scl_{G,N}(x) \le C \cdot \scl_G(x)\]
for every $x \in [G,N]$.
\end{proof}

\subsection{Proof of (1) of Theorem \ref{main thm 3.1}}\label{subsec:proof_of_scl}

The main difficulty in the proof of Theorem \ref{main thm 3.1} is to prove Theorem \ref{thm Q(N)^G Banach} mentioned below. Note that the defect $  D_N $ defines a seminorm on $\QQQ(N)^G$, and its kernel is $\HHH^1(N)^G$.

\begin{thm} \label{thm Q(N)^G Banach}
The normed space $(\QQQ(N)^G / \HHH^1(N)^G ,  D_N )$ is a Banach space.
\end{thm}

To show this theorem, we recall some concepts introduced in \cite{KKMM1}.
Let $\rQQQ_{\bG}(\hG) = \rQQQ_{\bG}(\hG;\RR)$ denote the $\RR$-vector space of $\bG$-quasimorphisms (see Definition \ref{def:N-quasimorphism}  and Table \ref{table:symbol_qm}).
We call $f \in \rQQQ_{\bG}(\hG)$ an {\it $\bG$-homomorphism} if $D''(f) = 0$, and let $\HHH^1_{\bG}(\hG)$ denote the space of $\bG$-homomorphisms on $\hG$.
It is clear that the defect $D''$ is a seminorm on $\rQQQ_{\bG}(\hG)$, and in fact,  the norm  space $\rQQQ_{\bG}(\hG)/ \HHH^1_{\bG}(\hG)$ is complete:

\begin{prop}[{\cite[Corollary 3.6]{KKMM1}}] \label{prop Q_N(G) Banach}
The normed space $(\rQQQ_N(G)/\HHH_{\bG}^1(\hG) , D'')$ is a Banach space.
\end{prop}

A quasimorphism $f \colon \bG \to \RR$ is said to be \textit{$\hG$-quasi-invariant} if the number
\[D'(f) = \sup_{g \in \hG  , x \in \bG } |f(gxg^{-1}) - f(x)|\]
is finite. Let $\rQQQ(\bG)^{\QQQ \hG}$ denote the space of $\hG$-quasi-invariant quasimorphisms on $\bG$. The function $ D_N  + D'$, which assigns $ D_N (f) + D'(f)$ to $f \in \rQQQ(\bG)^{\QQQ \hG}$ defines a seminorm on $\rQQQ(\bG)^{\QQQ \hG}$.
 For an $\bG$-quasimorphism $f$ on $\hG$,  the restriction $f|_{\bG}$ is a $\hG$-quasi-invariant quasimorphism (Lemma 2.3 of \cite{KKMM1}). Conversely, for every $\hG$-quasi-invariant quasimorphism $f$ on $\bG$, there exists an $\bG$-quasimorphism $f' \colon \hG \to \RR$ satisfying $f' |_{\bG} = f$ (Proposition 2.4 of \cite{KKMM1}).
 We summarize the concepts and symbols on quasimorphisms in Table \ref{table:symbol_qm}.

\begin{table}[hbtp]
   \caption{the concepts and symbols on quasimorphisms}
  \label{table:symbol_qm}
  \centering
  \begin{tabular}{c|c|c|c}
    \hline
    concept & defect  &  definition  & vector space \\
    \hline \hline
    quasimorphism on $G$  & $D$  & $f(g_1 g_2) \approx_{D} f(g_1)+f(g_2) $ & $\rQQQ(\hG)$  \\
\hline
  \begin{tabular}{c}
$G$-quasi-invariant \\ quasimorphism on $N$
\end{tabular}  & $D,D'$  & 
\begin{tabular}{c}
  $f(x_1x_2)\approx_D f(x_1)+f(x_2)$, \\ $f(gxg^{-1}) \approx_{D'} f(x) $ 
\end{tabular}
  & $\rQQQ(\bG)^{\QQQ \hG}$  \\
 \hline
   $N$-quasimorphism on $G$ & $D''$  &
  \begin{tabular}{c}
$f(gx) \approx_{D''} f(g)+f(x) $, \\ $f(xg) \approx_{D''} f(x)+f(g) $
\end{tabular}  & $\rQQQ_N(G)$  \\
    \hline
  \end{tabular}   
\end{table}

\begin{lem} \label{lem 5.3}
The normed space $(\rQQQ(\bG)^{\QQQ \hG} / \HHH^1(\bG)^{\hG} ,  D_N  + D')$ is a Banach space.
\end{lem}
\begin{proof}
In what follows, we will define bounded operators
\[A \colon \rQQQ_{\bG}(\hG)/ \HHH^1_{\bG}(\hG) \to \rQQQ(\bG)^{\QQQ \hG} / \HHH^1(\bG)^{\hG},\]
\[B \colon \rQQQ(\bG)^{\QQQ \hG} / \HHH^1(\bG)^{\hG} \to \rQQQ_{\bG}(\hG)/ \HHH^1_\bG(\hG)\]
such that $A \circ B$ is the identity of $\rQQQ(\bG)^{\QQQ \hG} / \HHH^1(\bG)^{\hG}$. First, we define $A$ by the restriction, {\it i.e.,} $A(f) = f|_{\bG}$.
 Then, the operator norm of $A$ is at most $3$ since $D_N \leq D''$ and $D' \leq 2D''$. Indeed, $D_N \leq D''$ follows by definition, and $D' \leq 2D''$ follows from the estimate
\[ f(\hg x\hg^{-1}) + f( \hg ) \approx_{D''(f)} f(\hg x) \approx_{D''(f)} f(\hg ) + f(x).\]
for $g \in \hG$ and $x \in \bG$.

Let $S$ be a subset of $\hG$ such that $1_{\hG} \in S$ and the map
\[S \times \bG \to \hG, \; (s,x) \mapsto sx\]
is bijective. For an $f \in \rQQQ(\bG)^{\QQQ \hG}$, define a function $B(f) \colon \hG \to \RR$ by $B(f)(sx) = f(x)$ for $s \in S$ and $x \in \bG$. Then $B(f)$ is an $\bG$-quasimorphism on $\hG$ satisfying $D''(B(f)) \le  D_N (f) + D'(f)$.
Hence the map $B$ induces a bounded operator $\rQQQ(\bG)^{\QQQ \hG} / \HHH^1(\bG)^{\hG} \to \rQQQ_{\bG}(\hG)/ \HHH^1_{\bG}(\hG)$ whose operator norm is at most $1$, and we  conclude that $\rQQQ(\bG)^{\QQQ \hG} / \HHH^1(\bG)^{\hG}$ is isomorphic to $B (\rQQQ(\bG)^{\QQQ \hG} / \HHH^1(\bG)^{\hG})$.
Proposition \ref{prop Q_N(G) Banach} implies that $\rQQQ_{\bG}(\hG)/ \HHH^1_{\bG}(\hG)$ is a Banach space.
Therefore it suffices to see that $B(\rQQQ(\bG)^{\QQQ \hG} / \HHH^1(\bG)^{\hG})$ is a closed subset of $\rQQQ_{\bG}(\hG)/ \HHH^1_{\bG}(\hG)$, but this is deduced from the following well-known fact (Lemma \ref{lem 5.3.1}).
\end{proof}

\begin{lem} \label{lem 5.3.1}
Let $X$ be a topological subspace of a Hausdorff space $Y$. If $X$ is a retract of $Y$, then $X$ is a closed subset of $Y$.
\end{lem}
\begin{proof}
Let $r \colon Y \to X$ be a retraction of the inclusion map $i \colon X \to Y$. Since $X = \{ y \in Y \; | \; i \circ r(y) = y\}$ and $Y$ is a Hausdorff space, we  conclude that $X$ is a closed subset of $Y$.
\end{proof}

\begin{proof}[Proof of Theorem $\ref{thm Q(N)^G Banach}$]
For $n \in \ZZ$ and $x \in \bG$, define a function $\alpha_{n,x} \colon \rQQQ(\bG)^{\QQQ \hG} \to \RR$ by
\[\alpha_{n,x}(f) = f(x^n) - n \cdot f(x).\]
Since $|\alpha_{n,x}(f)| \le (n-1) D_N(f) $, we  conclude that $\alpha_{n,x}$ is bounded with respect to the norm $ D_N  + D'$, and hence $\alpha_{n,x}$ induces a bounded operator $\overline{\alpha}_{n,x} \colon \rQQQ(\bG)^{\QQQ \hG} / \HHH^1(\bG)^{\hG} \to \RR$.
Since
\[\QQQ(\bG)^{\hG} / \HHH^1(\bG)^{\hG} = \bigcap_{n \in \ZZ, x \in \bG} \Ker(\overline{\alpha}_{n,x}),\]
the space $\QQQ(\bG)^{\hG} / \HHH^1(\bG)^{\hG}$ is a closed subspace of the Banach space $\rQQQ(\bG)^{\QQQ \hG} / \HHH^1(\bG)^{\hG}$ (see Lemma \ref{lem 5.3}). Since $D' = 0$ on $\QQQ(\bG)^{\hG}$ (Lemma 2.1 of \cite{KKMM1}), we conclude that the normed space $(\QQQ(\bG)^{\hG} / \HHH^1(\bG)^{\hG},  D_\bG)$ is a Banach space.
\end{proof}

\begin{proof}[Proof of $(1)$ of Theorem $\ref{main thm 3.1}$]
It is clear that $\scl_{\hG}(x) \le \scl_{\hG,\bG}(x)$ for every $x \in [\hG,\bG]$. Hence it suffices to show that there exists $C>1$ such that for every $x \in [\hG,\bG]$, the inequality $\scl_{\hG,\bG}(x) \le C \cdot \scl_{\hG}(x)$ holds.

It follows from Theorem \ref{thm Q(N)^G Banach} that $(\QQQ(\hG)/\HHH^1(\hG),  D_\hG)$ and $(\QQQ(\bG)^{\hG} / \HHH^1(\bG)^{\hG},  D_\bG)$ are Banach spaces.
Let $T \colon \QQQ(\hG) / \HHH^1(\hG) \to \QQQ(\bG)^{\hG} / \HHH^1(\bG)^{\hG}$ be the bounded operator induced by the restriction $\QQQ(\hG) \to \QQQ(\bG)^{\hG}$.
Let $X$ be the kernel of $T$. Then $T$ induces a bounded operator
\[\overline{T} \colon (\QQQ(\hG) / \HHH^1(\hG)) / X \to \QQQ(\bG)^{\hG} / \HHH^1(\bG)^{\hG}.\]
The assumption $\QQQ(\bG)^{\hG} = \HHH^1(\bG)^{\hG} + i^* \QQQ(\hG)$ implies that the map $T$ is surjective, and hence we  see that $\overline{T}$ is a bijective bounded operator. By the open mapping theorem, we  conclude that the inverse $S = \overline{T}^{-1}$ is a bounded operator, and set $C = \| S\| + 1$, where $\| S \|$ denotes the operator norm of $S$. Then for every $[f] \in \QQQ(\bG)^{\hG} / \HHH^1(\bG)^{\hG}$, there exists $f' \in \QQQ(\hG)$ such that $ D_G (f') \le C \cdot  D_N (f)$ and $f' |_{\bG} - f \in \HHH^1(\bG)^{\hG}$.
Hence Proposition \ref{prop equivalence criterion} implies that \[ \scl_G \le \scl_{G,N} \le C \cdot \scl_G\]
on $[G,N]$. This completes the proof of (1) of Theorem \ref{main thm 3.1}.
\end{proof}



\subsection{Proof of (2) of Theorem \ref{main thm 3.1}}

In this subsection, we prove (2) of Theorem \ref{main thm 3.1}.

Here we recall the definition of the seminorm on $\HHH^n_b(\hG)$. The space $C^n_b(G ; \RR)$ of bounded $n$-cochains on $G$ is a Banach space with respect to the $\infty$-norm
\[ \| \varphi\|_\infty  = \sup \{ |\varphi(x_1, \cdots, x_n)| \; | \; x_1, \cdots, x_n \in G\}.\]
Since $\HHH^n_b(G) = \HHH^n_b(G ; \RR)$ is a subquotient of $C^n_b(G ; \RR)$, it has a seminorm induced by the norm on $C^n_b(G ; \RR)$. Namely, for $\alpha \in \HHH^n_b(G)$ the seminorm $\| \alpha \|$ on $\HHH^n_b(G)$ is defined by
\[ \| \alpha \| = \inf \{ \| \varphi \|_\infty \; | \; \textrm{$\varphi$ is a bounded $n$-cocycle on $G$ satisfying $[\varphi] = \alpha$}\}. \]

\begin{thm}[See Theorem 2.47 of \cite{Ca}] \label{thm amenable cover}
If $\Gamma = \hG / \bG$ is  amenable, then the map $\HHH^n_b(\hG) \to \HHH^n_b(\bG)^\hG$ is an isometric isomorphism for every $n$.
\end{thm}

We recall the following estimate of the defect of the homogenization:

\begin{lem}[Lemma 2.58 of \cite{Ca}] \label{lem defect of homogenization}
Let $\delta \colon \QQQ(\hG) \to \HHH^2_b(G)$ be the natural map. Then
\[ \| [\delta f] \| \le  D_G(f)  \le 2 \cdot \| [\delta f] \|.\]
\end{lem}

\begin{proof}[Proof of $(2)$ of Theorem $\ref{main thm 3.1}$]
Suppose that $\Gamma = \hG / \bG$ is  amenable and $\QQQ(\bG)^\hG = \HHH^1(\bG)^{\hG} + i^* \QQQ(\hG)$. Let $ f  \in \QQQ(G)$. By Corollary \ref{cor equivalence criterion}, it suffices to show that $ 2 D_N(f |_N) \ge D_G(f) $ for every $f \in \QQQ(G)$. This is deduced from the following inequalities:
\[ 2  D_N( f \relax |_N) \ge 2 \| [\delta  f  |_N] \| \ge 2 \| [\delta  f  ]\| \ge  D_G(f) .\]
Here the third and last inequalities are deduced from Theorem \ref{thm amenable cover} and Lemma \ref{lem defect of homogenization}, respectively.
\end{proof}



\subsection{Proof of (3) of Theorem \ref{main thm 3.1}}

Next, we prove (3) of Theorem \ref{main thm 3.1}.

\begin{lem} \label{lem 5.2}
Let $f \colon \bG \to \RR$ be an extendable homogeneous quasimorphism on $\bG$. Then for each $a,b \in \hG$ satisfying $[a,b] \in \bG$, we have
\[|f([a,b])| \le  D_N(f) .\]
\end{lem}
\begin{proof}
We first prove the following equality:
\begin{eqnarray} \label{eqn 5.1}
[a^n, b] = a^{n-1} [a,b] a^{-(n-1)} \cdot a^{(n-2)} [a,b] a^{-(n-2)} \cdots [a,b].
\end{eqnarray}
Indeed, we have
\begin{eqnarray*}
[a^n, b] & = & a^n b a^{-n} b^{-1} \\
& = & a^{n-1} \cdot aba^{-1}b^{-1} \cdot a^{-(n-1)} \cdot a^{n-1} b a^{-(n-1)} b^{-1} \\
& = & a^{n-1} [a,b] a^{-(n-1)} \cdot [a^{n-1}, b].
\end{eqnarray*}
By induction on $n$, we have proved \eqref{eqn 5.1}. Since $f$ is $G$-invariant, we have
\[f([a^n, b]) \underset{(n-1) D_N(f) }{\approx} f(a^{n-1} [a,b] a^{-(n-1)}) + \cdots + f([a,b])
= n \cdot f([a,b]).\]
Therefore we have
\[|f([a^n,b])| \ge n \cdot \big( |f([a,b])| -  D_N(f)  \big).\]
Suppose that $|f([a,b])| >  D_N(f) $. Then the right of the above inequality can be unbounded with respect to $n$. However, since $f$ is extendable, the left of the above inequality is bounded. This is a contradiction.
\end{proof}

In Corollary \ref{cor extend homomorphism 1}, we provide a condition that a $G$-invariant homomorphism $f \colon N \to \RR$ cannot be extended to $G$ as a quasimorphism. Here we present another condition.

\begin{cor} \label{cor 5.3}
Let $f \colon N \to \RR$ be a $G$-invariant homomorphism and assume that $N$ is generated by single commutators of $G$. If $f$ is non-zero, then $f$ is not extendable.
\end{cor}
\begin{proof}
If $f$ is extendable, then Lemma \ref{lem 5.2} implies that $f(c) = 0$ for every single commutator $c$ of $G$ contained in $N$. Since $N$ is generated by single commutators of $G$, this means that $f = 0$.
\end{proof}

\begin{lem} \label{lem 5.4}
Let $f$ be a homogeneous quasimorphism on $\hG$, and assume that $\Gamma = \hG/\bG$ is solvable.
Then $ D_G(f)  =  D_N(f|_N) $.
\end{lem}
\begin{proof}
  We first assume that $\Gamma$ is abelian. It is known that the equality $ D_G(f)  = \sup_{a,b \in \hG} |f([a,b])|$ holds (see Lemma 2.24 of \cite{Ca}).
Applying Lemma \ref{lem 5.2} to $f|_N$, we have
\[ D_G(f)  = \sup_{a,b \in \hG} |f([a,b])| \le  D_N(f|_N)  \le  D_G(f) ,\]
 and in particular, $ D_G(f) = D_N(f|_N) $.

Next we consider the general case.
Let $\hG^{(n)}$ denote the $n$-th derived subgroup of $\hG$.
Then there  exists a positive integer $n$  such that  $\hG^{(n)} \subset \bG$ since $\Gamma$ is solvable. By the previous paragraph, we have
\[  D_G(f) = D_{G^{(1)}}(f|_{G^{(1)}}) = \cdots = D_{G^{(n)}}(f|_{G^{(n)}}) \le D_N(f|_N) \le D_G(f).  \qedhere \]
\end{proof}


\begin{proof}[Proof of $(3)$ of Theorem $\ref{main thm 3.1}$]
 Combine Lemma \ref{lem 5.4} and Corollary \ref{cor equivalence criterion}.
%
\end{proof}

Here we provide some  applications of (3) of Theorem \ref{main thm 3.1}.

\begin{cor}
If one of the following conditions holds, then $\scl_\hG = \scl_{\hG,\bG}$ on $[\hG,\bG]$. Here, $\Gamma=\hG/\bG$.
\begin{enumerate}
\item[$(1)$] $\Gamma$ is  a  finite solvable group.

\item[$(2)$] $\Gamma$ is  a  finitely generated abelian group whose rank is at most $1$.
\end{enumerate}
\end{cor}
\begin{proof}
This clearly follows from Proposition \ref{virtual split KKMM1} and (3) of Theorem \ref{main thm 3.1}.
\end{proof}

In Subsection \ref{subsec:equivalence}, we propose several problems on the coincidence and equivalence of $\scl_{\hG}$ and $\scl_{\hG, \bG}$.

\subsection{Examples with non-equivalent $\scl_G$ and $\scl_{G,N}$}\label{subsec:non-ex}
 At the end of this section, we provide some examples of group pairs $(G,N)$ for which $\scl_{G}$ and $\scl_{G,N}$ \emph{fail} to be bi-Lipschitzly equivalent on $[G,N]$.

\begin{example}\label{ex:KK}
Let $\genus$ be an integer at least $2$, and $\Omega$ be an area form of $\Sigma_{\genus}$. In this case, the flux group $\Gamma_\Omega$ is known to be trivial; thus we have the volume flux homomorphism $\flux_{\Omega}\colon \diff_0(\Sigma,\Omega)\to \HHH^1(\Sigma_{\genus})$. In \cite{KK}, the authors proved that for the pair
\[
(G,N)=(\diff_0(\Sigma_{\genus},\Omega), \Ker(\flux_{\Omega})),
\]
$\scl_{G}$ and $\scl_{G,N}$ are \emph{not} bi-Lipschitzly equivalent on $[G,N]$. More precisely, we  found  an element $\gamma$ in $[G,N]$ such that
\[
\scl_{G}(\gamma)=0 \quad \textrm{but}\quad \scl_{G,N}(\gamma)>0.
\]
\end{example}

\begin{example}\label{ex:KKMM}
We can provide the following example, which is related to Example~\ref{ex:KK} with smaller $G$, from results in \cite{KKMM2}. We stick to the setting of Example~\ref{ex:KK}. Take an arbitrary pair $(v,w)$ with $v,w \in \HHH^1(\Sigma_{\genus})$ that satisfies
\begin{eqnarray}\label{eq:cup}
v\smile w \ne 0.
\end{eqnarray}
Here, recall from Theorem~\ref{thm flux} that $\smile \colon \HHH^1(\Sigma_{\genus}) \times \HHH^1(\Sigma_{\genus}) \to \HHH^2(\Sigma_{\genus}) \cong \RR$ denotes the cup product. Then, from results in \cite{KKMM2} we can deduce the following: there exists a positive integer $k_0$ depending only on $w$ and the area  of $\Sigma_{\genus}$ such that for every $k\geq k_0$, if we set
\[
\Lambda_k=\langle v,w/k\rangle,
\]
 namely the subgroup of $\HHH^1(\Sigma_{\genus})$ generated by $v$ and $w/k$,  and
\[
(G,N)=(\flux_{\Omega}^{-1}(\Lambda_k),\Ker(\flux_{\Omega})),
\]
then $\scl_{G}$ and $\scl_{G,N}$ are \emph{not} bi-Lipschitzly equivalent on $[G,N]$. To see this, by following \cite[Section~4]{KKMM2}, we construct a sequence $(\gamma_m)_{m\in \NN}$ in $[G,N]$. Then, Proposition~4.6 in \cite{KKMM2}, together with Bavard's duality theorem, implies that
\[
\sup_{m\in \NN} \scl_{G}(\gamma_m)\leq \frac{3}{2}.
\]
Contrastingly, (3) of Proposition~4.7 in \cite{KKMM2}, together with Theorem~\ref{thm Bavard}, implies that
\[
\liminf_{m\to \infty} \frac{\scl_{G,N}(\gamma_m)}{m}\geq \frac{1}{2 k\cdot  D_N( f_P) }|\mathfrak{b}_I(v,w)|>0.
\]
Here $\mathfrak{b}_I(v,w) = \langle v \smile w, [\Sigma_l] \rangle_{\Sigma_l} \in \RR$ is the intersection number of $v$ and $w$, where $[\Sigma_l]$ is the fundamental class of $\Sigma_l$ and $\langle - , - \rangle_{\Sigma_l} \colon \HHH^2(\Sigma_{\genus}) \times \HHH_2(\Sigma_{\genus}) \to \RR$ denotes the Kronecker pairing of $\Sigma_l$.  The map $f_P\colon N\to \RR$ is Py's Calabi quasimorphism (recall Section~\ref{flux section}; see also \cite[Subsections~2.4 and 2.5]{KKMM2}). We also note that $v,w$ and $f_P$ here correspond to $\bar{v},\bar{w}$ and $\mu_P$ in \cite{KKMM2}, respectively.
\end{example}

\section{$\Aut(F_n)$ and $\IA_n$}\label{section=proofAut}
\subsection{Proof of Theorem \ref{thm 3.2}}

An {\it IA-automorphism} of a group $G$ is an automorphism $f$ on $G$ which acts as identity on the abelianization $\HHH_1(G; \ZZ)$ of $G$. We write $\IA_n$ to indicate the group of IA-automorphisms on $F_n$. Then we have exact sequences
\[1 \to \IA_n \to \Aut(F_n) \to \GL(n, \ZZ) \to 1,\]
\[1 \to \IA_n \to \Aut_+(F_n) \to \SL(n, \ZZ) \to 1.\]


Theorem \ref{thm 3.2} (1) claims that the equalities $\QQQ(\IA_n)^{\Aut(F_n)} = i^* \QQQ(\Aut(F_n))$ and $\QQQ(\IA_n)^{\Aut_+(F_n)} = i^* \QQQ(\Aut_+(F_n))$ hold.
To show it, we use the following facts, which can be derived from the computation of the second integral homology $\HHH_2(\SL(n,\ZZ),\ZZ)$.
\begin{thm}[See \cite{Milnor}] \label{thm 7.10 new}
For $n \ge 3$, $\HHH^2(\SL(n,\ZZ)) = 0$ and $\HHH^2(\GL(n, \ZZ)) = 0$.
\end{thm}
It is known that the following holds, which is obtained from \cite[Corollary 1.4]{Monod2} and \cite[Theorem 1.2]{Monod1}.
\begin{thm}\label{thm:3bdd_coh_vanish_SL}
Let $n$ be an integer at least $3$ and $\ppi_0$ a subgroup of finite index of $\SL(n, \ZZ)$.
Then $\HHH^3_b(\ppi_0) = 0$.
\end{thm}
\begin{remark} \label{rem 7.7 new}
In \cite{Monod2}, Monod used $\HHH^\bullet_b$ to mean the continuous bounded cohomology $\HHH^\bullet_{cb}$.
\end{remark}

The following theorem is known, which is a special case of \cite[Proposition 8.6.2]{MR1840942}.
\begin{thm}\label{thm:transfer}
  Let $\bG$ be a subgroup of finite index in $\hG$ and $V$ a Banach $\hG$-module, then the restriction $\HHH_b^n(\hG;V) \to \HHH_b^n(\bG;V)$ is injective for every $n \geq 0$.
\end{thm}

Now we proceed to the proof of (1) of Theorem \ref{thm 3.2}.
First, we show the following lemma.
\begin{lem}\label{lem:bdd_coh_vanish_GL}
  Let $n$ be an integer at least $3$, and $\ppi_0$ a subgroup of finite index of $\GL(n, \ZZ)$.
   Then $\HHH^3_b(\ppi_0) = 0$.
\end{lem}

\begin{proof}
Since the intersection $\ppi_0 \cap \SL(n,\ZZ)$ is a subgroup of finite index of $\SL(n, \ZZ)$, we have $\HHH_{b}^3(\ppi_0 \cap \SL(n, \ZZ)) = 0$ by Theorem \ref{thm:3bdd_coh_vanish_SL}.
Since $\ppi_0 \cap \SL(n,\ZZ)$ is a subgroup of finite index of $\ppi_0$, we obtain $\HHH^3_b(\ppi_0) = 0$ by Theorem \ref{thm:transfer}.
\end{proof}

\begin{proof}[Proof of $(1)$ of Theorem $\ref{thm 3.2}$]
  Suppose that $n = 2$.
  Then $\GL(n,\ZZ)$ and $\SL(n,\ZZ)$ have a subgroup of finite index which is isomorphic to a free group.
  Therefore this case is proved by Proposition \ref{virtual split KKMM1}.
In what follows, we treat the case where $n$ is greater than $2$.
  Let $\ppi$ be either $\GL(n, \ZZ)$ or $\SL(n,\ZZ)$.
  By Theorem \ref{thm 7.10 new}, Lemma \ref{lem:bdd_coh_vanish_GL}, and the cohomology long exact sequence, we have $\HHH_{/b}^2(\ppi) = 0$.
  Hence Theorem \ref{main thm} implies that $\QQQ(\IA_n)^{\Aut(F_n)}/ i^* \QQQ(\Aut(F_n)) = 0$ and $\QQQ(\IA_n)^{\Aut_+(F_n)}/ i^* \QQQ(\Aut_+(F_n)) = 0$.
\end{proof}

Next, we prove (2) of Theorem \ref{thm 3.2}.
In the proof, we use the following theorem, due to Borel \cite{Borel74}, \cite{Borel80}, \cite{Borel81} and Hain \cite{Hain} (and Tshishiku \cite{Tshishiku}).


\begin{thm} \label{thm 7.9 new}
The following hold:
\begin{enumerate}[$(1)$]
\item  For every $n \ge 6$ and for every subgroup $\ppi_0$ of finite index of $\GL(n,\ZZ)$, $\HHH^2(\ppi_0) = 0$.


\item For every $l \ge 3$ and for every subgroup $\ppi_0$ of finite index of $\Sp(2l,\ZZ)$, the inclusion map $\ppi_0 \hookrightarrow \Sp(2l,\ZZ)$ induces an isomorphism of cohomology $\HHH^2(\Sp(2l, \ZZ)) \cong \HHH^2(\ppi_0)$.
In particular, the cohomology $\HHH^2(\ppi_0)$ is isomorphic to $\RR$.
\end{enumerate}
\end{thm}

 For the convenience of the reader, we describe the deduction of Theorem \ref{thm 7.9 new} from the work of Borel, Hain and Tshishiku.
\begin{proof}
 First, we discuss (2). It is stated in Theorem 3.2 of \cite{Hain}; see \cite{Tshishiku} for the complete proof.

 Next, we treat (1). Let $\Lambda=\ppi_0\cap \SL(n,\ZZ)$. Then $\Lambda$ is a subgroup of finite index of $\SL(n,\ZZ).$ An argument using the transfer (similar to one in the proof of Theorem~\ref{thm:transfer}) shows that the restriction $\HHH^2(\ppi_0)\to \HHH^2(\Lambda)$ is injective.  Hence, to prove (1), it suffices to show that $\HHH^2(\Lambda)=0$. In what follows, we sketch the deduction of $\HHH^2(\Lambda)=0$ from the work of Borel; see also the discussion in the introduction of \cite{Tshishiku}.

 We appeal to Borel's theorem, Theorem~$1$ in \cite{Borel80}, with $G=\SL_n$, $\Gamma=\Lambda$, and $r$ being the trivial complex representation. (See also Theorem~11.1 in \cite{Borel74}.)
 Then we have the following conclusion: there exists a natural homomorphism $\HHH^q(\mathfrak{g},\mathfrak{k};\mathbb{C})^{\Lambda}\to \HHH^q(\Lambda;\mathbb{C})$ and if $q\leq \min \{c(\SL_n),\mathrm{rank}_{\RR}(\SL(n,\RR))-1\}$, then this map is an isomorphism.
 Here $\HHH^q(\mathfrak{g},\mathfrak{k};\mathbb{C})^{\Lambda}$ is a Lie algebraic cohomology ($\mathfrak{g}$ and  $\mathfrak{k}$ stand for the Lie algebras of $\SL(n,\RR)$ and $\mathrm{SO}(n)$, respectively); it is known that $\HHH^2(\mathfrak{g},\mathfrak{k};\mathbb{C})^{\Lambda}=0$; see 11.4 of \cite{Borel74}. For the definition of the constant $c(G)=c(G,0)$ for $G$ being a connected semisimple group defined over $\mathbb{Q}$, see 7.1 of \cite{Borel74}. We remark that for the trivial complex representation $r$, $c(G,r)=(c(G,0)=)c(G)$. Since $\SL_n$ is of type $\mathrm{A}_{n-1}$, the constant $c(\SL_n)$ equals $\lfloor \frac{n-2}{2}\rfloor$; see \cite{Borel80} (and 9.1 of \cite{Borel74}). Here, $\lfloor\cdot\rfloor$ denotes the floor function. The number $\mathrm{rank}_{\RR}(\SL(n,\RR))$ means the real rank of $\SL(n,\RR)$; it equals $n-1$. Since $n\geq 6$, we hence have $\min\{c(\SL_n),\mathrm{rank}_{\RR}(\SL(n,\RR))-1\}\geq 2$.
Therefore, we conclude that $\HHH^2(\Lambda;\mathbb{C})=0$. This immediately implies that $\HHH^2(\Lambda)=0$, as desired. (In \cite{Borel81}, Borel considered a better constant $C(G)$ than $c(G)$ in general, but $C(\SL_n)=\lfloor \frac{n-2}{2}\rfloor$; see \cite{Borel81} and also \cite{Tshishiku}.)
\end{proof}

\begin{remark}
  In the proof of (2) of Theorem \ref{thm 3.2}, we only use (1) of Theorem \ref{thm 7.9 new}.
  We will use (2) of Theorem \ref{thm 7.9 new} in the proofs of claims in the next subsection.
\end{remark}

\begin{proof}[Proof of $(2)$ of Theorem $\ref{thm 3.2}$]
 Let $n$ be an integer at least $6$. Let $\hG$ be a group of finite index of $\Aut(F_n)$. Set $\bG = \hG \cap \IA_n$ and $\ppi = \hG / \bG$. Then we have an exact sequence
  \[1 \to \bG \to \hG \to \ppi \to 1\]
  and $\ppi$ is a subgroup of finite index of $\GL(n, \ZZ)$.
  By Lemma \ref{lem:bdd_coh_vanish_GL} and (1) of Theorem \ref{thm 7.9 new}, the second relative cohomology group $\HHH_{/b}^2(\ppi)$ is trivial.
  Therefore, by Theorem \ref{main thm}, we have $\QQQ(\bG)^{\hG} / i^* \QQQ(\hG) = 0$.
\end{proof}

\begin{remark}\label{rem:bdd_3}
By Theorem~\ref{thm:3bdd_coh_vanish_SL} and \cite[Theorem~1.4]{MS04}, the following holds true: for every $n\geq 3$, every subgroup of finite index of $\SL(n,\ZZ)$ is boundedly $3$-acyclic.
\end{remark}

\subsection{Quasi-cocycle analogues of Theorem \ref{thm 3.2}}\label{subsec:quasicocycle_extension}
To state our next result, we need some notation.
In Subsection \ref{subsec:proof_of_scl}, we introduced the notion of $\hG$-quasi-equivariant quasimorphism.
Let $V$ be an $\RR[\hG]$-module whose $\hG$-action on $V$ is trivial at $\bG$.
The $\hG$-quasi-invariance can be extended to the $V$-valued quasimorphisms as the $\hG$-quasi-equivariance.
Recall from Remark \ref{rem:V-valued_quasimorphisms} that a $V$-valued quasimorphism $f \colon \bG \to V$ is $\hG$-equivariant if the condition $\qm(gxg^{-1}) - g \cdot \qm(x) = 0$ holds.
A $V$-valued quasimorphism $\qm \colon \bG \to V$ is said to be \textit{$\hG$-quasi-equivariant} if the number
\[D'(\qm) = \sup_{g \in \hG, x \in \bG} \| \qm(gxg^{-1}) - g \cdot \qm(x) \|\]
is finite.
Let $\rQQQ(\bG;V)^{\QQQ \hG}$ denote the $\RR$-vector space of all $\hG$-quasi-equivariant $V$-valued quasimorphisms.
Let $F \colon \hG \to V$ be a quasi-cocycle, then the restriction $F|_{\bG}$ belongs to $\rQQQ(\bG;V)^{\QQQ \hG}$ by definition.
It is straightforward to show that the quotient $\rQQQ(\bG;V)^{\QQQ \hG}/i^* \rQQQ Z(\hG;V)$ is isomorphic to $\QQQ(\bG;V)^{\hG}/i^* \HHH_{/b}^1(\hG;V) = \HHH_{/b}^1(\bG;V)^{\hG}/i^* \HHH_{/b}^1(\hG;V)$.
 We summarize the concepts and symbols on quasi-cocycles in Table \ref{table:symbol_qc}.

\begin{table}[hbtp]
   \caption{the concepts and symbols on quasi-cocycles}
  \label{table:symbol_qc}
  \centering
  \begin{tabular}{c|c|c|c}
    \hline
    concept & defect  &  definition  & vector space \\
    \hline     \hline
    quasi-cocycle on $G$  & $D$  & $F(g_1 g_2) \approx_D F(g_1)+ g_1 \cdot F(g_2) $ & $\rQQQ Z(\hG;V)$  \\
\hline
  \begin{tabular}{c}
$G$-quasi-equivalent \\ quasimorphism on $N$
\end{tabular}  & $D,D'$ & 
\begin{tabular}{c}
$f(x_1x_2)\approx_D f(x_1)+f(x_2)$, \\ $f(gxg^{-1}) \approx_{D'} g \cdot f(x)$ 
\end{tabular} 
& $\rQQQ(\bG;V)^{\QQQ \hG}$  \\
 \hline
   $N$-quasi-cocycle on $G$ & $D''$  &
  \begin{tabular}{c}
$F(gx) \approx_{D''} F(g)+g \cdot F(x) $, \\ $F(xg) \approx_{D''} F(x)+F(g) $
\end{tabular}  & $\rQQQ Z_N(G;V)$  \\
    \hline
  \end{tabular} 
\end{table}

Let $\hG$ be a subgroup of $\Aut(F_n)$. Then we set $\bG = \hG \cap \IA_n$ and set $\ppi = \hG / \bG$.
Our main results in this section are the following two theorems:

\begin{thm} \label{thm 7.1.2}


 Let $n$ be an integer at least $6$, and $\hG$ a subgroup  of finite index of $\Aut(F_n)$. Then for every finite dimensional unitary representation $\pi$ of $\ppi$, the equality
\[\rQQQ(\bG ; \HH)^{\QQQ \hG} = i^* \rQQQ Z(\hG ; \overline{\pi}, \HH)\]
holds. Here $(\overline{\pi}, \HH)$ is the pull-back representation of $\hG$ of the representation $(\pi, \HH)$ of $\ppi$.
\end{thm}



\begin{thm} \label{thm 7.4}
Let $l$ be an integer at least $3$, and $\hG$ a subgroup of finite index of $\Mod(\Sigma_l)$. Set $\bG = \hG \cap \II(\Sigma_{\genus})$ and $\ppi = \hG / \bG$. Let $(\pi, \HH)$ be a finite dimensional $\ppi$-unitary representation such that $1_\hG \not\subseteq \pi$, {\it i.e.}, $\HH^{\pi(\ppi)} = 0$.
Then we have the equality
\[\rQQQ(\bG ; \HH)^{\QQQ \hG} = i^* \rQQQ Z(\hG ; \overline{\pi}, \HH).\]
Here $\overline{\pi}$ is the pull-back of $\pi$ by the quotient homomorphism $\hG \to \ppi$.
\end{thm}

Before proceeding to the proofs of Theorems \ref{thm 7.1.2} and \ref{thm 7.4}, we mention some known results we need in the proofs. The following theorem is well known (see Corollary 4.C.16 and Corollary 4.B.6 of \cite{BH}).

\begin{thm} \label{thm 7.5}
Let $\ppi_0$ be a subgroup of finite index of $\GL(n,\ZZ)$ for $n \ge 3$ or $\Sp(2l, \ZZ)$ for $l \ge 3$, and $(\pi, \HH)$ a finite dimensional unitary $\ppi_0$-representation.
Then $\ppi_0(\pi) := \Ker (\pi \colon \ppi_0 \to \mathcal{U}(\HH))$ is a subgroup of finite index of $\ppi_0$,
where $\mathcal{U}(\HH)$ denotes the group of unitary operators on $\HH$.
\end{thm}

\begin{thm}[{\cite[Corollary 1.6]{Monod2}}] \label{thm:3bdd_vanish_Sp}
Let $l$ be an integer at least $2$ and $\ppi_0$ a subgroup of finite index of $\Sp(2l, \ZZ)$. Let $(\pi, \HH)$ be a unitary $\ppi_0$-representation such that $\pi \not\supset 1$. Then 
$\HHH^3_b(\ppi_0 ; \pi, \HH) = 0$.
\end{thm}

By the higher inflation-restriction exact sequence (\cite[Theorem 2 of Chapter III]{MR52438}), we obtain the following:
\begin{lem} \label{lem HIR}
Let $\bG$ be a normal subgroup of finite index of $\hG$, $V$ a real $G$-module, and $q_0$ a positive integer. Assume that $\HHH^q(\bG ; V) = 0$ for every $q$ with $1 \le q < q_0$.
Then the restriction induces an isomorphism $\HHH^{q_0}(\hG ; V) \xrightarrow{\cong} \HHH^{q_0}(\bG ; V)^{\ppi}$.
\end{lem}

\begin{cor} \label{cor 7.11}
The following hold:
\begin{enumerate}[$(1)$]
\item

 Let $n$ be an integer at least $6$, and $\ppi_0$ a subgroup of finite index of $\GL(n, \ZZ)$. Let $(\pi, \HH)$ be a finite dimensional unitary $\ppi_0$-representation.
 Then $\HHH^2(\ppi_0 ; \pi, \HH) = 0$.

\item Let $l$ be an integer at least $3$, $\ppi_0$ a subgroup of finite index of $\Sp(2l, \ZZ)$, and $(\pi, \HH)$ a finite dimensional unitary $\ppi_0$-representation such that $\pi \not\supset 1$. Then $\HHH^2(\ppi_0 ; \pi, \HH) = 0$.
\end{enumerate}
\end{cor}
\begin{proof}
We first prove (2). Set $\ppi_0(\pi) = \Ker (\pi)$. Then, Theorem~\ref{thm 7.5} implies that $\ppi_0(\pi)$ is of finite index in $\ppi_0$.
We claim that $\HHH^1(\ppi_0(\pi) ; \HH) =0$. Indeed, it follows from the Matsushima vanishing theorem \cite{Matsushima}.
Or alternatively, we may appeal to the fact that $\ppi_0(\pi)$ has property (T); see \cite{BHV}.
By Lemma \ref{lem HIR}, we have an isomorphism $\HHH^2(\ppi_0 ; \pi, \HH) \cong \HHH^2(\ppi_0(\pi) ; \HH)^{\ppi_0 / \ppi_0(\pi)}$.

We now show the following claims:

\vspace{2mm} \noindent
{\bf Claim.} The conjugation action by $\ppi_0$ on the cohomology $\HHH^2(\ppi_0(\pi))$ is trivial.

\vspace{2mm} \noindent
{\it Proof of Claim.}
By (2) of Theorem \ref{thm 7.9 new} and Theorem \ref{thm 7.5}, the inclusion $i \colon \ppi_0(\pi) \hookrightarrow \ppi_0$ induces an isomorphism $i^* \colon \HHH^2(\ppi_0) \cong \HHH^2(\ppi_0(\pi))$.
Hence, for every $a \in \HHH^2(\ppi_0(\pi))$, there exists a cocycle $c \in C^2(\ppi_0)$ such that $[i^* c] = a$.
By definition, the equalities
\[
  {}^{\gamma}a = [{}^{\gamma} (i^*c)] = [i^* ({}^{\gamma} c)] = i^* ({}^{\gamma}[c])
\]
hold for every $\gamma \in \ppi_0$.
Since the conjugation $\ppi_0$-action on $\HHH^2(\ppi_0)$ is trivial, the class ${}^{\gamma}[c] \in \HHH^2(\ppi_0)$ is equal to $[c]$.
Therefore we have
\[
  {}^{\gamma}a = i^* {}^{\gamma}[c] = i^* [c] = a,
\]
and the claim follows.

\vspace{2mm} \noindent
{\bf Claim.}
There exists a canonical isomorphism $\HHH^2(\ppi_0(\pi);\HH) \cong \HH$, and this isomorphism induces an isomorphism $\HHH^2(\ppi_0(\pi);\HH)^{\ppi_0/\ppi_0(\pi)} \cong \HH^{\ppi_0/\ppi_0(\pi)}$.

\vspace{2mm} \noindent
{\it Proof of Claim.}
By (2) of Theorem \ref{thm 7.9 new}, the cohomology $\HHH^2(\ppi_0(\pi))$ is isomorphic to $\RR$, and hence the cohomology $\HHH^2(\ppi_0(\pi); \HH)$ is isomorphic to $\HH$.
In what follows, we exhibit a concrete isomorphism.
For $\alpha \in \HH$, we define a cochain $c_{\alpha} \in C^2(\ppi_0(\pi);\HH)$ by
\begin{align*}
  c_{\alpha}(\gamma_1, \gamma_2) = c(\gamma_1, \gamma_2)\cdot \alpha \in \HH,
\end{align*}
where $c \in C^2(\ppi_0(\pi))$ is a cocycle whose cohomology class corresponds to $1 \in \RR$ under the isomorphism $\HHH^2(\ppi_0(\pi)) \cong \RR$.
This cochain $c_{\alpha}$ is a cocycle since the $\ppi_0(\pi)$-action on $\HH$ is trivial.
Then the map sending $\alpha$ to $[c_\alpha]$ gives rise to an isomorphism $\HH \xrightarrow{\cong} \HHH^2(\ppi_0(\pi); \HH)$.
For $\gamma \in \ppi_0$ and $\gamma_1, \gamma_2 \in \ppi_0(\pi)$, the equalities
\begin{align*}
  ({}^{\gamma}c_{\alpha})(\gamma_1, \gamma_2) &= \pi(\gamma) \cdot c_{\alpha}(\gamma^{-1}\gamma_1 \gamma, \gamma^{-1}\gamma_2 \gamma)\\
  &= \pi(\gamma) \cdot (({}^{\gamma}c)(\gamma_1, \gamma_2)\cdot \alpha)\\
  &= ({}^{\gamma}c)(\gamma_1, \gamma_2)\cdot (\pi(\gamma) \cdot \alpha)
\end{align*}
hold.
Moreover, by the claim above, there exists a cochain $b \in C^1(\ppi_0(\pi))$ satisfying ${}^{\gamma}c = c + \delta b$.
Hence we have
\begin{align*}
  ({}^{\gamma}c_{\alpha})(\gamma_1, \gamma_2) &= ({}^{\gamma}c)(\gamma_1, \gamma_2)\cdot (\pi(\gamma) \cdot \alpha) = (c + \delta b)(\gamma_1, \gamma_2) \cdot (\pi(\gamma) \cdot \alpha)\\
  &= (c + \delta b)_{\pi(\gamma) \cdot \alpha}(\gamma_1, \gamma_2).
\end{align*}
Therefore the cohomology class ${}^{\gamma}[c_{\alpha}]$ corresponds to the element $\pi(\gamma) \cdot \alpha$ under the isomorphism, and this implies the claim.

By claims above and the assumption that $\pi$ does not contain trivial representation, we have $\HHH^2(\ppi_0 ; \pi, \HH) = 0$. This completes the proof of (2).

We can deduce (1) by the same arguments as above with Theorem \ref{thm 7.9 new}, Theorem \ref{thm 7.5}, and Lemma \ref{lem HIR}.
\end{proof}

\begin{proof}[Proof of Theorem $\ref{thm 7.1.2}$]

 Let $n$ be an integer at least $6$. Let $\hG$ be a group of finite index of $\Aut(F_n)$.
Set $\bG = \hG \cap \IA_n$ and $\ppi = \hG / \bG$. Then we have an exact sequence
  \[1 \to \bG \to \hG \to \ppi \to 1\]
  and $\ppi$ is a subgroup of finite index of $\GL(n, \ZZ)$.
  Let $(\pi, \HH)$ be a finite dimensional unitary $\ppi$-representation. Set $\ppi(\pi) = \Ker (\pi)$. By Theorem \ref{thm 7.5}, $\ppi(\pi)$ is a normal subgroup of finite index of $\ppi$.
  By using Lemma \ref{lem:bdd_coh_vanish_GL}, we have 
  $\HHH^3_b(\ppi(\pi) ; \HH) = 0$.
  Together with Theorem \ref{thm:transfer}, we obtain $\HHH_b^3(\ppi;\pi, \HH) = 0$.
  Hence, by Corollary \ref{cor 7.11} (1), we have $\HHH_{/b}^2(\ppi;\pi, \HH) = 0$.
  Therefore, the quotient $\HHH_{/b}^1(\bG;\HH)/i^* \HHH_{/b}^1(\hG; \pi, \HH)$ is trivial by Theorem \ref{main thm}.
  Since the quotient $\HHH_{/b}^1(\bG;\HH)/i^* \HHH_{/b}^1(\hG; \pi, \HH)$ is isomorphic to $\rQQQ(\bG;\HH)^{\QQQ \hG}/i^* \rQQQ Z(\hG;\pi, \HH)$, this completes the proof.
\end{proof}

\begin{proof}[Proof of Theorem $\ref{thm 7.4}$]
  Let $l$ be an integer at least $3$. Let $\hG$ be a subgroup of finite index of $\Mod(\Sigma_l)$. Set $\bG = \hG \cap \II(\Sigma_{\genus})$ and $\ppi = \hG / \bG$.
  Let $(\pi, \HH)$ be a finite dimensional unitary $\ppi$-representation not containing trivial representation.
  Then, Theorem \ref{thm:3bdd_vanish_Sp} and (2) of Corollary \ref{cor 7.11} imply that the second relative cohomology group $\HHH_{/b}^2(\ppi;\pi,\HH)$ is trivial.
  Hence, by the arguments similar to ones in the proof of Theorem \ref{thm 7.1.2}, we obtain the theorem.
\end{proof}

We conclude this subsection by an extension theorem of quasi-cocycles.
Recall that every $\hG$-quasi-invariant quasimorphism on $\bG$ is extendable to $\hG$ if the projection $\hG \to \hG/\bG$ virtually splits (Proposition \ref{virtual split KKMM1}).
This can be generalized as follows:

\begin{thm}\label{thm:extension_thm_genral_coeff}
  Let $1 \to \bG \to \hG \xrightarrow{p} \ppi \to 1$ be an exact sequence and $V$ an $\RR[\ppi]$-module with a $\ppi$-invariant norm $\| \cdot \|$.
  Assume that the exact sequence virtually splits.
  Then for every $V$-valued $\hG$-quasi-equivariant quasimorphism $\qm \in \rQQQ(\bG;V)^{\QQQ \hG}$, there exists a quasi-cocycle $F \in \rQQQ Z(\hG;V)$ such that the equality $F|_{\bG} = \qm$ and the inequality $D(F) \le D(\qm) + 3D'(\qm)$ hold.
\end{thm}

The proof is parallel to that of \cite[Proposition 6.4]{KKMM1} (Proposition \ref{virtual split KKMM1}).
For the sake of completeness, we include the proof; see \cite[the proof of Proposition 6.4]{KKMM1} for more details.

\begin{proof}[Proof of Theorem $\ref{thm:extension_thm_genral_coeff}$]
  Let $(s, \fis)$ be a virtual section of $p\colon \hG \to \ppi$ (see Section 2).
  Let $B$ be a finite subset of $\ppi$ such that the map $\fis \times B \to \ppi, (\lambda, b) \mapsto \lambda b$ is bijective.
  Let $s' \colon B \to  \hG $ be a map satisfying $p \circ s'(b) = b$ for every $b \in B$.
  Define a map $t \colon \ppi \to \hG$ by setting $t(\lambda b) = s(\lambda)s'(b)$.
   Note that $t$ is a (set-theoretic) section of $p$. 
  Given $\qm \in \rQQQ(N ; V)^{\QQQ G}$, define a function $F \colon \hG \to V$ by
  \begin{align*}
    F(g) = \frac{1}{\# B} \sum_{b \in B}\qm(g \cdot t(b \cdot p(g))^{-1}\cdot t(b)).
  \end{align*}
  Then we have $F|_{\bG} = \qm$.
  Moreover, for $g_1, g_2 \in \hG$,
   by using that $f(h_1h_2) \approx_{D'(f)} p(h_1)\cdot  f(h_2h_1)$ and  $f(h_1h_2) \approx_{D'(f)} p(h_2)^{-1} \cdot  f(h_2h_1)$ for every $h_1,h_2 \in G$ with $h_1h_2\in N$,
  we have
  \begin{align*}
    F(g_1 g_2) &= \frac{1}{\# B} \sum_{b \in B}\qm(g_1g_2 \cdot t(b \cdot p(g_1g_2))^{-1}t(b))\\
    &\underset{D'(\qm)}{\approx} \frac{1}{\# B} \sum_{b \in B}b^{-1} \cdot \qm(t(b) \cdot g_1g_2 \cdot t(b \cdot p(g_1g_2))^{-1})\\
    &= \frac{1}{\# B} \sum_{b \in B}b^{-1} \cdot \qm(t(b) \cdot g_1 \cdot t(b \cdot p(g_1))^{-1} \cdot t(b \cdot p(g_1)) \cdot g_2 \cdot t(b \cdot p(g_1g_2))^{-1})\\
    &\underset{D(\qm)}{\approx}\frac{1}{\# B} \sum_{b \in B}b^{-1} \cdot \Big( \qm(t(b) \cdot g_1 \cdot t(b \cdot p(g_1))^{-1}) + \qm(t(b \cdot p(g_1)) \cdot g_2 \cdot t(b \cdot p(g_1g_2))^{-1}) \Big)\\
    &\underset{2D'(\qm)}{\approx} \frac{1}{\# B} \sum_{b \in B}\qm(g_1 \cdot t(b \cdot p(g_1))^{-1}\cdot t(b)) \\
    &\hspace{2cm} + \frac{1}{\# B} \sum_{b \in B}
    p(g_1)\cdot \qm(g_2 \cdot t(b \cdot p(g_1g_2))^{-1} \cdot t(b \cdot p(g_1)))\\
    &= F(g_1) + g_1 \cdot \left(\frac{1}{\# B} \sum_{b \in B}
    \qm\big( g_2 \cdot t\big( (b \cdot p(g_1)) \cdot p(g_2)\big)^{-1} \cdot t(b \cdot p(g_1))\big)\right).
  \end{align*}
  By the arguments in the proof of \cite[Proposition 6.4]{KKMM1}, we have
  \[
    \frac{1}{\# B} \sum_{b \in B} \qm\big( g_2 \cdot t\big( (b \cdot p(g_1)) \cdot p(g_2)\big)^{-1} \cdot t(b \cdot p(g_1))\big) = F(g_2).
  \]
  Therefore we have $F(g_1g_2) \underset{D(\qm) + 3D'(\qm)}{\approx} F(g_1) + g_1 \cdot F(g_2)$.
  This completes the proof.
\end{proof}

The counterpart of Theorem \ref{thm 7.4} in the case of the trivial real coefficient is an open problem.

\begin{problem}\label{prob:mcg}
  Let $\hG$ be a subgroup of finite index of $\Mod(\Sigma_l)$. Set $\bG = \hG \cap \II(\Sigma_{\genus})$ and $\ppi = \hG / \bG$. Then does $\QQQ(N)^G = i^{\ast}\QQQ(G)$ hold?
\end{problem}

In \cite{CHH}, Cochran, Harvey, and Horn constructed $\Mod(\Sigma)$-invariant quasimorphisms on $\II(\Sigma)$ for a surface $\Sigma$ with at least one boundary component. 
The problem asking whether their quasimophisms are extendable may be of special interest.

\section{Open problems}\label{open problem section}

\subsection{Mystery of the Py class}\label{Py class subsec}



Let $ \Sigma_\genus $ be a closed connected orientable surface whose genus $\genus$ is at least $2$ and $\Omega$ a volume form on $ \Sigma_\genus $.
Recall that Py \cite{Py06} constructed a Calabi quasimorphism $f_P$ on
 $\Ker(\flux_\Omega)$ 
which is $\diff_0( \Sigma_\genus ,\Omega)$-invariant, and the first and second authors showed that $f_P$ is not extendable to $\diff_0( \Sigma_\genus ,\Omega)$  (recall Section~\ref{flux section} and Example~\ref{ex:KKMM}).
We define $\bar{c}_P\in\HHH^2(\HHH^1( \Sigma_\genus ))$ and $c_P\in\HHH^2\left(\diff_0( \Sigma_\genus ,\Omega)\right)$ by $\bar{c}_P=\tae_4^{-1} \circ \tau_{/b}(\qm_P)$ and $c_P=\flux_\Omega^\ast(\bar{c}_P)$, respectively.
We call $c_P$ the \textit{Py class}.
Note that we essentially proved the non-triviality of the Py class in the proof of (1) of Theorem \ref{thm diff}.

When we construct the class $\bar{c}_P\in\HHH^2(\HHH^1( \Sigma_\genus ))$, we used the morphism $\tae_4\colon \HHH^2( \HHH^1(\Sigma_\genus )) \to \HHH_{/b}^2( \HHH^1(\Sigma_\genus ))$.
 Here we apply the exact sequence
\[ 1 \to \Ker(\flux_\Omega) \to \diff_0( \Sigma_\genus ,\Omega) \xrightarrow{\flux_\Omega} \HHH^1(\Sigma_\genus) \to 1 \]
to diagram \eqref{diagram_coh_qm_rel}.  
Since the bounded cohomology groups of an  amenable group are zero, the map $\tae_4$ is an isomorphism and we see that there exists the inverse $\tae_4^{-1}\colon \HHH_{/b}^2( \HHH^1(\Sigma_\genus )) \to \HHH^2( \HHH^1(\Sigma_\genus ))$ of $\tae_4$.
Because the vanishing of the bounded cohomology of  amenable groups is shown by a transcendental method, we do not have a precise description of the map $\tae_4^{-1}$.

\begin{remark}
If we fix a (right-)invariant mean $m$ on the  amenable group $\ppi$, then we have the following description of the map $\tae_4^{-1}$. For a cocycle $[c] \in C_{/b}^2(\ppi)$, a  cocycle $f \in C^2(\ppi)$ representing the class $\tae_4^{-1}( \bigl[ [c] \bigr])$ can be given by the following formula:
  \[
    f(\gamma_1, \gamma_2) = c(\gamma_1, \gamma_2) - m(\delta c(\, \cdot \, , \gamma_1, \gamma_2)).
  \]
\end{remark}


However, we have the following observations on the Py class. Here we consider $\HHH^1( \Sigma_\genus )$ as a symplectic vector space by the intersection form.

\begin{thm}\label{Property of Py class}
Let $ \Sigma_\genus $ be a closed connected orientable surface whose genus $\genus$ is at least $2$ and $\Omega$ a volume form on $ \Sigma_\genus $.  For a subgroup $\Lambda$ of $\HHH^1( \Sigma_\genus )$, let $\iota_\Lambda\colon\Lambda\to \HHH^1( \Sigma_\genus )$ be the inclusion map.
\begin{itemize}
 \item[$(1)$]  Let $v$ and $w$ be elements in $\HHH^1( \Sigma_\genus )$ with $v\smile w\ne 0$. Then there exists a positive integer $k_0$ such that for every integer $k$ at least $k_0$, the following holds: for the subgroup $\Lambda =\langle v,\frac{w}{k}\rangle$ of $\HHH^1( \Sigma_\genus )$, we have $\iota_\Lambda^\ast\bar{c}_P\neq0$. Here, $\smile$ denotes the cup product.
\item[$(2)$] If  a subgroup $\Lambda$ of $\HHH^1( \Sigma_\genus )$  is contained in linear subspaces $\langle[\alpha_1]^\ast,\ldots,[\alpha_\genus]^\ast\rangle$ or $\langle[\beta_1]^\ast,\ldots,[\beta_\genus]^\ast\rangle$, then $\iota_\Lambda^\ast\bar{c}_P=0$, where $\alpha_1,\ldots,\alpha_\genus,\beta_1,\ldots,\beta_\genus$ are curves described in Figure $\ref{curves}$.
\end{itemize}
\end{thm}

\begin{figure}[h]
\centering
\includegraphics[width=10truecm]{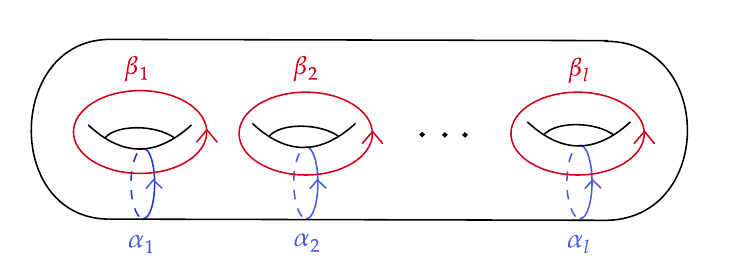}
\caption{$\alpha_1,\ldots,\alpha_\genus,\beta_1,\ldots,\beta_\genus\colon[0,1]\to  \Sigma_\genus $}
\label{curves}
\end{figure}

To prove Theorem \ref{Property of Py class}, we use the following observation.

Let $1 \to \bG \xrightarrow{i} \hG \xrightarrow{p} \ppi \to 1$ be an exact sequence of groups such that $\ppi$ is  amenable.
For a subgroup $\ppi^{0}$ of $\ppi$, $1 \to \bG \xrightarrow{i} p^{-1}(\ppi^{0}) \xrightarrow{p} \ppi^{0} \to 1$ is also an exact sequence and it is known that $\ppi^{0}$ is also  amenable ((3) of Theorem \ref{amenable base}).

Then, by Theorem \ref{main thm}, we have the following commuting diagrams.

\begin{align} \label{diag:gamma}
  \xymatrix{
  0 \ar[r] & \HHH^1(\ppi) \ar[r]^-{p^*} \ar[d]^-{\tae_1} & \HHH^1(\hG) \ar[r]^-{i^*} \ar[d]^-{\tae_2} & \HHH^1(\bG)^{ \hG } \ar[r]^-{\tau} \ar[d]^-{\tae_3} & \HHH^2(\ppi) \ar[r]^-{p^*} \ar[d]^-{\tae_4} & \HHH^2(\hG) \ar[d]^-{\tae_5} \\
  0 \ar[r] & \QQQ(\ppi) \ar[r]^-{p^*} &  \QQQ(\hG) \ar[r]^-{i^*} &  \QQQ(\bG)^{ \hG } \ar[r]^-{\tau_{/b}} & \HHH_{/b}^2(\ppi) \ar[r]^-{p^*} & \HHH_{/b}^2(\hG),
  }
  \end{align}

\begin{align} \label{diag:gamma0}
  \xymatrix{
  0 \ar[r] & \HHH^1(\ppi^{0}) \ar[r]^-{p^*} \ar[d]^-{\tae_1^{0}} & \HHH^1(p^{-1}(\ppi^{0})) \ar[r]^-{i^*} \ar[d]^-{\tae_2^{0}} & \HHH^1(\bG)^{ p^{-1}(\ppi^0) } \ar[r]^-{\tau^{0}} \ar[d]^-{\tae_3^{0}} & \HHH^2(\ppi^{0}) \ar[r]^-{p^*} \ar[d]^-{\tae_4^{0}} & \HHH^2(p^{-1}(\ppi^{0})) \ar[d]^-{\tae_5^{0}} \\
  0 \ar[r] & \QQQ(\ppi^{0}) \ar[r]^-{p^*} &  \QQQ(p^{-1}(\ppi^{0})) \ar[r]^-{i^*} &  \QQQ(\bG)^{ p^{-1}(\ppi^0) } \ar[r]^-{\tau_{/b}^{0}} & \HHH_{/b}^2(\ppi^{0}) \ar[r]^-{p^*} & \HHH_{/b}^2(p^{-1}(\ppi^{0})).
  }
  \end{align}
Since $\ppi$ and $\ppi^{0}$ are  boundedly $3$-acyclic ((5) of Theorem \ref{amenable base}), $\tae_4\colon \HHH^2(\ppi)\to \HHH_{/b}^2(\ppi)$ and $\tae_4^{0}\colon \HHH^2(\ppi^{0})\to \HHH_{/b}^2(\ppi^{0})$ are isomorphisms .
The following lemma is deduced from the definitions of $\tau_{/b}$ and $\tau_{/b}^0$.

\begin{lem} \label{functoriality of embedding}
\[(\tae_4^{0})^{-1} \circ \tau_{/b}^{0} \circ I_1^\ast = I_2^\ast  \circ (\tae_4)^{-1} \circ \tau_{/b}  ,\]
where $I_1^\ast \colon \QQQ(\bG)^{ \hG } \to \QQQ(\bG)^{ p^{-1}(\ppi^{0})}$, $I_2^\ast \colon \HHH^2(\ppi) \to \HHH^2(\ppi^{0})$ are the homomorphisms induced from the inclusion $I\colon\ppi^{0}\to\ppi$.
\end{lem}

 We employ the following theorem, which is related to Example~\ref{ex:KKMM}, in order to prove Theorem \ref{Property of Py class}.

\begin{thm}[{\cite[Theorem~1.6 and Theorem~1.10]{KKMM2}}] \label{KKMM flux thm}
Let $ \Sigma_\genus $ be a closed connected orientable surface whose genus $\genus$ is at least $2$ and $\Omega$ a volume form on $ \Sigma_\genus $.
 Let $\Lambda$ be a subgroup of $\HHH^1( \Sigma_\genus )$ and set $G=\flux^{-1}(\Lambda)$ and $N=\Ker(\flux_\Omega)$.
Then,  the following hold true.
\begin{itemize}
\item[$(1)$]  Let $v$ and $w$ be elements in $\HHH^1( \Sigma_\genus )$ with $v\smile w\ne 0$. Then there exists a positive integer $k_0$ such that for every integer $k$ at least $k_0$, the following holds for $\Lambda=\langle v,\frac{w}{k}\rangle$: $[\qm_P]$ is a non-trivial element of $\QQQ(\bG)^\hG/i^\ast\QQQ(\hG)$.

\item[$(2)$] If $\Lambda$ is contained in linear subspaces $\langle[\alpha_1]^\ast,\ldots,[\alpha_\genus]^\ast\rangle$ or $\langle[\beta_1]^\ast,\ldots,[\beta_\genus]^\ast\rangle$ ,  then $[\qm_P]$ is the trivial element of $\QQQ(\bG)^\hG/i^\ast\QQQ(\hG)$, where $\alpha_1,\ldots,\alpha_\genus,\beta_1,\ldots,\beta_\genus$ are curves described in Figure $\ref{curves}$.
\end{itemize}
\end{thm}

 On (2), see also Remark~4.8 of \cite{KKMM2}.

\begin{proof}[Proof of Theorem $\ref{Property of Py class}$]
Set $\Gamma=\HHH^1( \Sigma_\genus )$, $\Gamma^{0}=\Lambda$ and $G=\flux_\Omega^{-1}(\Lambda)$.
We use the notation in diagrams \eqref{diag:gamma} and \eqref{diag:gamma0}.

First, to prove (1), suppose that the dimension of $\Lambda$ is larger than $\genus$.
Then, since $[\qm_P]$ is a non-trivial element of $\QQQ(\bG)^\hG/i^\ast\QQQ(\hG)$,
by Theorem \ref{easy cor},
 $(\tae_4^{0})^{-1} \circ \tau_{/b}^{0} \circ I_1^\ast(\qm_P)$ is also a non-trivial element of $\HHH^2(\Gamma^{0})=\HHH^2(\Lambda)$.
Hence, by Lemma \ref{functoriality of embedding}, $\iota_\Lambda^\ast\bar{c}_P = I_2^\ast  \circ (\tae_4)^{-1} \circ \tau_{/b}(\qm_P)=(\tae_4^{0})^{-1} \circ \tau_{/b}^{0} \circ I_1^\ast (\qm_P)$ is also a non-trivial element of $\HHH^2(\Gamma^{0})=\HHH^2(\Lambda)$.

Next, to prove (2), suppose that  $\Lambda$ is contained in linear subspaces $\langle[\alpha_1]^\ast,\ldots,[\alpha_\genus]^\ast\rangle$ or $\langle[\beta_1]^\ast,\ldots,[\beta_\genus]^\ast\rangle$.
Then, since $[\qm_P]$ is the trivial element of $\QQQ(\bG)^\hG/i^\ast\QQQ(\hG)$,
by Theorem \ref{easy cor} and Proposition \ref{prop diff},
 $(\tae_4^{0})^{-1} \circ \tau_{/b}^{0} \circ I_1^\ast(\qm_P)$ is also the trivial element of $\HHH^2(\Gamma^{0})=\HHH^2(\Lambda)$.
Hence, by Lemma \ref{functoriality of embedding}, $\iota_\Lambda^\ast\bar{c}_P = I_2^\ast  \circ (\tae_4)^{-1} \circ \tau_{/b}(\qm_P)=(\tae_4^{0})^{-1} \circ \tau_{/b}^{0} \circ I_1^\ast (\qm_P)$ is also the trivial element of $\HHH^2(\Gamma^{0})=\HHH^2(\Lambda)$.
\end{proof}

Finally, we pose the following problems on the Py class.

\begin{problem}
Give precise descriptions of a cochain representing $\bar{c}_P\in\HHH^2\left(\HHH^1( \Sigma_\genus )\right)$ and a bounded cochain representing $c_P\in\HHH^2\left(\diff_0(M,\Omega)\right)$.
\end{problem}

\begin{problem}\label{uniueness cp}
Let $ \Sigma_\genus $ be a closed connected orientable surface whose genus $\genus$ is at least $2$ and $\Omega$ a volume form on $ \Sigma_\genus $.
Is the vector space $\Im(\flux_\Omega^\ast) \cap \Im(c_{\diff_0( \Sigma_\genus ,\Omega)})$ spanned by $c_P$?
\end{problem}

By Theorem \ref{easy cor}, Problem \ref{uniueness cp} is rephrased as follows.

\begin{problem}\label{uniueness cp2}
Let $ \Sigma_\genus $ be a closed connected orientable surface whose genus $\genus$ is at least $2$ and $\Omega$ a volume form on $ \Sigma_\genus $.
Is the vector space $\QQQ\left(\Ker(\flux_\Omega)\right)^{\diff_0( \Sigma_\genus ,\Omega)}/i^{\ast}\QQQ\left(\diff_0( \Sigma_\genus ,\Omega)\right)$ spanned by $[\qm_P]$?
\end{problem}

\subsection{Problems on equivalences and coincidences of $\scl_\hG$ and $\scl_{\hG,\bG}$}\label{subsec:equivalence}

By Theorem \ref{main thm 3.1}, $\QQQ(N)^G = \HHH^1(N)^G + i^* \QQQ(G)$ implies that $\scl_\hG$ and $\scl_{\hG,\bG}$ are equivalent on $[G,N]$.
Moreover, if $N$ is the commutator subgroup of $G$ and $\QQQ(\bG)^\hG = \HHH^1(\bG)^\hG + i^* \QQQ(\hG)$, then $\scl_\hG$ and $\scl_{G,N}$ coincide on $[G,N]$.
Since $\HHH^2(G) = 0$ implies $\QQQ(N)^G = \HHH^1(N)^G + i^* \QQQ(G)$ (Theorem \ref{easy cor}), 
there are several examples of pairs $(G,N)$ such that $\scl_{G,N}$ and $\scl_G$ are equivalent (see Subsection \ref{equiv subsection}).
In Section 3, we provided several examples of groups $G$ with $\QQQ(N)^G \ne \HHH^1(N)^G + i^*\QQQ(G)$ (see Theorems \ref{thm surface group}, \ref{mapping torus thm}, and \ref{thm:one-relator}), but  in this paper, we are unable to determine whether
$\scl_G$ and $\scl_{G,N}$ are equivalent on $[G,N]$
 in these examples.
Hence, the example that $G = \diff(\Sigma_l, \omega)$ with $l \ge 2$ and $N = [G,G]$ raised by \cite{KK}  (see also \cite{KKMM2}) has remained the  essentially only one known example that $\scl_G$ and $\scl_{G,N}$ are not equivalent on $[G,N]$. In fact, this is the only one example that $\scl_G$ and $\scl_{G,N}$ do not coincide on $[G,N]$. Here, we provide several problems on equivalences and coincidences of $\scl_{\hG}$ and $\scl_{\hG,\bG}$.


\begin{problem}\label{scl onaji}
Is it true that $\QQQ(\bG)^{\hG} = \HHH^1(\bG)^{\hG} + i^* \QQQ(\hG)$ implies that $\scl_\hG = \scl_{\hG,\bG}$ on $[\hG,\bG]$?
\end{problem}

\begin{problem}\label{scl not equiv}
Find a pair $(\hG,\bG)$ such that $\hG$ is finitely generated and $\scl_\hG$ and $\scl_{\hG,\bG}$ are not equivalent.
In particular, for $\genus \ge 2$, are $\scl_{\pi_1(\Sigma_\genus)}$ and $\scl_{\pi_1(\Sigma_\genus), [\pi_1(\Sigma_\genus),\pi_1(\Sigma_\genus)]}$ equivalent  on $[\pi_1(\Sigma_\genus), [\pi_1(\Sigma_\genus),\pi_1(\Sigma_\genus)]]$?
\end{problem}

 After the current work, Problem \ref{scl not equiv} was solved by some of the authors \cite{MMM};
they proved that $\scl_{\pi_1(\Sigma_\genus)}$ and $\scl_{\pi_1(\Sigma_\genus), [\pi_1(\Sigma_\genus),\pi_1(\Sigma_\genus)]}$ are not equivalent for $\genus \ge 2$.
Moreover, the authors
\cite{coarse_group} proved that $\scl_\hG$ and $\scl_{\hG,[\hG,\hG]}$ are not equivalent if $\QQQ([\hG,\hG])^{\hG} \neq \HHH^1([\hG,\hG])^{\hG} + i^* \QQQ(\hG)$.

We also pose the following problem.
Let $B_n$ be the $n$-th braid group and $P_n$ the $n$-th pure braid group.

\begin{problem}\label{pnpn}
For $n \ge 3$, does $\scl_{B_n}=\scl_{B_n, [P_n,P_n]}$ hold on $\left[B_n, [P_n,P_n]\right]$?
\end{problem}

From the aspect of the following proposition, we can regard Problem \ref{pnpn} as a special case of Problem \ref{scl onaji}.

\begin{prop}\label{pnpn prop}
For $n \ge 2$, let $\hG=B_n$ and $\bG=[P_n,P_n]$. Then $\QQQ(\bG)^{\hG} = \HHH^1(\bG)^{\hG} + i^* \QQQ(\hG)$.
 In particular,  $\scl_\hG(x) \le \scl_{\hG, \bG}(x) \le 2 \cdot \scl_{\hG}(x)$ holds for all $x \in [\hG, \bG]$.
\end{prop}
\begin{proof}
Consider the exact sequence
\[1 \to P_n/[P_n,P_n] \to B_n/[P_n,P_n] \to \mathfrak{S}_n \to 1,\]
where $\mathfrak{S}_n$ is the symmetric group.
By (1) and (2) of Theorem \ref{amenable base}, $\mathfrak{S}_n$ and $P_n/[P_n,P_n]$ are  amenable.
Hence (4) of Theorem \ref{amenable base} implies that $B_n/[P_n,P_n]$ is also  amenable.
As pointed out in Subsection \ref{equiv subsection}, the second cohomology of the braid group $B_n$ vanishes.
Hence Theorem \ref{easy cor} implies that $\QQQ(\bG)^{\hG} = \HHH^1(\bG)^{\hG} + i^* \QQQ(\hG)$. The equivalence between $\scl_{B_n}$ and $\scl_{B_n, [P_n, P_n]}$ follows from (2) of Theorem \ref{main thm 3.1}.
\end{proof}

As another special case of Problem \ref{scl onaji}, we provide the following problem.

\begin{problem}\label{aut scl onaji}
For $n \ge 2$, does $\scl_{\Aut(F_n)}=\scl_{\Aut(F_n),\IA_n}$ hold on $\left[\Aut(F_n),\IA_n\right]$?
\end{problem}

 Even the following weaker variant of Problem~\ref{aut scl onaji} seems  open . We note that (2) of Theorem~\ref{main thm 3.1} does \emph{not} apply to the setting of Theorem~\ref{thm 3.2}.

\begin{problem}\label{problem:scl_C}
Let $n\geq 3$. Find an explicit real constant $C\geq 1$ such that $\scl_{\Aut(F_n),\IA_n}\leq C \cdot \scl_{\Aut(F_n)}$ holds on $\left[\Aut(F_n),\IA_n\right]$.
\end{problem}



In \cite{KKMM1}, the first, second, fourth, and fifth authors considered the equivalence problem between  $\cl_G$ and $\cl_{G,N}$.
We provide the following problem.

\begin{problem}\label{cl prob}
Is it true that $\QQQ(N)^G = \HHH^1(N)^G + i^* \QQQ(G)$ implies the bi-Lipschitz equivalence of  $\cl_G$ and $\cl_{G,N}$ on $[G,N]$?
\end{problem}

We note that (1) of Theorem \ref{main thm 3.1} states that $\QQQ(N)^G = \HHH^1(N)^G + i^* \QQQ(G)$ implies the bi-Lipschitz equivalence of  $\scl_G$ and $\scl_{G,N}$.
To the best knowledge of the authors, Problem \ref{cl prob}, even for the case where $1\to \bG \to \hG \to \ppi \to 1$ virtually splits, might be open in general.

From the aspect of Proposition \ref{pnpn prop} and Theorem \ref{thm 3.2}, we can regard the following problem as special cases of Problem \ref{cl prob}.
\begin{problem}
For $(\hG,\bG)=\left(B_n, [P_n,P_n]\right) (n \ge 3),\left(\Aut(F_n),\IA_n\right)  (n \ge 2)$, are $\cl_{\hG}$ and $\cl_{\hG,\bG}$ equivalent  on $[\hG,\bG]$?
\end{problem}

We note that $\cl_{\hG}$ and $\cl_{\hG,\bG}$ are known to be bi-Lipschitzly equivalent when $(\hG,\bG)=(B_n, P_n\cap[B_n,B_n]=[P_n,B_n])$ (\cite{KKMM1}).

\subsection{A question by De Chiffre, Glebsky, Lubotzky and Thom}\label{subsec:2-Kazhdan}

In Definition~4.1 of \cite{CGLT}, De Chiffre, Glebsky, Lubotzky and Thom introduced the following property.

\begin{definition}[{\cite{CGLT}}]
Let $n$ be a positive integer. A discrete group $\Gamma$ is said to be \emph{$n$-Kazhdan} if for every unitary $\Gamma$-representation $(\varpi,\mathcal{K})$, $\HHH^n(\Gamma;\varpi,\mathcal{K})=0$ holds.
\end{definition}

The celebrated Delorme--Guichardet theorem states that for a finitely generated group, the $1$-Kazhdan property is equivalent to Kazhdan's property (T); see \cite{BHV} for details.

In \cite[Question~4.4]{CGLT}, they asked the following question.

\begin{problem}[{\cite{CGLT}}]\label{problem:2-Kazhdan}
Is $\SL(n,\ZZ)$ $2$-Kazhdan for $n\geq 4$? Or weakly, does there exist $n_1\geq 4$ such that for all $n\geq n_1$, $\SL(n,\ZZ)$ is $2$-Kazhdan?
\end{problem}


The motivation of De Chiffre, Glebsky, Lubotzky and Thom to study the $2$-Kazhdan property is the stability on group approximations by finite dimensional unitary groups with respect to the Frobenius norm; see \cite{CGLT} and also \cite{Thom}. The present work shows that the $2$-Kazhdan property furthermore relates to the space of non-extendable quasimorphisms with non-trivial coefficients. For example, the positive solution to Problem~\ref{problem:2-Kazhdan} will provide a generalization of Theorem~\ref{thm 7.1.2} for \emph{all} unitary representations, including infinite dimensional ones. The following proposition gives the precise statement.

\begin{prop}\label{prop:2-Kazhdan}
Fix  an integer $n$ with $n\geq 4$.
Assume that $\SL(n,\ZZ)$ is $2$-Kazhdan.
Then, for every subgroup $\hG$ of finite index of $\Aut(F_n)$, and for every unitary representation $\pi$ of $\ppi$, the equality
\[\rQQQ(\bG ; \HH)^{\QQQ \hG} = i^* \rQQQ Z(\hG ; \overline{\pi}, \HH)\]
holds. Here we set $\bG = \hG \cap \IA_n$ and $\ppi = \hG / \bG$; the representation $(\overline{\pi}, \HH)$ of $\hG$ is the pull-back of the representation $(\pi, \HH)$ of $\ppi$.
\end{prop}

\begin{proof}
By Theorem~\ref{main thm} and exact sequence \eqref{seq:long_seq}, it suffices to prove that $\HHH^2(\ppi;\pi,\HH)=0$ and that $\HHH^3_b(\ppi;\pi,\HH)=0$.
Here, recall Remark~\ref{remark:isomQZ}. Note that $\ppi$ is a subgroup of finite index of $\GL(n,\ZZ)$.

First, we will verify that $\HHH^3_b(\ppi;\pi,\HH)=0$.
Set $\ppi_0:=\ppi \cap \SL(n,\ZZ)$ and $\pi_0:=\pi\mid_{\ppi_0}$.
Then $\ppi_0$ is a subgroup of finite index of $\SL(n,\ZZ)$, and $(\pi_0,\HH)$ is a unitary $\ppi_0$-representation. Decompose the $\ppi_0$-representation space $\HH$ as $\HH=\HH^{\ppi_0}\oplus (\HH^{\ppi_0})^{\perp}$, where $(\HH^{\ppi_0})^{\perp}$ is the orthogonal complement of $\HH^{\ppi_0}$ in $\HH$.
Then, the restriction $\pi_{0}^{\mathrm{inv}}$ of $\pi_0$ on $\HH^{\ppi_0}$ is trivial, and the restriction $\pi_{0}^{\mathrm{orth}}$ of $\pi_0$ on $(\HH^{\ppi_0})^{\perp}$ does not admit a non-zero $\ppi_0$-invariant vector.
Theorem~\ref{thm:3bdd_coh_vanish_SL} (Monod's theorem) implies that $\HHH^3_b(\ppi_0;\pi_0^{\mathrm{inv}},\HH^{\ppi_0})=0$.
By another theorem of Monod \cite[Theorem~2]{Monod07}, we also have $\HHH^3_b(\ppi_0;\pi_0^{\mathrm{orth}},(\HH^{\ppi_0})^{\perp})=0$. (See Corollary~1.6 of \cite{Monod2} for a more general statement.)
Hence, we have $\HHH^3_b(\ppi_0;\pi_0,\HH)=0$. Now, Theorem~\ref{thm:transfer} implies that $\HHH^3_b(\ppi;\pi,\HH)=0$, as desired.

Finally, we will prove $\HHH^2(\ppi;\pi,\HH)=0$ under the assumption of the theorem. The Shapiro lemma (for group cohomology) implies that the $2$-Kazhdan property passes to a group of finite index. In what follows, we sketch the deduction above. Let $H_0$ be a subgroup of a group $H$ of finite index. Take an arbitrary unitary $H_0$-representation $(\sigma,\mathfrak{H})$. Then since $H_0$ is of finite index in $H$, the coinduced module $\mathrm{Coind}_{H_0}^H(\mathfrak{H})$ is canonically isomorphic to the induced module $\mathrm{Ind}_{H_0}^H(\mathfrak{H})$.  Therefore, the Shapiro lemma shows that $\HHH^2(H_0;\sigma,\mathfrak{H})\cong \HHH^2(H;\varsigma,\mathfrak{K})$.
Note that the induced representation $(\varsigma,\mathfrak{K})=(\mathrm{Ind}_{H_0}^H(\sigma),\mathrm{Ind}_{H_0}^H(\mathfrak{H}))$ is a unitary $H$-representation. Hence, if $H$ is $2$-Kazhdan, then $\HHH^2(H;\varsigma,\mathfrak{K})=0$ holds; it then follows that $H_0$ is $2$-Kazhdan.
Also, a standard argument using the transfer shows the following: if $H_0$ is a subgroup of $\tilde{H}$ of finite index and if $H_0$ is $2$-Kazhdan, then $\tilde{H}$ is $2$-Kazhdan. (See Proposition~4.4 of \cite{CGLT} for a more general statement.) Recall that $\ppi_0$ is a subgroup of finite index of $\SL(n,\ZZ)$, and that $\ppi_0$ is a subgroup of $\ppi$ of index at most $2$. Therefore, since we assume that $\SL(n,\ZZ)$ is $2$-Kazhdan,  we conclude that $\ppi$ is $2$-Kazhdan. Hence we have $\HHH^2(\ppi;\pi,\HH)=0$, and this completes the proof.
\end{proof}

A counterpart of Proposition~\ref{prop:2-Kazhdan} in the setting of mapping class groups can be stated in the following manner.
Proposition~\ref{prop:Mod2-Kazhdan} asserts that under a certain assumption, Theorem~\ref{thm 7.4} may be extended to infinite dimensional cases.

\begin{prop}\label{prop:Mod2-Kazhdan}
Fix  an integer $l$ with $l\geq 3$.
Assume that for every unitary $\Sp(2l,\ZZ)$-representation $(\varpi,\mathcal{K})$ with $\varpi \not \supset 1$, $\HHH^2(\Sp(2l,\ZZ);\varpi,\mathcal{K})=0$ holds. Then, for every subgroup $\hG$ of finite index of $\Mod(\Sigma_l)$, and for every unitary representation $\pi$ of $\ppi$ \emph{with $\pi \not \supset 1$}, the equality
\[\rQQQ(\bG ; \HH)^{\QQQ \hG} = i^* \rQQQ Z(\hG ; \overline{\pi}, \HH)\]
holds. Here we set $\bG = \hG \cap \mathcal{I}(\Sigma_l)$ and $\ppi = \hG / \bG$; the representation $(\overline{\pi}, \HH)$ of $\hG$ is the pull-back of the representation $(\pi, \HH)$ of $\ppi$.
\end{prop}

\begin{proof}
Since $\pi \not \supset 1$, Theorem~\ref{thm:3bdd_vanish_Sp} (Monod's theorem)
 shows that $\HHH^3_b(\ppi;\pi,\HH)=0$.
 In addition, since $\pi \not \supset 1$, the induced unitary $\Sp(2l,\ZZ)$-representation $\mathrm{Ind}_{\ppi}^{\Sp(2l,\ZZ)}(\pi)$ does not admit a non-zero invariant vector. Therefore, by the assumption on $\Sp(2l,\ZZ)$, the Shapiro lemma implies that $\HHH^2(\ppi;\pi,\HH)=0$. Now Theorem~\ref{main thm}, together with exact sequence \eqref{seq:long_seq} and Remark~\ref{remark:isomQZ}, ends the proof.
\end{proof}

In relation to Proposition~\ref{prop:Mod2-Kazhdan}, the following problem may be of interest.

\begin{problem}\label{problem:Mod2-Kazhdan}
Does there exist $l_1\geq 3$  satisfying 
 the following:  for all $l\geq l_1$, for every unitary $\Sp(2l,\ZZ)$-representation $(\varpi,\mathcal{K})$ with $\varpi \not \supset 1$, $\HHH^2(\Sp(2l,\ZZ);\varpi,\mathcal{K})=0$ holds? 
\end{problem}

We note that $\HHH^2(\Sp(2l,\ZZ))=\RR$ for every $l\geq 2$; see \cite{Borel74}. In particular, $\Sp(2l,\ZZ)$ is \emph{not} $2$-Kazhdan for any $l\geq 2$.

\begin{remark}
 From Corollary~\ref{cor 7.11}, we have the following: if we impose an additional condition on $\varpi$ that it is \emph{finite dimensional} in the settings of Problem~\ref{problem:2-Kazhdan} and Problem~\ref{problem:Mod2-Kazhdan}, then we may take $n_1=6$ and $l_1=3$, respectively.

 It is also a question of interest asking what the best bound of $n$ in (1) of Theorem \ref{thm 7.9 new} is. The bound we have is `$n\geq 6$.'
Is it possible to improve this bound to `$n\geq 4$?'
\end{remark}

\section*{Acknowledgment}
The authors thank Toshiyuki Akita, Tomohiro Asano, Masayuki Asaoka, Nicolas Monod, Yuta Nozaki, Takuya Sakasai, Masatoshi Sato, and Takao Satoh for fruitful discussions on applications of results of the present paper and for providing several references.
The fifth author is grateful to  Andrew Putman and Bena Tshishiku for the bound of  $n$ in (1) of Theorem \ref{thm 7.9 new}, and Masatoshi Sato for the discussion on (2) of Corollary \ref{cor 7.11}.

The first author is supported in part by JSPS KAKENHI Grant Number JP18J00765 and JP21K13790.
The second author and the third author are supported by JSPS KAKENHI Grant Number JP20H00114 and JP21J11199, respectively.
The fourth author is partially supported by JSPS KAKENHI Grant Number JP19K14536.
The fifth author is partially supported by JSPS KAKENHI Grant Number  JP17H04822 and JP21K03241.

\appendix

\section{Other exact sequences related to $\QQQ(N)^G / (\HHH^1(N)^G + i^* \QQQ(G))$}
In this appendix, we show some exact sequences which are related to the quotient space $\QQQ(N)^G / (\HHH^1(N)^G + i^* \QQQ(G))$ and the seven-term exact sequence, and show that these sequences give alternative proofs of some results (Theorem \ref{thm surface group}, \ref{main thm 2}, \ref{easy cor} and \ref{mapping torus thm}) in this paper.

The seven-term exact sequence (Theorem \ref{thm seven-term}) is one of the main tools in this appendix.
In particular, we will use the following cocycle description of the map $\rho$ in the seven-term exact sequence.

\begin{thm}[Section 10.3 of \cite{DHW}] \label{thm DHW}
Let $c \in \Ker (i^* \colon \HHH^2(G) \to \HHH^2(N))$, and let $f$ be a 2-cocycle on $G$ satisfying $f|_{N \times N} = 0$ and $[f] = c$. Then
\[\big( \rho (c) (p(g))\big) (n) =  f(g, g^{-1} ng) - f(n,g),\]
where $g \in \hG$ and $n \in \bG$.
\end{thm}

Let $\EH_b^2$ denote the kernel of the comparison map $\HHH_b^2 \to \HHH^2$.
We are now ready to state our main results in this appendix:

\begin{thm} \label{main thm 1}
Let $G$ be a group, $N$ a normal subgroup of $G$, and $\ppi$ the quotient $G/N$. Then the following hold:
\begin{itemize}
\item[$(1)$] There exists the following exact sequence
\[0 \to \QQQ(\bG)^{\hG} / (\HHH^1(\bG)^{\hG} + i^* \QQQ(\hG)) \to \EH^2_b(N)^G / i^* \EH^2_b(G) \xrightarrow{\alpha}  \HHH^1(G ; \HHH^1(N)).\]

\item[$(2)$] There exists the following exact sequence
\[\HHH^2_b(\ppi) \to \Ker(i^*) \cap \Im(c_G) \xrightarrow{\beta} \EH^2_b(N)^G / i^* \EH^2_b(G) \to \HHH^3_b(\ppi).\]
Here $i^*$ is the map $\HHH^2(G) \to \HHH^2(N)$ induced by the inclusion $N \hookrightarrow G$, and $c_G \colon \HHH^2_b(G) \to \HHH^2(G)$ is the comparison map.

\item[$(3)$] The following diagram is commutative:
\[\xymatrix{
\Ker(i^*) \cap \Im(c_G) \ar[r]^-{j} \ar[d]_{\beta} & \Ker(i^*) \ar[r]^-{\rho} & \HHH^1(\ppi ; \HHH^1(N)) \ar[d] \\
\EH^2_b(N)^G / i^* \EH^2_b(G) \ar[rr]^-{\alpha} & {} & \HHH^1(G ; \HHH^1(N)).
}\]
Here $j$ is an inclusion, and $\alpha$, $\beta$, and $\rho$ are the maps appearing in $(1)$, $(2)$, and the seven-term exact sequence, respectively.
\end{itemize}
\end{thm}

From (1) and (2) of Theorem \ref{main thm 1}, we obtain the following:
\begin{cor}\label{cor:Appendix}
If $G/N$ is  boundedly $3$-acyclic, there exists the following exact sequence
\[0 \to \QQQ(\bG)^{\hG} / (\HHH^1(\bG)^{\hG} + i^* \QQQ(\hG)) \to \Ker(i^*) \cap \Im(c_G) \to \HHH^1(G ; \HHH^1(N)).\]
\end{cor}

\begin{proof}[Proof of $(1)$ of Theorem $\ref{main thm 1}$]
Recall that $\EH^2_b(G)$ is the kernel of the comparison map $c_G \colon \HHH^2_b(G) \to \HHH^2(G)$. By Lemma \ref{lem 3.1}, $\EH^2_b(G)$ coincides with the image of $\delta \colon \QQQ(G) \to \HHH^2_b(G)$. Therefore we have a short exact sequence
\begin{eqnarray} \label{eqn 3.1}
0 \to \HHH^1(G) \to \QQQ(G) \to \EH^2_b(G) \to 0.
\end{eqnarray}
For a $G$-module $V$, we write $V^G$ the subspace consisting of the elements of $V$ which are fixed by every element of $G$. Since the functor $(-)^G$ is a left exact and its right derived functor is $V \mapsto \HHH^\bullet(G ; V)$, we have an exact sequence
\begin{eqnarray} \label{eqn 3.2}
0 \to \HHH^1(N)^G \to \QQQ(N)^G \to \EH^2_b(N)^G \to \HHH^1(G ; \HHH^1(N)).
\end{eqnarray}
Thus we have the following commutative diagram
\begin{eqnarray} \label{eqn 3.3}
\xymatrix{
0 \ar[r] & \HHH^1(G) \ar[r] \ar[d]_{i^*} & \QQQ(G) \ar[r] \ar[d]^{i^*} & \EH^2_b(G) \ar[r] \ar[d]^{i^*} & 0 \ar[d] \\
0 \ar[r] & \HHH^1(N)^G \ar[r] & \QQQ(N)^G \ar[r] & \EH^2_b(N)^G \ar[r] & \HHH^1(G ; \HHH^1(N)).
}
\end{eqnarray}
Taking cokernels of the vertical maps, we have a sequence
\begin{align} \label{eqn 3.4}
\HHH^1(N)^G /i^* \HHH^1(G) \to \QQQ(N)^G / i^* \QQQ(G) \to \EH^2_b(N)^G / i^* \EH^2_b(G) \to \HHH^1(G ; \HHH^1(N)).
\end{align}
The exactness of the first three terms of this sequence follows from the snake lemma. The exactness of the last three terms
 can be checked by the diagram chasing. Since the cokernel of $\HHH^1(N)^G / i^* \HHH^1(G) \to \QQQ(N)^G / i^* \QQQ(G)$ is $\QQQ(N)^G / (i^* \QQQ(G) + \HHH^1(N)^G)$, we have an exact sequence
\begin{align} \label{eqn 3.5}
0 \to \QQQ(N)^G / (i^* \QQQ(G) + \HHH^1(N)) \to \EH^2_b(N)^G / i^* \EH^2_b(G) \to \HHH^1(G ; \HHH^1(N)).
\end{align}
This completes the proof of (1) of Theorem \ref{main thm 1}.
\end{proof}

To prove (2) of Theorem \ref{main thm 1}, we recall the following result by Bouarich.

\begin{thm}[\cite{MR1338286}] \label{thm Bouarich}
There exists an exact sequence
\[0 \to \HHH^2_b(\ppi) \to \HHH^2_b(G) \to \HHH^2_b(N)^G \to \HHH^3_b(\ppi).\]
\end{thm}

\begin{proof}[Proof of $(2)$ of Theorem $\ref{main thm 1}$]
  By Lemma \ref{lem 3.1}, we have the following commutative diagram
  \begin{eqnarray} \label{eqn 3.6}
  \xymatrix{
  0 \ar[r] & \EH^2_b(G) \ar[r] \ar[d] & \HHH^2_b(G) \ar[r] \ar[d]^{i^*} & \Im(c_G) \ar[r] \ar[d] & 0 \\
  0 \ar[r] & \EH^2_b(N)^G \ar[r] & \HHH^2_b(N)^G \ar[r] & \HHH^2(N)^G &
  \;\;\;\;\; ,}
  \end{eqnarray}
  where each row is exact. The exactness of the second row follows from Lemma \ref{lem 3.1} and the left exactness of the functor $(-)^G$. Let $K$ and $W$ denote the kernel and cokernel of the map $i^* \colon \HHH^2_b(G) \to \HHH^2_b(N)^G$. Note that the kernel of $\Im(c_G) \to \HHH^2(N)^G$ is $\Im(c_G) \cap \Ker(i^* \colon \HHH^2(G) \to \HHH^2(N)^G)$. Applying the snake lemma, we have the following exact sequence
  \begin{eqnarray} \label{eqn 3.7}
  K \to \Ker(i^*) \cap \Im(c_G) \to \EH^2_b(N)^G / i^* \EH^2_b(G) \to W.
  \end{eqnarray}

  By Theorem \ref{thm Bouarich}, $K$ is isomorphic to $\HHH^2_b(G/N)$, and there exists a monomorphism from $W$ to $\HHH^3_b(G)$.
  Hence we have an exact sequence
  \begin{eqnarray} \label{eqn 3.8}
  \HHH^2(G/N) \to \Ker(i^*) \cap \Im(c_G) \to \EH^2_b(N) / i^* \EH^2_b(G) \to \HHH_b^3(G).
  \end{eqnarray}
  Here the last map $\EH^2_b(N)^G / i^* \EH^2_b(G) \to \HHH^3_b(G)$ is the composite of the map $\EH^2_b(N)^G / i^* \EH^2_b(G) \to W$ and the monomorphism $W \to \HHH^3_b(G)$.
  This completes the proof of (2) of Theorem \ref{main thm 1}.
\end{proof}

\begin{proof}[Proof of $(3)$ of Theorem $\ref{main thm 1}$]
  Recall that $\alpha \colon \EH^2_b(N)^G / i^* \EH^2_b(G) \to \HHH^1(G ; \HHH^1(N))$ in (1) of Theorem \ref{main thm 1} is induced by the last map $\varphi$ of the exact sequence
  \[0 \to \HHH^1(N)^G \to \QQQ(N)^G \xrightarrow{\delta} \EH^2_b(N)^G \xrightarrow{\varphi} \HHH^1(G; \HHH^1(N)).\]
  We first describe $\varphi$. Let $c \in \EH^2_b(N)^G$. Since $\delta \colon \QQQ(N) \to \EH^2_b(N)$ is surjective, there exists a homogeneous quasimorphism $\qm$ on $N$ such that $c = [\delta \qm]$. Since $c$ is $G$-invariant, we have $^g c = c$ for every $g \in G$. Namely, for each $g \in G$, there exists a bounded $1$-cochain $b_G \in C^1_b(N)$ such that
  \begin{eqnarray} \label{eqn 3.12}
  ^g (\delta \qm) = \delta \qm + \delta b_g.
  \end{eqnarray}
  Note that this $b_g$ is unique. Indeed, if $\delta b_g = \delta b_g'$, then $b_g - b'_g$ is a homomorphism $G \to \RR$ which is bounded, and is $0$.

  Define the cochain $\varphi_\qm \in C^1(G ; \HHH^1(N))$ by
  \[\varphi_\qm(g) = \qm - {}^g \qm - b_g.\]
  It follows from \eqref{eqn 3.12} that $\varphi_\qm \in \HHH^1(N)$. Now we show that this correspondence induces a map from $\EH^2_b(N) / i^* \EH^2_b(G)$ to $\HHH^1(G ; \HHH^1(N))$.
  Suppose that $c = [\delta \qm] = [\delta \qm']$ for $\qm, \qm' \in \QQQ(N)$. Then $h = \qm - \qm' \in \HHH^1(N)$. Therefore we have $\delta \qm = \delta \qm'$, and hence we have
  \[^g(\delta \qm') = \delta \qm' + \delta b_g.\]
  Hence we have
  \[(\varphi_{\qm'} - \varphi_\qm)(g) = (\qm' - {}^g\qm' + b_g) - (\qm - {}^g \qm + b_g) = {}^g h - h = \delta h (g).\]
  Therefore $\varphi_{\qm'}$ and $\varphi_\qm$ represent the same cohomology class of $\HHH^1(G ; \HHH^1(N))$. This correspondence is the precise description of $\alpha \colon \EH^2_b(N)^G \to \HHH^1(G ; \HHH^1(N))$.

  Next, we see the precise description of the composite of
  \[\Ker (i^*) \cap \Im (c_G) \xrightarrow{\beta} \EH^2_b(G) / i^* \EH^2_b(G) \xrightarrow{\alpha} \HHH^1(G ; \HHH^1(N)).\]
  Let $c \in \Ker (i^*) \cap \Im (c_G)$. Since $c \in \Im (c_G)$, there exists a bounded cocycle $f \colon G \times G \to \RR$ with $c = [f]$ in $\HHH^2(G)$. Since $i^* c = 0$, there exists $\qm' \in C^1(N)$ such that $f|_{N \times N} = \delta \qm'$ in $C^2(N)$. Since $f$ is bounded, $\qm'$ is a quasimorphism on $N$. Define $\qm$ to be the homogenization of $\qm'$. Then $b_N = \qm - \qm' \colon N \to \RR$ is a bounded 1-cochain on $N$. Next define the function $b \colon G \to \RR$ by
  \[b(x) = \begin{cases}
  b_N(x) & x \in N \\
  0 & {\rm otherwise.}
  \end{cases}\]
  Since $b \in C^1_b(G)$, $f + \delta b$ is a bounded cocycle which represents $c$ in $\HHH^2(G)$. Replacing $f + \delta b$  by  $f$, we can assume that $f|_{N \times N} = \delta \qm$. Then by the definition of the connecting homomorphism in snake lemma, we have $\beta(c) = [\delta \qm]$.

  Recall that there exists a unique bounded function $b_g \colon N \to \RR$ such that
  \[\varphi([\delta \qm])(g) = \qm - {}^g \qm + b_g.\]

  \vspace{2mm}
  \noindent
  {\bf Claim.} $b_g(n) = f(g, g^{-1} n g)$.

  \vspace{2mm}
  Define $a_g$ by $a_g(n) = f(g, g^{-1}ng)$. Let $n$ and $m$ be elements of $N$. Since $\delta f = 0$, we have
  \begin{eqnarray*}
  \delta a_g(n,m) & = & \delta a_g + \delta f (g, g^{-1} ng, g^{-1} m g) + \delta f(n,m,g) - \delta f(n,g, g^{-1} mg) \\
  & = & f(g^{-1} ng, g^{-1} m g) - f(n,m) \\
  & = & ({}^g \delta \qm - \delta \qm)(n,m).\\
  & = & \delta b_g.
  \end{eqnarray*}
  By the uniqueness of $b_g$, we have $a_g = b_g$.
  This completes the proof of Claim.
  Hence we have $\varphi_\qm(g) = \qm - {}^g \qm + a_g$, and thus we obtain a precise description of $\alpha \circ \beta$.

  Now we complete the proof of (3) of Theorem \ref{main thm 1}. For $c \in \Ker (i^*) \cap \Im (c_G)$, there exists a bounded $2$-cocycle $f$ of $G$ such that $f |_{N \times N} = \delta \qm'$ for some $\qm' \in \QQQ(N)$. Define $\qm \colon G \to \RR$ by
  \[\qm(x) = \begin{cases}
  \qm' (x) & x \in N \\
  0 & {\rm otherwise.}
  \end{cases}\]
  Then $f - \delta \qm$ is a (possibly unbounded) cocycle such that $(f - \delta \qm) |_{N \times N} = 0$. Hence Theorem \ref{thm DHW} implies
  \begin{eqnarray*}
  ((p^* \rho (c))(g))(u) & = & (\rho(c) (p(g))) (u) \\
  & = & (f - \delta \qm)(g , g^{-1} n g) - (f - \delta \qm)(u, g) \nonumber \\
  & = & f(g, g^{-1} n g) - f(n,g) + \qm(ng) - \qm (g) - \qm(g^{-1} n g) + \qm(g) - \qm (ug) + \qm(u) \\
  & = & \qm(n) - {}^g \qm(u) + b_g(n) \\
  & = & (\qm - {}^g \qm + b_g)(n) \\
  & = & \varphi_\qm(g)(n).
  \end{eqnarray*}
  Here the second equality follows from Theorem \ref{thm DHW} and the fourth equality follows from Claim. Hence we have
  \begin{eqnarray} \label{eqn B}
  ((p^* \rho (c))(g))(u) = \varphi_\qm(g) (n),
  \end{eqnarray}
  and $\alpha \circ \beta (c) = p^* \circ \rho (c)$ follows from the description of $\alpha \circ \beta$ and \eqref{eqn B}. This completes the proof.
\end{proof}

Finally, we show that Theorem \ref{main thm 1} implies some results in this paper.

\begin{proof}[Proof of Theorem $\ref{easy cor}$ by using Theorem $\ref{main thm 1}$]
It follows from (1) of Theorem \ref{main thm 1} that $\Ker (\alpha)$ and $\QQQ(\bG)^\hG / (\HHH^1(\bG)^\hG + i^* \QQQ(\hG))$ are isomorphic. Since 
 $\ppi$ is boundedly $3$-acyclic, (2) of  Theorem \ref{main thm 1} implies that $\beta$ is an isomorphism. Since the homomorphism $\HHH^1(\Gamma ; \HHH^1(\bG)) \to \HHH^1(\hG ; \HHH^1(\bG))$ is injective, (3) of Theorem \ref{main thm 1} implies
\begin{eqnarray*}
\Ker (\alpha) & \cong & \Ker \big( \rho \circ j \colon \Ker(i^*) \cap \Im (c_G) \to \HHH^1(\Gamma ; \HHH^1(\bG)) \big) \\
& = & \Ker (\rho) \cap \Im (c_G).
\end{eqnarray*}
By the seven term exact sequence (Theorem \ref{thm seven-term}), we have $\Ker (\rho) = \Im (p^*)$. This completes the proof.
\end{proof}

\begin{proof}[Proof of Theorem $\ref{main thm 2}$ by using Theorem $\ref{main thm 1}$]
We first show that the map $\HHH^1(\bG)^\hG / i^* \HHH^1(\hG) \to \QQQ(\bG)^\hG / i^* \QQQ(\hG)$ is injective if $\Gamma$ is  boundedly $3$-acyclic. Indeed, applying the snake lemma to diagram \eqref{eqn 3.6}, we have $\Ker (\EH^2_b(\hG) \to \EH^2_b(\bG)^\hG) = 0$ since $\HHH^2_b(\Gamma) = 0$. Next, applying the snake lemma to diagram \eqref{eqn 3.3}, we  see that the map $\HHH^1(\bG)^\hG / i^* \HHH^1(\hG) \to \QQQ(\bG)^\hG / i^* \QQQ(\hG)$ is injective.

Thus we have two exact sequences
\[0 \to \HHH^1(\bG)^\hG / i^* \HHH^1(\hG) \to \QQQ(\bG)^\hG / i^* \QQQ(\hG) \to \Ker (\alpha) \to 0\]
and
\[0 \to \HHH^1(\bG)^\hG / i^* \HHH^1(\hG) \to \HHH^2(\Gamma) \to \Ker (\rho) \to 0.\]
Here the second exact sequence is deduced from the seven term exact sequence (Theorem \ref{thm seven-term}). It suffices to see that $\Ker (\rho) \cong \Ker(\alpha)$ by (3) of Theorem \ref{main thm 1}.
\begin{itemize}
\item Since $G$ is hyperbolic, we have $\Im (c_\hG) = \HHH^2(\hG)$. Therefore $j$ is an isomorphism.

\item Since $\Gamma$ is  boundedly $3$-acyclic, it follows from (2) of Theorem \ref{main thm 1} that $\beta$ is an isomorphism.

\item The map $\HHH^1(\Gamma ; \HHH^1(\bG)^\hG) \to \HHH^1(\hG ; \HHH^1(N)^G)$ is injective.
\end{itemize}
From the above facts, we conclude that $\Ker(\rho) \cong \Ker(\alpha)$.
\end{proof}


\begin{proof}[Proof of Theorem $\ref{thm surface group}$ by using Theorem $\ref{main thm 1}$ and Corollary $\ref{cor:Appendix}$]
  Let $\hG$ be the surface group $\pi_1(\Sigma_\genus)$ and $\bG$ the commutator subgroup $[\pi_1(\Sigma_\genus),\pi_1(\Sigma_\genus)]$.
  Then the quotient $\hG/\bG$ is isomorphic to $\ZZ^{2\genus}$.
  By (3) of Theorem \ref{main thm 1} and Corollary \ref{cor:Appendix}, we have the following commutative diagram whose rows are exact:
  \[
  \xymatrix{
  0 \ar[r] & \QQQ(\bG)^{\hG} / (\HHH^1(\bG) + i^* \QQQ(\hG)) \ar[r] & \Ker(i^*) \cap \mathrm{Im}(c_{\hG}) \ar[r]^-{\alpha \circ \beta} \ar@{^{(}-_>}[d] & \HHH^1(\hG;\HHH^1(\bG))\\
  \HHH^1(\bG)^{\hG} \ar[r] & \HHH^2(\ZZ^{2\genus}) \ar[r] & \Ker(i^*) \ar[r]^-{\rho} & \HHH^1(\ZZ^{2\genus};\HHH^1(\bG)) \ar[u],
  }
  \]
  where $i^* \colon \HHH^2(\hG) \to \HHH^2(\bG)$ and the second row is a part of the seven-term exact sequence.
  Since $\dim(\HHH^1(\bG)^{\hG}) = \genus(2\genus - 1) - 1$ (Proposition \ref{prop:inv_hom_surface_group}) and $\dim(\HHH^1(\ZZ^{2\genus})) = \genus(2\genus - 1)$,
  we have $\dim \Ker (i^*) = 1$ and $\rho = 0$.
  Since the comparison map $c_{\hG} \colon \HHH_b^2(\hG) \to \HHH^2(\hG)$ is surjective, we have $\Ker(i^*) \cap \mathrm{Im}(c_{\hG}) = \Ker(i^*) \cong \RR$.
  Since $\rho = 0$, the map $\alpha \circ \beta$ is also the zero map, and this implies
  \[
    \QQQ(\bG)^{\hG} / (\HHH^1(\bG) + i^* \QQQ(\hG)) \cong \Ker(i^*) \cap \mathrm{Im}(c_{\hG}) \cong \RR.
  \]
\end{proof}

\bibliography{reference}
\bibliographystyle{amsalpha}

\end{document}